\documentclass[11pt]{article}
\usepackage[letterpaper,left=1in,top=1in,textwidth=6.6in,textheight=8.9in]{geometry}
\def\baselinestretch{1.3}

\usepackage{url}
\usepackage{amsmath}
\usepackage{amssymb}
\usepackage{amsfonts}
\usepackage{amsthm}
\usepackage{bm}
\usepackage[bottom]{footmisc}
\usepackage{etoolbox}
\patchcmd{\maketitle}{\rlap}{\mbox}{}{}
\usepackage{authblk}

\usepackage{endnotes}

\usepackage{natbib}
 \bibpunct[, ]{(}{)}{,}{a}{}{,}%

\usepackage{graphicx}
\usepackage{enumerate}
\usepackage{ifthen}
\usepackage{booktabs}      
\usepackage[normalem]{ulem}
\usepackage{longtable}     
\usepackage{array}         
\usepackage{rotating}
\usepackage{lscape}
\usepackage{makecell}
\usepackage{color}
\usepackage[dvipsnames]{xcolor}
\usepackage{caption}
\DeclareCaptionFont{bodysize}{\fontsize{11}{13.6}\selectfont}
\let\footnote=\endnote

\usepackage[vlined]{algorithm2e}
\RestyleAlgo{ruled}
\SetKwComment{tcp}{$\triangleright$\ }{}
\SetCommentSty{textnormal}
\SetAlgoSkip{}
\SetKwProg{Fn}{function}{}{end function}
\SetKwComment{Comment}{$\triangleright$\ }{}
\newcommand{\RN}[1]{\uppercase\expandafter{\romannumeral#1}}

\newlength{\figbase}
\makeatletter
\newcommand\tablenormalsize{\@setfontsize\normalsize{10}{12}}
\newcommand\tablesmall{\@setfontsize\small{9}{11}}
\newcommand\tablefootnotesize{\@setfontsize\footnotesize{8}{10}}
\newcommand\algosmall{\@setfontsize\small{10}{17.25}}
\makeatother
\newcommand{\floatsetup}{%
  \def\baselinestretch{1}%
  \setlength{\tabcolsep}{4pt}%
  \setlength{\arraycolsep}{1.8pt}}
\AtBeginEnvironment{figure}{\floatsetup}
\AtBeginEnvironment{table}{
  \floatsetup
  \let\normalsize\tablenormalsize
  \let\small\tablesmall
  \let\footnotesize\tablefootnotesize}
\AtBeginEnvironment{algorithm}{\floatsetup\let\small\algosmall}
\AtBeginEnvironment{longtable}{\floatsetup\fontsize{9}{11}\selectfont}

\newcommand{\floatnote}[1]{{\footnotesize\leftskip=0pt\rightskip=0pt\parfillskip=0pt plus 1fil #1\par}}
\newcommand{\FIGURE}[3]{%
  \centering #1\par
  \caption{#2}%
  \ifstrempty{#3}{}{\par\floatnote{#3}}}
\newcommand{\TABLE}[3]{%
  \centering
  \caption{#1}%
  #2\par
  \ifstrempty{#3}{}{\par\floatnote{#3}}}

\theoremstyle{plain}

\newtheorem{lemma}{Lemma}
\newtheorem{proposition}{Proposition}

\newtheorem{assumption}{Assumption}
\theoremstyle{definition}

\newtheorem{remark}{Remark}
\newcommand{\Halmos}{\qedhere}

\newcommand{\argmin}{\mathop{\rm arg\,min}}

\allowdisplaybreaks

\begin{document}

\title{Dynamic Control of a Make-to-Order Manufacturing System under Throughput Time Constraints: An Effective Computational Method in the High-Dimensional Case}

\author{Baris Ata\thanks{\texttt{baris.ata@chicagobooth.edu}}}
\author{Chun Li\thanks{\texttt{cli11@chicagobooth.edu}}}
\affil{Booth School of Business, University of Chicago}

\date{}

\maketitle

{\renewcommand{\thefootnote}{}%
 \footnotetext{Our code is available at \url{https://github.com/chunli0574/RBM-TPT-code}.}}

\begin{abstract}
We consider dynamic control of a make-to-order production system that serves price- and delay-sensitive customers. The production system is modeled as a multiclass queueing network where the customers' delay sensitivity is modeled via throughput-time constraints. To maximize her profit while meeting these throughput time constraints, the system manager uses three different types of control: i) dynamic pricing decisions; ii) dynamic order-of-service (scheduling) decisions; and iii) outsourcing decisions upon arrival of new orders. We approximate this problem by a multi-dimensional Brownian control problem in heavy traffic, and develop a neural network-based computational method tailored for high-dimensional problems. By interpreting the solution in the context of the original model, we propose an integrated dynamic pricing, outsourcing, and scheduling policy. The effectiveness of this policy is demonstrated through a simulation study of two test problems drawn from the literature.

\end{abstract}

\section{Introduction\label{sec: introduction}}%

This paper studies dynamic control of a multiclass queueing network with throughput-time constraints that is motivated by make-to-order manufacturing systems, serving price- and delay-sensitive customers. The system manager adjusts prices and allocates the service effort among different classes dynamically over time. She can also outsource the orders upon their arrival if the backlog of work is judged excessive. Her objective is to maximize her profit while satisfying the throughput-time constraints. Two important features of our formulation are the following: First, the customer orders of different types follow deterministic routes in the system. Second, there are no holding costs. Instead, congestion concerns are captured through the aforementioned throughput time constraints. Arguably, such constraints are more realistic and easier to estimate for the purpose of model calibration than the more commonly used formulation with holding costs that is analytically convenient \citep{Harrison_2024}.

Because the dynamic control problem is intractable in its original form, we resort to an approximation. We follow the
heavy-traffic approach pioneered by \cite{Harrison_1988_brownian} and derive the
approximating Brownian control problem. An important simplification that occurs in the heavy traffic limit is that one can replace the throughput time constraints by upper bounds on the sum of certain (scaled) queue lengths, as observed by \cite{Plambeck_Kumar_Harrison_2001}.

The approximating Brownian control problem is equivalent to a
lower-dimensional workload problem. In our setting, the equivalent workload problem is a
multi-dimensional control problem with two salient features: First, it
includes both drift and singular controls, inheriting them from the pricing,
outsourcing and scheduling controls in the original queueing system. Second, its state space is a bounded convex polyhedron. Although it is simpler than the original control problem, the workload problem
remains multi-dimensional and does not admit a closed-form solution. Thus, we
develop a neural network-based method to solve this control
problem numerically and obtain an effective control policy. The resulting control policy has two distinct features: First, the pushing controls (i.e., singular controls) are active only on the boundary of the state space and they partition it with respect to different active pushing controls. The pushing controls correspond to either idling the servers or outsourcing activities. Second, in the interior of the state space, the control is a state-dependent drift rate control that corresponds to the system manager's dynamic pricing policy.

We interpret the solution of the workload problem in the context of the original queueing network by proposing a policy for it. We illustrate our solution approach and test its effectiveness in two examples: The Rybko-Stolyar network \citep{Rybko_Stolyar_1992, Maglaras_2000_discrete} and the extended criss-cross network. The latter builds on the criss-cross network and provides us with a set of high-dimensional test problems \citep{Harrison_Wein_1989_scheduling, Martins_Shreve_Soner_1996_heavy, Budhiraja_Ross_2008_optimal, Budhiraja_Liu_Saha_2017_construction, Kumar_Muthuraman_2004}.

To solve the multi-dimensional dynamic control problem, we
develop a computational method that relies heavily on deep neural network
technology. Two important antecedents of our work are
\cite{Ata_Harrison_Si_2024_drift, Ata_Harrison_Si_2024_singular}.
\cite{Ata_Harrison_Si_2024_drift} builds on the seminal work of \cite{Han_Jentzen_E_2018} to develop neural network-based method for solving
drift control problems for reflected Brownian motion. \cite{Ata_Harrison_Si_2024_singular} extends it to
solve singular control problems, by approximating singular controls using
drift controls, and shows its effectiveness in queueing applications. Our computational method builds on \cite{Ata_Harrison_Si_2024_drift, Ata_Harrison_Si_2024_singular}, and extends them to the setting with bounded convex polyhedral state space, which is not covered by the positive orthant considered in
\cite{Ata_Harrison_Si_2024_drift, Ata_Harrison_Si_2024_singular}.

\section{Literature Review\label{sec: literature}}%

Our work relates to several streams of literature. The first one uses the heavy traffic approximations to study dynamic control of stochastic networks. Pioneered by \cite{Harrison_1988_brownian, Harrison_2000_brownian, Harrison_2003_broader}, the dynamic control problems are approximated by control problems involving Brownian motions to obtain more tractable solutions in this literature. An important group of papers in this literature study the criss-cross network in the heavy traffic asymptotic regime. \cite{Harrison_Wein_1989_scheduling} study the criss-cross network for a set of cost parameters under which the problem admits a pathwise solution. 
\cite{Martins_Shreve_Soner_1996_heavy} revisits the problem and identifies a total of five cases in terms of possible cost parameters of the problem. The authors label them as Case \RN{1} and Cases \RN{2}A--\RN{2}D. Case \RN{1} is studied by \cite{Chen_Yang_Yao_1994}. \cite{Martins_Shreve_Soner_1996_heavy} prove the asymptotic optimality of the solution proposed by \cite{Harrison_Wein_1989_scheduling} for Case \RN{2}A, and they provide conjectures for the other cases. Two of these cases (Cases \RN{2}B and \RN{2}C) were rigorously justified in later papers by \cite{Budhiraja_Ross_2008_optimal} and \cite{Budhiraja_Liu_Saha_2017_construction}. The criss-cross network constitutes the basis for our second test problem, whereas our first test problem is based on the Rybko-Stolyar network; see \cite{Rybko_Stolyar_1992}.

Among papers studying dynamic control of queueing networks in the heavy traffic regime, ours is closely related to those involving throughput time constraints. \cite{Plambeck_Kumar_Harrison_2001} incorporates such throughput-time constraints and replaces them with upper bounds on queue lengths; also see \cite{Maglaras_Van_Mieghem_2005}, \cite{Ata_2006_thin}, \cite{Celik_Maglaras_2008} and \cite{Rubino_Ata_2009}. Other related papers include those studying dynamic pricing and outsourcing controls.

Many researchers have studied models concerning pricing, outsourcing or admission control for queueing systems; see for example, \cite{Low_1974_optimal}, \cite{Stidham_1985_optimal, Stidham_2004_pricing}, \cite{Lewis_Ayhan_Foley_1999_bias}, \cite{Carr_Duenyas_2000_optimal}, \cite{Yoon_Lewis_2004_optimal}, \cite{Ziya_Ayhan_Foley_2006_optimal}, \cite{Maoui_Ayhan_Foley_2009_optimal}, \cite{Zhang_Ayhan_2012_optimal}. The closest work to ours is \cite{Ata_Barjesteh_2023} that studies dynamic pricing and outsourcing in a multiclass, single-server make-to-stock manufacturing system. Similar to our paper, their pricing (resp. outsourcing) control  becomes drift control within state space (resp. singular control on reflecting barriers) in limiting formulation. Our proposed pricing policy is similar to the one in \cite{Ata_Barjesteh_2023}, whose limiting formulation is one-dimensional, whereas our multiclass queueing network setting leads to a multi-dimensional formulation.


Another related stream of literature develops computational methods for dynamic control of queueing networks. The most closely related papers to ours are \cite{Ata_Harrison_Si_2024_drift, Ata_Harrison_Si_2024_singular}. The authors develop neural network-based computational methods for solving high-dimensional drift-rate and singular control problems. To be specific, \cite{Ata_Harrison_Si_2024_drift} develops a neural network-based method for solving
drift control problems for reflected Brownian motion, and  \cite{Ata_Harrison_Si_2024_singular} extends it to
solve singular control problems, by approximating singular controls using
drift controls, and shows its effectiveness in queueing applications. We build on \cite{Ata_Harrison_Si_2024_drift, Ata_Harrison_Si_2024_singular} by extending their work to allow a bounded convex polyhedron as the state space.

Other closely related papers also develop neural network-based computational method to solve control problems. Motivated by call center applications, \cite{Ata_Kasikaralar_2023_dynamic} focuses on limiting control problems in Halfin-Whitt heavy traffic regime, and develops a computational method for finite-horizon formulation with non-stationary data; also see \cite{Ata_Kasikaralar_2026_parallel}. \cite{Ata_Zhou_2024} solves a sequential vehicle routing problem motivated by an eviction enforcement application following a similar approach. Similarly, \cite{Ata_vanEekelen_Zhong_2024} proposes a computational method for solving the stochastic joint replenishment problem in high dimensions by exploiting a connection among multi-dimensional impulse control problems, backward stochastic differential equations and the stochastic target problem. Those researchers build on the literature started by the seminal paper \cite{Han_Jentzen_E_2018}, that focuses on solving (semilinear) partial differential equations using neural networks, relying on a connection between (semilinear) partial differential equations and backward stochastic differential equations; see \cite{E_Han_Jentzen_2021_algorithms}, \cite{Beck_Hutzenthaler_Jentzen_Kuckuck_2023}, \cite{Chessari_Kawai_2023} and \cite{Han_Jentzen_E_2025} for surveys of that literature.

Prior to neural network-based methods for solving Hamilton-Jacobi-Bellman (HJB) equations, solution to multi-dimensional stochastic control problems was usually restricted to low-dimensional case. Two numerical methods for solving singular control problems are proposed in \cite{Kushner_Martins_1991} and \cite{Kumar_Muthuraman_2004}. \cite{Kushner_Martins_1991} approximates the singular control problems by control problems of discrete-time Markov chain that is readily solvable, which is an extension of Markov chain approximation method in \cite{Kushner_1990_numerical} for stochastic control problems. \cite{Kumar_Muthuraman_2004} combines a finite element method with policy iteration to solve their HJB equation that is based on \cite{Taksar_1997_infinite}.


Drift and singular control problems arise in various manufacturing applications; see for example \cite{Ata_2006_thin}, \cite{Celik_Maglaras_2008}, \cite{Ata_Olsen_2009}, \cite{Rubino_Ata_2009}, \cite{OrmeciMatoglu_VandeVate_2011_drift} and \cite{OrmeciMatoglu_VandeVate_Wang_2015_solving} for related models of make-to-order manufacturing system, and \cite{Wein_1992_dynamic}, \cite{Veatch_Wein_1996_scheduling}, \cite{Caldentey_Wein_2006}, \cite{Ata_Barjesteh_2023} for related models of make-to-stock systems. Also see \cite{Dai_Yao_2013_brownian} for a summary of other applications. Other applications of drift and singular control problems include power control in wireless communication \citep{Ata_Harrison_Shepp_2005}, ride hailing \citep{Ata_Barjesteh_Kumar_2020, Alwan_Ata_Zhou_2024_queueing}, volunteer capacity management for nonprofits \citep{Ata_Lee_Sonmez_2019_dynamic, Ata_Tongarlak_Lee_Field_2024_dynamic}, flexible staffing in call center \citep{Sun_Liu_2025_expanding}, and drone delivery systems \citep{Liu_Sun_2022_energy}.

\section{Model\label{sec: model}}

We consider a make-to-order manufacturing system selling $K$ different products to price-sensitive customers. We model this system as a multiclass queueing network, serving $K$ customer types. In our model, type $k$ customer orders follow a pre-determined route through a sequence of buffers, denoted by $(k,1),\ldots,(k,n_{k})$, denoting their manufacturing process ($k = 1,\dots, K$). In other words, they go through $n_k$ production stages. During stage $l$ ($l = 1,\dots,n_k$), they are held in buffer $(k,l)$ that is dedicated to holding them. Thus, a total of $m = \sum_{k=1}^{K} n_k$ buffers hold jobs of various types that are in different stages of production. In what follows, we use the terms buffer and class interchangeably. To repeat, the class of a job corresponds to the combination of the product type and the stage of production it is in currently.

We assume jobs of various classes are served by a total of $p$ servers. We let $s(k,l)$ denote the server that serves class $(k,l)$ jobs, whose processing time has a general distribution with mean $m_{k,l}$ and squared coefficient-of-variation $v_{k,l}^2$ for $k = 1,\dots,K$ and $l = 1,\dots,n_k$; we let $\mu_{k,l} = 1/m_{k,l}$ denote the corresponding service rate. Also, we let $S_{k,l}(t)$ denote the total number of class $(k,l)$ jobs served if server $s(k,l)$ works on class $(k,l)$ jobs continuously during $[0,t]$. There is a many-to-one assignment of job classes to servers. To be specific, we let $\mathcal{C}_i$ denote the set of classes or buffers served by server $i$ ($i = 1,\dots,p$). That is, $\mathcal{C}_i = \{(k,l): s(k,l)= i\}$ for $i = 1,\dots, p$.

It is worth mentioning that the model described so far is a special class of multiclass queueing networks as defined in \cite{Harrison_1988_brownian}. The special structure we impose is that the probability routing matrix of \cite{Harrison_1988_brownian}, denoted by $P$ there, consists solely of zeros and ones. That is, we assume deterministic routing. However, a distinguishing feature of our model is the throughput time constraints: Associated with each product type is a throughput time or a delay constraint, which allows us to model congestion concerns in the system. More will be said about the model feature later in this section.

The system manager dynamically chooses price vectors $p(t) = (p_k(t)) \in \mathcal{P}$ for $t \geq 0$, where $p_k(t)$ is the price of type $k$ products at time $t$ $(k = 1,\dots,K$) and $\mathcal{P} \subset \mathbb{R}_+^K$ is the set of feasible price vectors. The price sensitivity of customers is modeled by a nonnegative demand function $\Lambda$ that maps the price vector $p$ to an instantaneous demand rate vector $\lambda$, i.e., $\lambda(t) = \Lambda(p(t))$, where $\Lambda:\mathcal{P} \to \mathcal{L} \subset \mathbb{R}_+^K$ and $\mathcal{L} =\{\Lambda(p):p\in \mathcal{P}\}$ is the set of feasible demand-rate vectors.

Given the system manager's dynamic pricing policy $\{p(t): t\geq 0\}$ and the resulting demand-rate process $\lambda = \{\lambda(t): t\geq 0\}$, we model the cumulative demand process as a nonhomogenous Poisson process. More specifically, demand for product $k$ arrives according to a nonhomogenous Poisson process with instantaneous rate $\lambda_k(t)$ for $t \geq 0$ and $k = 1,\dots,K$. That is, the cumulative demand for product $k$ up to time $t$ is \[
N_{k}\left(  \int_{0}^{t}\lambda_{k}(s)ds\right),\quad t \geq 0  ,
\]
where $N_{k}$ is a unit rate Poisson process ($k = 1,\dots,K$).

Following \cite{Celik_Maglaras_2008} and \cite{Ata_Barjesteh_2023}, we assume the following regularity conditions on the demand function $\Lambda$. 
\begin{assumption}
\label{assumption: demand_function}The following hold:
\begin{itemize}
    \item[i)] $\Lambda$ is bounded, continuously differentiable, and bijective; its inverse is denoted by $\Lambda^{-1}:\mathcal{L} \to \mathcal{P}$;
    \item[ii)] $\partial \Lambda_k(p)/ \partial p_k < 0$ for $p \in \mathcal{P}$, $k = 1,\dots, K$;
    \item[iii)] $\mathcal{L}$ is convex.
\end{itemize}
\end{assumption}

For mathematical convenience, we view the instantaneous demand rate process $\lambda = \{\lambda(t): t\geq 0\}$ as the system manager's pricing control, from which the price process $p = \{p(t):t\geq 0\}$ can be computed via $p(t) = \Lambda^{-1} (\lambda(t))$ for $t \geq 0$. As a preliminary to expressing the system manager's objective function, we denote the variable cost of manufacturing for product $k$ by $\delta_k$ ($k = 1,\dots,K$) and let $\delta = (\delta_k)$ denote the associated cost rate vector. Then we define the profit rate function $\Pi$ as a function of demand rate vector as follows: 
\begin{equation}
\Pi(x)=x \cdot (\Lambda^{-1}(x)-\delta),\quad x\in\mathcal{L}.\label{revenue_def}
\end{equation}
We also make the following regularity assumption on the profit rate function, which implies $\Pi(\cdot)$ is strictly concave.
\begin{assumption}
\label{assumption: profit_function} The profit rate function $\Pi$ is twice continuously differentiable and its Hessian is negative definite.
\end{assumption}

To facilitate the analysis that follows, we consider the following static planning problem that ignores the randomness in the system: Find the demand rate vector $\lambda$ so as to
\begin{equation}
\operatorname*{maximize}\quad\Pi(\lambda)\text{\qquad subject to\qquad}%
\lambda\in\mathcal{L}.\label{static_planning}
\end{equation}
\begin{assumption}
\label{assumption: static_planning} The static planning problem (\ref{static_planning}) has a unique optimal solution $\lambda^*$ that is in the interior of the set $\mathcal{L}$ of the feasible demand rate vectors.
\end{assumption}

Recall that the system manager strives to impose an upper bound, say $d_k$, on the throughput time, or total delay, experienced by orders of product type $k$ ($k = 1,\dots, K$). Of course, this formulation is not really meaningful in a conventional queueing model, because the throughput times are random variables. However, the formulation becomes meaningful in an asymptotic sense. To be specific, one can show how the law of large numbers implies an equivalence between (scaled) bounds on the number of jobs in the system and (scaled) bounds on throughput times as one approaches the heavy traffic limit; see for example \cite{Plambeck_Kumar_Harrison_2001}. Namely, throughput time constraints can be met with probability one asymptotically because the delay bounds correspond to upper bounds on buffer contents in the heavy traffic limit. Thus, we replace the throughput time bounds for type $k$ products by a bound $q_k$ on the total number of type $k$ orders waiting or being worked on in the system, where we set $q_k = \lambda_k^*d_k$ for $k = 1,\dots,K$, and $\lambda_k^*$ is the nominal arrival rate of type $k$ products; see Assumption \ref{assumption: static_planning}.

To meet the throughput time constraints, in addition to dynamic pricing and scheduling decisions, the system manager can also outsource product orders upon arrival. We let $O_k(t)$ denote the number of type $k$ orders outsourced until time $t$ ($k = 1,\dots,K$). She incurs a cost of $\nu_k$ (in excess of the marginal cost $\delta_k$ of production) per order of product type $k$ that is outsourced. We denote excess outsourcing cost vector by $\nu = (\nu_k)$. We let $O(t) = (O_k (t))$ so that $O = \{O(t): t\geq 0\}$ denotes the system manager's outsourcing policy.

The system manager's dynamic scheduling decisions are modeled through an $m$-dimensional cumulative allocation process, denoted by $T(t) = (T_{k,l}(t))$, where $T_{k,l}(t)$ denotes the cumulative amount of time server $s(k,l)$ devotes to serving buffer $(k,l)$ during $[0,t]$. The system manager's control is then denoted by $(O,T,\lambda)$, where the process $O$ describes the outsourcing decisions, the allocation process $T$ describes the scheduling decisions and the demand rate process $\lambda$ captures the dynamic pricing decisions.

We let $Q_{k,l}(t)$ denote the number of jobs in buffer $(k,l)$ at time $t$ ($k = 1,\dots,K$ and $l = 1,\dots,n_k$). Then $Q(t) = (Q_{k,l}(t))$ denotes the $m$-dimensional queue length vector at time $t$ and the queue length process is denoted by $Q = \{Q(t):t \geq 0\}$. Its dynamics under a control $(O,T,\lambda)$ is given as follows: For $t\geq 0$ and $k = 1,\dots,K$,
\begin{align}
Q_{k,1}(t)  &  =Q_{k,1}(0) + N_{k}\left(  \int_{0}^{t}\lambda_{k}(s)ds\right)
-S_{k,1}(T_{k,1}(t))-O_{k}(t),\label{Q_k_1}\\
Q_{k,l}(t)  &  =Q_{k,l}(0) + S_{k,l-1}(T_{k,l-1}(t))-S_{k,l}(T_{k,l}(t)),\quad
l=2,\ldots,n_{k}. \label{Q_k_l}%
\end{align}
A control $(O,T,\lambda)$ is admissible if it is nonanticipating and the following hold:
\begin{align}
&O(\cdot),T(\cdot)\text{ are nondecreasing with }O(0)=T(0) = 0,\ \label{O_feasible}\\
&\sum_{(k,l)\in \mathcal{C}_i}\left(T_{k,l}(t) - T_{k,l}(s)\right) \leq t-s, \ 0\leq s < t < \infty, \ i = 1,\dots,p,\label{T_absolute_continuous}\\
&\lambda(t)\in\mathcal{L}\text{ for }t\geq0, \label{lambda_feasible} \\
&Q(t) \geq 0,\ \sum_{l=1}^{n_k} Q_{k,l}(t) \leq q_k, \ k = 1,\dots,K,\ t\geq 0.\label{Q_nonnegative_upper_bound}
\end{align}

Given an admissible control $(O,T,\lambda)$, we define the cumulative idleness process for each server as follows:
\begin{equation}
I_{i}(t)=t-\sum_{(k,l)\in \mathcal{C}_{i}}T_{k,l}(t),\quad t\geq0, \quad i = 1,\dots,p, \label{I_i}%
\end{equation}
and $I(t) = (I_i(t))$ is the $p$-dimensional cumulative idleness vector at time $t$; and the idleness process is denoted by $I = \{I(t): t\geq 0\}$. Note that it follows from (\ref{O_feasible})--(\ref{T_absolute_continuous}) and (\ref{I_i}) that $I(\cdot)$ is continuous and nondecreasing with $I(0) = 0$. 

Recall that outsourcing a type $k$ product order incurs cost $\nu_k$ in excess of its marginal production cost $\delta_k$; so the total outsourcing cost is $\delta_k + \nu_k$ ($k = 1,\dots,K$). Under an admissible policy $(O,T,\lambda)$, we model the system manager's cumulative profit, denoted by $R(t)$, as follows:
\begin{equation}
R(t)=\int_{0}^{t}\Pi(\lambda(s))ds-\nu\cdot O(t). \label{R_t}%
\end{equation}
The first term on the right hand side of (\ref{R_t}) characterizes the profit
obtained by selling different products, and the second term represents the outsourcing cost. The system manager aims to
choose the optimal policy $(O,T,\lambda)$ to%
\begin{equation}
\operatorname*{maximize}\quad\mathbb{E}\left[  \int_{0}^{\infty}%
e^{-rt}dR(t)\right]  \quad\text{subject to}\quad\text{(\ref{Q_k_1}%
)--(\ref{R_t}),} \label{objective}%
\end{equation}
where $r$ is the interest rate for discounting.

This formulation is not tractable. Therefore, we derive an approximation, the Brownian control problem, in the next section.

\section{Approximating Brownian Control Problem\label{sec: BCP}}

To derive the approximating Brownian control problem, we consider a sequence of systems indexed by the system parameter $n$ and study it as $n$ gets large. We attach superscript $n$ to various quantities of interest in the $n$th system. Our approximation relies on the following heavy traffic assumption.

\begin{assumption}
[Heavy-traffic Assumption]\label{assumption: heavy_traffic}For $n\geq 1$, the following hold:
\begin{align}
\Lambda^{n}(x)  =n\Lambda(x)\text{ for }x\in\mathcal{P},%
\quad \mu^{n}    =n\mu+\sqrt{n}\eta.\label{assumption_Lambda_n_mu_n}
\end{align}
Moreover, the unique optimal solution $\lambda^*$ in Assumption \ref{assumption: static_planning} satisfies
\begin{align*}
    \sum_{(k,l)\in \mathcal{C}_i} \frac{\lambda_k^*}{\mu_{k,l}} = 1,\ i = 1,\dots,p.
\end{align*}
\end{assumption}

The heavy traffic assumption implies that the traffic intensity for each server is near one for large $n$; and it tends to one as $n\to \infty$. Also, it follows from Equations (\ref{revenue_def}) and (\ref{assumption_Lambda_n_mu_n}) that
\begin{equation}
\Pi^{n}(nx)=n\Pi(x). \label{property_revenue}%
\end{equation}
If there were no randomness, i.e., in the fluid model, it would be optimal to set the demand rate as $n\lambda^*$ at all times in the $n$th system. However, due to randomness, the system manager makes adjustments that are of order $\sqrt{n}$, as observed in \cite{Ata_Barjesteh_2023}. Thus, we assume the following form for the instantaneous demand rate vector $\lambda^n(\cdot)$ for the $n$th system:
\begin{equation}
  \lambda^{n}(t)=n\lambda^{\ast}+\sqrt{n}\zeta(t),\  t\geq 0, \label{lambda_assumption} 
\end{equation}
where $\zeta(\cdot)$ process corresponds to the second-order corrections. In essence, it can be viewed as the system manager's pricing control in the diffusion scale.

In the $n$th system, the demand and system capacity are of order $n$. Thus, we expect the queue length and the outsourcing process to be of order $\sqrt{n}$. As such, we scale them accordingly and introduce the following diffusion-scaled processes:
\begin{equation}
Z^{n}(t)=\frac{Q^{n}(t)}{\sqrt{n}}\ \  \text{and} \ \  U^{n}(t)=\frac{O^{n}(t)}{\sqrt{n}},\ t\geq 0.
\label{Z_n_U_n}%
\end{equation}
We let $x_{k,l}^* = \lambda_k^*/\mu_{k,l}$ denote the long-run average fraction of time server $s(k,l)$ should allocate to serving buffer $(k,l)$. Also let $x^* = (x_{k,l}^*)$ denote the corresponding vector of allocation rates. As argued in \cite{Harrison_1988_brownian}, the policies that are worthy of consideration must satisfy $T^n(t) \approx x^* t$. To be more specific, we assume
\begin{equation}
T^n(t) = x^* t - \frac{1}{\sqrt{n}} Y^{n}(t), \ t\geq 0.
\end{equation}
Or equivalently, we write 
\begin{equation}
Y^{n}(t) = \sqrt{n} (x^* t - T^n(t)) , \ t\geq 0,\label{scaled_scheduling_process}%
\end{equation}
where $Y^{n}(\cdot)$ is the centered and scaled allocation process. Similarly, we define the scaled idleness process by 
\begin{equation}
L^{n}(t)=\sqrt{n}I^n(t), \ t\geq 0. \label{scaled_idle_process}%
\end{equation}
For large $n$, as mentioned earlier, we expect queue lengths to be of order $\sqrt{n}$ and delays to be of order $1/\sqrt{n}$. Thus, we scale the upper bounds $d_k^n$ and $q_k^n$ on throughput times and queue lengths of type $k$ customers, respectively, as follows:
\begin{equation}
d_k^n = \frac{\ell_k}{\sqrt{n}} \ \ \text{and}\ \ q_k^n = \sqrt{n} \beta_k,\label{scaled_bound}
\end{equation}
where $\beta_k = \lambda_k^* \ell_k$ for $k = 1,\dots,K$. Also, we assume that the outsourcing cost $\nu^{n}$ varies with $n$ as follows: $\nu^{n}=\alpha/\sqrt{n}$, where $\alpha = (\alpha_{k})$ denotes (scaled) outsourcing cost vector.

Next, we center the cumulative profit process $R^n(t)$ and focus on its deviation from the idealized upper bound $n\Pi(\lambda^*)t$ in the Brownian control problem. That is, we define the cumulative cost process
\begin{equation}
C^{n}(t)=n\Pi(\lambda^{\ast})t-R^{n}(t),\quad t \geq 0. \label{cumulative_cost_n}%
\end{equation}
Lastly, we assume the interest rate $r^n$ for discounting in the $n$th system satisfies $r^n = \tilde{r}$. Appendix \ref{appendix: derivation_BCP} derives the approximating Brownian control problem by passing to the limit formally as $n\to \infty$ in the heavy traffic asymptotic regime. In the heavy-traffic limit, the
processes $C^{n}$, $Z^{n}$, $Y^{n}$, $L^{n}$, and $U^{n}$
are replaced with their limits $C$, $Z$, $Y$, $L%
$, and $U$, which jointly satisfy the following:
\begin{align}
C(t)  &  =\int_{0}^{t}\zeta(s)^{\prime}H\zeta(s)ds + \alpha
\cdot U(t),\label{BCP_xi}\\
Z_{k,1}(t)  &  =Z_{k,1}(0) + X_{k,1}(t)+\int_{0}^{t}\zeta_{k}(s)ds+\mu_{k,1}%
Y_{k,1}(t)-U_{k}(t),\ \ k = 1,\dots,K,\label{BCP_Z_k_1}\\
Z_{k,l}(t)  &  =Z_{k,l}(0) + X_{k,l}(t)-\mu_{k,l-1}Y_{k,l-1}(t)+\mu
_{k,l}Y_{k,l}(t),\ \ l=2,\ldots,n_{k},\ \  k = 1,\dots,K,\label{BCP_Z_k_l}\\
L_{i}(t)  &  =\sum_{(k,l)\in \mathcal{C}_{i}}Y_{k,l}(t),
\label{BCP_I_i}\ \  i = 1,\dots,p.%
\end{align}
Here, the matrix $H$ is determined by the Hessian of the profit rate function $\Pi$ as follows:
\begin{align}
    H=-\frac{1}{2}\nabla^{2}\Pi(\lambda^{\ast}).\label{Hessian}
\end{align}
In addition, $X_k = (X_{k,l})$ is an $n_k$-dimensional Brownian motion with drift vector $\varphi_k$ and covariance matrix $\Sigma_k$. Equations (\ref{BCP_BM_drift_1})--(\ref{BCP_BM_variance}) in Appendix \ref{appendix: notation_derivation_BCP} express $\varphi_k$ and $\Sigma_k$ in terms of the problem primitives. The processes $X_1,\dots,X_K$ are mutually independent.

The cumulative cost process $C(t)$ of the Brownian control problem involves two terms: The integral term on the right-hand side of Equation (\ref{BCP_xi}) represents the (scaled) lost profit whereas the term $\alpha\cdot U(t)$ is the (scaled) cost of outsourcing. Equations (\ref{BCP_Z_k_1})--(\ref{BCP_Z_k_l}) describe how the (scaled) queue length processes evolve, whereas Equation (\ref{BCP_I_i}) expresses the cumulative idleness process in terms of the centered and scaled allocation control $Y(t)$. Furthermore, it follows from Equation (\ref{Q_nonnegative_upper_bound}) that
\begin{equation}
Z(t) \geq 0,\ \sum_{l=1}^{n_{k}}Z_{k,l}(t)\leq\beta_{k},\ k = 1,\dots,K,\ t\geq 0.\label{BCP_queue_length_bound}%
\end{equation}
The idleness process $L$ and the outsourcing
process $U$ satisfy the following:%
\begin{align}
&L(\cdot),U(\cdot) \ \text{are continuous and non-decreasing with}\ L(0)=U(0)=0.
\label{BCP_I_U_nondecrease}
\end{align}

Then the approximating Brownian control problem can be stated as follows: Choose a non-anticipating
control $(U,Y,\zeta)$ to
\begin{equation}
\operatorname*{minimize}\quad\mathbb{E}\left[  \int_{0}^{\infty}e^{-\tilde{r} t}%
dC(t)\right]  \quad\text{subject to}\quad\text{(\ref{BCP_xi}%
)--(\ref{BCP_I_U_nondecrease}).} \label{BCP}%
\end{equation}

In the next section, we derive an equivalent yet lower-dimensional formulation, i.e., the workload problem. In the analysis that follows we solve the workload problem and interpret that solution in the context of the original queueing network model.

\section{Workload Problem\label{sec: EWF}}

To state the workload problem, we first define the $p$-dimensional workload process $W = \{W(t): t\geq 0\}$ as follows:
\begin{equation}
W_{i}(t)=\sum_{(k,l)\in \mathcal{C}_{i}}m_{k,l}\sum_{j=1}^{l} Z_{k,j}(t), \ t \geq 0, \ i = 1,\dots,p.
\label{workload_process}%
\end{equation}
Here, $W_i(t)$ represents the expected total work for server $i$ that is embodied in all jobs that are currently in the system at time $t \geq 0$. Note that the inner summation on the right-hand side of (\ref{workload_process}) corresponds to the total number of type $k$ jobs that are either with server $i$ or with the upstream servers currently and will visit server $i$ later.

Combining Equation (\ref{workload_process}) with Equations (\ref{BCP_Z_k_1})--(\ref{BCP_I_i}), it is straightforward to derive the evolution of the workload process given control $(U,Y,\zeta)$:
\begin{equation}
W_{i}(t)=W_{i}(0) + \xi_{i}(t)+\int_{0}^{t}\theta_{i}(s)ds+L%
_{i}(t)-\sum_{(k,l)\in \mathcal{C}_{i}}m_{k,l}U_{k}(t), \ t\geq 0,\ i = 1,\dots,p.
\label{workload_process_diffusion}%
\end{equation}
The process $\xi_i = \{\xi_i(t):t\geq 0\}$ is a Brownian motion that is defined as follows:
\begin{equation}
\xi_{i}(t)=\sum_{(k,l)\in \mathcal{C}_{i}}m_{k,l}\sum_{j=1}^l X_{k,j}(t),\ t \geq 0, \ i = 1,\dots,p. \label{workload_BM}%
\end{equation}
Then $\xi = (\xi_i)$ is a $p$-dimensional Brownian motion with drift vector $\gamma$ and covariance matrix $\Gamma$; see Equation (\ref{gamma_Gamma_def}) in Appendix \ref{appendix: notation_derivation_BCP} for formulae of $\gamma$ and $\Gamma$.
The pricing control $\zeta$ translates to a $p$-dimensional drift-rate control process $\theta = \{\theta(t):t\geq 0\}$ defined as follows:
\begin{equation}
\theta_{i}(t)=\sum_{(k,l)\in \mathcal{C}_{i}}m_{k,l}\zeta_{k}(t), \ t \geq 0,\ i = 1,\dots,p.
\label{workload_drift}%
\end{equation}
We define the set of admissible drift rates for the workload problem, denoted by $\Theta$, as follows:
\begin{align}
    \Theta = \{\theta \in \mathbb{R}^p:\ \exists \zeta\in \mathbb{R}^K\ \text{such that}\ \theta_i=\sum_{(k,l)\in \mathcal{C}_{i}}m_{k,l}\zeta_{k}\ \text{for}\ i = 1,\dots,p\}. \label{workload_Theta}
\end{align}
Then we require that $\theta(t) \in \Theta$ for $t\geq 0$. Using Equation (\ref{BCP_queue_length_bound}), one can define the state space $\mathcal{Z}$ of the (scaled) queue length process as follows:
\begin{equation}
\mathcal{Z}=\left\{  z\in \mathbb{R}_+^m: \sum_{l=1}^{n_{k}}z_{k,l}\leq\beta_{k},\ k = 1,\dots,K\right\}  ,
\label{Z_state_space}%
\end{equation}
which facilitates the definition of the state space $\mathcal{W}$ of the workload process $W$ as follows:
\begin{equation}
\mathcal{W}=\left\{  w\in \mathbb{R}_+^p:\exists z\in\mathcal{Z}\text{ such that }w_{i}%
=\sum_{(k,l)\in \mathcal{C}_{i}}m_{k,l}\sum_{j=1}^l z_{k,j},\ i=1,\ldots
,p\right\} . \label{W_state_space}%
\end{equation}
We write the state space constraint as 
\begin{align}
    W(t)\in \mathcal{W}\ \text{for} \ t\geq 0.\label{W_state_space_constraint}
\end{align}
Lastly, we define the cost function $c(x)$ associated with the chosen drift rate $\theta(t) = x$ as follows:
\begin{align}
c(x)  &  =\min_{\zeta} \left\{  \zeta^{\prime}H\zeta:\ \sum_{(k,l)\in \mathcal{C}_{i}}m_{k,l}\zeta_{k} = x_i,\ i = 1,\dots,p\right\}  ,\ x \in \Theta.\label{c_x}
\end{align} 
The next lemma provides closed-form expressions for the cost function $c(\cdot)$ and the minimizer $\zeta^*$ in its definition (\ref{c_x}). To facilitate its statement, we define the $p\times K$ dimensional workload matrix $\tilde{M} = (\tilde{M}_{ik})$ as follows:
\begin{align}
    \tilde{M}_{ik} = \sum_{l:(k,l)\in\mathcal{C}_i} m_{k,l} \ \text{for} \ i = 1,\dots, p\ \text{and}\  k = 1,\dots,K,\label{c_x_direction}
\end{align}
where $\tilde{M}_{ik}$ denotes the work content for server $i$ embodied in a type $k$ job arriving to the system. For a proof of Lemma \ref{lemma:c_x}, see Appendix \ref{appendix: proof_lemma_c_x}.

\begin{lemma}\label{lemma:c_x}
For $x\in \Theta$, we have that $c(x)=x^{\prime}(\tilde{M}H^{-1}\tilde{M}^{\prime})^{\dagger}x$ and $\zeta^{\ast
}(x)=H^{-1}\tilde{M}^{\prime}(\tilde{M}H^{-1}\tilde{M}^{\prime})^{\dagger}x$, where the superscript $\dagger$ denotes pseudo-inverse.
\end{lemma}

\begin{remark}\label{remark: c_x}
   The pseudo-inverse $(\tilde{M}H^{-1}\tilde{M}^{\prime})^{\dagger}$ is needed because the positive semidefinite matrix $\tilde{M}H^{-1}\tilde{M}'$ need not be of full rank in general, e.g., when $p > K$. When it is of full rank, one can replace the pseudo-inverse with the inverse.
\end{remark}

The workload problem is then stated as follows: Given initial state $W(0) = w  \in \mathcal{W}$, choose an adapted control $(L,U,\theta)$ so as to
\begin{align}
&  \operatorname*{minimize}\quad J(w;L,U,\theta) = \mathbb{E}\left[  \int_{0}^{\infty
}e^{-\tilde{r} t}\left(c(\theta(t))dt+ \alpha\cdot dU(t)\right)\right]  \quad\label{EWF}\\
&  \text{subject to\quad(\ref{BCP_I_U_nondecrease}),
(\ref{workload_process_diffusion}) and (\ref{W_state_space_constraint}).}\nonumber
\end{align}
We then define its value function, denoted by $V$, as follows:
\begin{equation}
    V(w) = \inf_{L,U,\theta} J(w;L,U,\theta),\ w\in \mathcal{W}.\label{value_function}
\end{equation}
Proposition \ref{proposition: equivalent_BCP_EWF} establishes an equivalence between the Brownian control problem (\ref{BCP}) and the workload problem (\ref{EWF}); see Appendix \ref{appendix: proof_equivalent_BCP_EWF} for its proof.

\begin{proposition}
\label{proposition: equivalent_BCP_EWF}
The workload problem (\ref{EWF}) is equivalent to the Brownian control problem (\ref{BCP}) in the following sense:
\begin{itemize}
    \item[i)] Suppose that $(U,Y,\zeta)$ is an admissible control for Brownian control problem (\ref{BCP}) with Brownian motion $X$ and state process $Z$. Then $(L,U,\theta)$ defined by (\ref{BCP_I_i}) and (\ref{workload_drift}) is an admissible control for the workload problem (\ref{EWF}) with Brownian motion $\xi$ given by (\ref{workload_BM}) and state process $W$ given by (\ref{workload_process}). Moreover, the cost of control $(L,U,\theta)$ for workload problem (\ref{EWF}) is less than or equal to the cost of control $(U,Y,\zeta)$ for the Brownian control problem (\ref{BCP}).
    \item[ii)] Suppose that $(L,U,\theta)$ is an admissible control for the workload problem (\ref{EWF}) with Brownian motion $\xi$ and state process $W$. Then, there exists a Brownian motion $X$ that satisfies (\ref{workload_BM}). Furthermore, there exists $(U,Y,\zeta)$ that is an admissible control for the Brownian control problem (\ref{BCP}) with Brownian motion $X$ and state process $Z$ satisfying (\ref{BCP_queue_length_bound}) and (\ref{workload_process}). These two policies have the same cost.
\end{itemize}
\end{proposition}

Although it is simpler than the Brownian control problem (\ref{BCP}), the workload problem (\ref{EWF}) remains multi-dimensional and does not admit a closed-form solution. In the next section, we develop a computational method to solve it that relies on the neural network technology.

\section{Computational Method\label{sec: computational_method}}

This section develops the computational method to solve the workload problem (\ref{EWF}). To solve the workload problem, that is a singular control problem, we approximate it by a drift control problem, following \cite{Ata_Harrison_Si_2024_singular}. As a preliminary to that analysis, we introduce some auxiliary notation in the next subsection. Sections \ref{sec: computational_method_approximation}--\ref{sec: computational_method_SDE} describe the computational method. Then we propose a policy for the workload problem in Section \ref{sec: computational_method_policy}.

\subsection{Auxiliary Definitions\label{sec: computational_method_notation}}

For the workload problem (\ref{EWF}), we define the $(p+K)$-dimensional process $\mathcal{U}$ as follows:
\begin{equation}
    \mathcal{U} = \begin{pmatrix}
        L\\
        U
    \end{pmatrix}\label{mathcal_U}
\end{equation}
that combines the $p$-dimensional idleness process $L$ and the $K$-dimensional outsourcing process $U$. 
The pushing directions associated with the singular control $\mathcal{U}$ are given by the columns of the $p\times (p+K)$ matrix $\tilde{G}$ given next
\begin{equation}
    \tilde{G} = \begin{pmatrix}
        I_p & -\tilde{M}
    \end{pmatrix},\label{G_bar}
\end{equation}
where $I_p$ is the $p$-dimensional identity matrix and is associated with the idleness process $L(\cdot)$, and $\tilde{M}$ is defined in Equation (\ref{c_x_direction}) and associated with the outsourcing process $U(\cdot)$.
The $(p+K)$-dimensional cost rate vector associated with the singular control $\mathcal{U}$, denoted by $\tilde{\alpha}$, is defined as follows: 
\begin{equation}
    \tilde{\alpha} = \begin{pmatrix}
       0_p \\
       \alpha
    \end{pmatrix},\label{alpha_bar}
\end{equation}
where $0_p$ is the $p$-dimensional vector of zeros. Using this notation, we rewrite Equation (\ref{workload_process_diffusion}), which describes the evolution of the workload process, as follows:
\begin{equation}
W(t)=W(0) +\xi(t)+\int_{0}^{t}\theta(s)ds+\tilde{G}\,\mathcal{U}(t),\ t\geq 0.
\label{workload_process_singular}%
\end{equation}
In addition, the singular control process $\mathcal{U}$ satisfies the following:
\begin{align}
&\mathcal{U}(\cdot) \ \text{is continuous and non-decreasing with}\ \mathcal{U}(0)=0.
\label{mathcal_U_nondecrease}
\end{align}
We can then re-express the workload problem (\ref{EWF}) as follows: Given initial state $W(0) = w\in \mathcal{W}$, choose a non-anticipating control $(\mathcal{U},\theta)$ so as to 
\begin{align}
&    \operatorname*{minimize}\quad J(w;\mathcal{U},\theta) = \mathbb{E}\left[  \int_{0}^{\infty
}e^{-\tilde{r} t}\left(c(\theta(t))dt+ \tilde{\alpha}\cdot d\,\mathcal{U}(t)\right)\right],\label{EWF_simplified} \\
&  \text{subject to\quad 
(\ref{W_state_space_constraint}), (\ref{workload_process_singular}) and (\ref{mathcal_U_nondecrease}).}\nonumber
\end{align}
Then the value function $V$, defined in (\ref{value_function}), can be expressed equivalently as follows:
\begin{equation}
    V(w) = \inf_{\mathcal{U},\theta} J(w;\mathcal{U},\theta),\ w\in \mathcal{W}.\label{value_function_simplified}
\end{equation}

To facilitate the analysis to follow, we represent the bounded convex polyhedral state space $\mathcal{W}$ as the intersection of half spaces as done in \cite{Dai_Williams_1996_existence}. To this end, let $J$ denote the number of associated hyperplanes. Also let $\mathcal{J} = \{1,\dots,J\}$ and $\{\sigma_j, \ j\in \mathcal{J}\}$ denote the set of unit vectors that are normal to the faces of the state space $\mathcal{W}$, pointing towards the interior of the state space. Thus, we write
\begin{align}
    \mathcal{W} = \{w\in \mathbb{R}^p : \ \sigma_j \cdot w \geq h_j\ \text{for all}\ j\in \mathcal{J}\},\label{W_cvx_polyhedron}
\end{align}
where $h_j$ is the constant defining the hyperplane associated with the unit normal vector $\sigma_j$ for $j\in \mathcal{J}$. In addition, we define the $(p-1)$-dimensional faces of $\mathcal{W}$ as 
\begin{align}
    F_j = \{ w \in \mathcal{W}:\ \sigma_j \cdot w = h_j\}, \ j \in \mathcal{J}. \label{W_face}
\end{align}

\subsection{Approximation of the Workload Problem by a Drift Control Problem\label{sec: computational_method_approximation}}

The primary objective of our computational method is to solve the workload problem (\ref{EWF_simplified}), a multi-dimensional singular control problem. 
\cite{Ata_Harrison_Si_2024_singular} proposes a computational method for solving high-dimensional singular control problems that approximates them by a drift control problem. We build on their approach and replace the singular control $\mathcal{U}$ with a drift control $\mathcal{U}_b$, where
\begin{equation}
\mathcal{U}_b(t)=\int_{0}^{t}\psi(s)ds,\ t\geq 0.
\label{singular_control_approximation}%
\end{equation}
Here, $\psi(\cdot)$ is a $(p+K)$-dimensional drift rate process that satisfies the following:
\begin{align}
    \psi_i(t)\in [0,b]\ \text{for}\ t\geq 0 \ \text{and}\ i = 1,\dots,p+K,\label{theta_U_bound}
\end{align}
where $b$ is a given constant that can be viewed as a tuning parameter. 

In addition, we introduce a $J$-dimensional backup control $\Upsilon(\cdot)$ where $\Upsilon_j$ is the backup control assigned to face $F_j$ and its pushing direction is denoted by $\hat{g}_j$. We further require that the backup control satisfies the following:
\begin{align}
    &\Upsilon  \text{ is adapted to } \xi,\ \text{continuous and non-decreasing with } \Upsilon(0) = 0,\label{backup_nondecresing}\\
    &\Upsilon_j(\cdot) \text{ only increases at times } t\geq 0 \text{ when } W(t) \in F_j,\ j = 1,\dots,J.\label{backup_face_assign}
\end{align}
Letting $\hat{G} = (\hat{g}_1,\dots,\hat{g}_J)$, the approximation corresponds to replacing the singular control term $\tilde{G}\,\mathcal{U}(t)$ by $\tilde{G}\,\mathcal{U}_b(t) + \hat{G} \Upsilon(t)$. This approximation gives rise to the following system dynamics equation for the approximating drift control problem:
\begin{equation}
W(t)=W(0) +\xi(t)+\int_{0}^{t}\theta(s)ds+\int_{0}^{t}\tilde{G}\psi(s)ds+\hat{G}\Upsilon(t),\ t\geq 0.
\label{workload_process_singular_approximated}%
\end{equation}
For $j = 1,\dots,J$, we assume there exist $a_j \in \mathbb{R}_+^{p+K}$ such that $\hat{g}_j = \tilde{G}a_j$. That is, $\hat{g}_j$ belongs to the convex cone of the pushing directions $\tilde{g}_1,\dots,\tilde{g}_{p+K}$, the columns of matrix $\tilde{G}$. Then we set 
\begin{align*}
    \hat{\alpha}_j = a_j \cdot \tilde{\alpha} \text{ for } j = 1,\dots, J,
\end{align*}
which corresponds to the cost rate associated with the backup singular control $\Upsilon_j$. And we let $\hat{\alpha} = (\hat{\alpha}_j)$ denote the corresponding $J$-dimensional cost rate vector.

\begin{remark}\label{remark:additional_backup_cost}
    As noted in \cite{Ata_Harrison_Si_2024_singular}, one might make the backup control slightly more expensive than approximating drift control. For our backup controls, we increase their cost slightly to discourage their use and to encourage the use of drift controls that approximate the singular controls.
\end{remark}

Of course, we require that $W(t) \in \mathcal{W}$ for $t\geq 0$. More will be said about this requirement in Appendix \ref{appendix: SRBM}. Assuming for now that one can choose the backup control $\Upsilon(\cdot)$ to ensure this, we next state the approximating drift-control problem: Choose non-anticipating drift controls $\theta(\cdot)$ and $\psi(\cdot)$ to
\begin{align}
    &    \operatorname*{minimize}\quad \tilde{J}(w;\theta,\psi) = \mathbb{E}\left[  \int_{0}^{\infty
}e^{-\tilde{r} t}\Big(\big(c(\theta(t))+ \tilde{\alpha}\cdot \psi(t)\big) dt + \hat{\alpha}\cdot d\Upsilon(t)\Big)\right],\label{EWF_drift_approximation}\\
&  \text{subject to\quad 
(\ref{W_state_space_constraint}), (\ref{theta_U_bound})--(\ref{workload_process_singular_approximated}).}\nonumber
\end{align}
And the corresponding value function, denoted by $\tilde{V}$, is given as follows:
\begin{equation}
    \tilde{V}(w) = \inf_{\theta,\psi} \tilde{J}(w;\theta,\psi),\ w\in \mathcal{W}.\label{value_function_drift_approximation}
\end{equation}

In order to solve the approximating drift-rate control problem, we consider the associated Hamilton-Jacobi-Bellman (HJB) equation in the next section.

\subsection{HJB Equation\label{sec: computational_method_HJB}}

As a preliminary to introducing the HJB equation, given a function $f:\mathbb{R}^p \to \mathbb{R}$ that is twice continuously differentiable, we let $\nabla f$ define its gradient and let $\mathcal{L}$ be the second-order differential operator defined as
$$\mathcal{L}f = \gamma\cdot \nabla f + \frac{1}{2} \sum_{i=1}^p \sum_{j=1}^p \Gamma_{ij}\frac{\partial^2 f}{\partial w_i \partial w_j},$$
where $\gamma$ and $\Gamma$ are drift and variance parameters of Brownian motion $\xi$. We also define a first-order differential operator $\mathcal{D}f = \hat{G}^\top \nabla f$, which corresponds to the directional derivatives at the boundary of the state space along the specified pushing direction for each face.

Following \cite{Ata_Harrison_Si_2024_drift}, it is straightforward to derive the HJB equation, to be solved for the value function $\tilde{V}$, as follows: 
\begin{align}
\mathcal{L}\tilde{V}(w) + \min_{\theta \in \Theta} \{\nabla \tilde{V}(w) \cdot \theta + c(\theta)\} + \min_{\psi\in [0,b]^{p+K}} \{(\tilde{G}^\top \nabla \tilde{V}(w) +\tilde{\alpha})\cdot \psi\}=\tilde{r} \tilde{V}(w),\ w\in \mathcal{W},
\label{HJB}%
\end{align}
with boundary conditions
\begin{align}
\mathcal{D}_j\tilde{V}(w)   = -\hat{\alpha}_j   \text{ for } w \in F_j,\ j  = 1,\dots,J.
\label{HJB_BC}%
\end{align}
Moreover, the policy
\begin{align}
\theta^*(w) = \operatorname*{argmin}_{\theta\in \Theta}\{\nabla \tilde{V}(w) \cdot \theta + c(\theta)\}\ \text{and}\ \psi^*(w) = \operatorname*{argmin}_{\psi\in [0,b]^{p+K}} \{(\tilde{G}^\top \nabla \tilde{V}(w) +\tilde{\alpha})\cdot \psi\},\ w \in \mathcal{W}\label{drift_problem_policy}
\end{align}
is optimal for the approximating drift rate control problem. Since $c(\theta)$ according to Lemma \ref{lemma:c_x} is a quadratic function of $\theta$, the optimal drift control $\theta^*(w)$ in (\ref{drift_problem_policy}) is given as follows:
\begin{align}
    \theta^*(w) = -\frac{1}{2} \tilde{M}H^{-1}\tilde{M}'\nabla\tilde{V}(w),\quad w\in \mathcal{W}.\label{drift_problem_policy_theta_explicit}%
\end{align}
One can verify that $\theta^*(w) \in \Theta$. Also, note that it is linear with respect to the gradient $\nabla \tilde{V}(w)$.
The optimal (approximating) drift control $\psi^*(w)$ in (\ref{drift_problem_policy}) is given as follows: for $i = 1,\dots,p+K$,
\begin{align}
    \psi_i^*(w) = \left\{\begin{array}{ll}
    b & \text{if}\ \tilde{g}_i \cdot \nabla \tilde{V}(w) + \tilde{\alpha}_i\leq 0,\\
    0 & \text{otherwise,}
    \end{array}\right.\label{drift_problem_policy_psi_explicit}
\end{align}
where $\tilde{g}_i$ is the $i$th column of $\tilde{G}$.

\subsection{Reference Process\label{sec: computational_method_reference}}

A key element of the computational method is the reference policy that we choose as one involving constant drift controls, i.e., $\theta(t) = \tilde{\theta}$ and $\psi(t) = \tilde{\psi}$ for $t \geq 0$. Also letting $B(t) = \xi(t)- \gamma t$ for $t\geq 0$ so that $B$ is a driftless Brownian motion with covariance matrix $\Gamma$, the evolution of the controlled workload process $\tilde{W}$ under the reference policy can then be expressed as follows:
\begin{equation}
\tilde{W}(t)=\tilde{W}(0) + \gamma t + B(t) + \tilde{\theta} t   + \tilde{G} \tilde{\psi}t + \hat{G}\tilde{\Upsilon}(t), \label{reference_process}%
\end{equation}
where $\tilde{\Upsilon}$ is the backup singular process that uses the pushing direction $\hat{g}_j$ on face $F_j$ of the state space $\mathcal{W}$ for $j = 1,\dots,J$. We refer to the process $\tilde{W}(\cdot)$ as the reference process. Intuitively, one seeks to choose a reference policy $(\tilde{\theta},\tilde{\psi})$ so that the resulting sample paths of the reference process tend to occupy parts of the state space that the optimal policy is expected to visit frequently. This has been observed in prior literature, e.g., \cite{Ata_Harrison_Si_2024_drift} and \cite{Ata_Kasikaralar_2023_dynamic}. In our context, the reference process should also generate sufficiently many samples near the boundary, especially the upper faces where the costly controls may be used. Thus, learning the gradient is essential for making sound control decisions.

Appendix \ref{appendix: SRBM} argues that one can choose the pushing directions of the backup singular control $\tilde{\Upsilon}$ so that the reference process $\tilde{W}(\cdot)$ is well defined for our problem.

\subsection{An Equivalent SDE and the Solution Method\label{sec: computational_method_SDE}}

To solve the multi-dimensional HJB equation (\ref{HJB})--(\ref{HJB_BC}), we build on \cite{Ata_Harrison_Si_2024_drift} and develop a neural network-based computational method tailored for problems with a bounded convex polyhedral state space. The key idea is to derive a stochastic differential equation (SDE) that is equivalent to the HJB equation (\ref{HJB})--(\ref{HJB_BC}). Utilizing the reference process introduced in Section \ref{sec: computational_method_reference}, Proposition \ref{proposition: equivalent_SDE} derives that SDE, where the value function $\tilde{V}$ and its gradient $\nabla \tilde{V}$ are the unknown functions of interest. In what follows, we compute them using the identity given in Equation (\ref{reference_SDE}); see Appendix \ref{Appendix: proof: equivalent_SDE} for the proof of Proposition \ref{proposition: equivalent_SDE}.

\begin{proposition}\label{proposition: equivalent_SDE}
    If $\tilde{V}$ is a $C^2$ function with polynomial growth that satisfies the HJB equation (\ref{HJB})--(\ref{HJB_BC}), then it satisfies the following SDE almost surely for any $T>0$:
    \begin{align}
    e^{-\tilde{r} T}\tilde{V}(\tilde{W}(T))-\tilde{V}(\tilde{W}(0))  &  =\int_{0}^{T}e^{-\tilde{r} t}\nabla
    \tilde{V}(\tilde{W}(t))\cdot dB(t)\nonumber\\
    &  -\int_{0}^{T}e^{-\tilde{r} t}F(\nabla \tilde{V}(\tilde{W}(t)))dt-\int_{0}^{T}e^{-\tilde{r} t}%
    \hat{\alpha}\cdot d\tilde{\Upsilon}(t), \label{reference_SDE}%
    \end{align}
    where function $F:\mathbb{R}%
    ^{p}\rightarrow\mathbb{R}$ is defined as%
    \begin{equation}
    F(x)=-x\cdot\tilde{\theta}  + \min_{\theta\in\Theta}\{ x\cdot\theta+c(\theta)\} - (\tilde{G}^{\top} x)\cdot \tilde{\psi}+\min_{\psi\in[0,b]^{p+K}}\{(\tilde{G}^{\top}x+\tilde{\alpha})\cdot\psi\}. \label{F_function}%
    \end{equation}
\end{proposition}

To compute the value function $\tilde{V}$ and its gradient $\nabla \tilde{V}$ using identity (\ref{reference_SDE}), we proceed in four steps (see Algorithm \ref{alg:discounted_control} for details). First, we simulate workload process $\tilde{W}$ under reference policy $(\tilde{\theta},\tilde{\psi})$, which will be used as the neural network training data. To be specific, with a chosen time horizon $T$, we simulate $\bar{B}$ independent time-discretized paths of workload process under reference policy using discretization step $h$ (assuming $N:=T/h$ is an integer), and store the following time-discretized workload paths, backup singular control increments, and Brownian increments as data: $\{(\tilde{W}^{(i)}(t),\Delta \tilde{\Upsilon}^{(i)}(t),\Delta B^{(i)}(t)),t = 0,h,\dots,T\},\ i = 1,\dots,\bar{B}$, where the superscript $(i)$ represents $i$th simulated path and $\bar{B}$ denotes the batch size. Algorithm \ref{alg:euler_discretization} in Appendix \ref{appendix: computational_description} provides a detailed description of the first step.

Second, we parameterize
functions $\tilde{V}(\cdot)$ and its gradient $\nabla \tilde{V}(\cdot)$ in identity (\ref{reference_SDE}) by two different neural networks $V_{w_{1}}(\cdot)$ and $G_{w_{2}}(\cdot)$, where $w_1$ and $w_2$ denote the parameters in two neural networks, respectively. The consistency requirement of parameterizing value function $\tilde{V}(\cdot)$ and gradient function $\nabla \tilde{V} (\cdot)$ by different neural networks is justified in Proposition 3 of \cite{Ata_Harrison_Si_2024_drift} for their setting.
Namely, their results suggest the following conjecture: If functions $\tilde{V}(\cdot)$ and $G(\cdot)$ solve 
\begin{align}
    e^{-\tilde{r} T}\tilde{V}(\tilde{W}(T))-\tilde{V}(\tilde{W}(0))  &  =\int_{0}^{T}e^{-\tilde{r} t}G(\tilde{W}(t))\cdot dB(t)\nonumber\\
    &  -\int_{0}^{T}e^{-\tilde{r} t}F(G(\tilde{W}(t)))dt-\int_{0}^{T}e^{-\tilde{r} t}%
    \hat{\alpha}\cdot d\tilde{\Upsilon}(t), \label{reference_SDE_double_parameterize}%
\end{align}
then $G(\cdot) = \nabla \tilde{V}(\cdot)$ and $\tilde{V}(\cdot)$ solves HJB equation (\ref{HJB})--(\ref{HJB_BC}). For brevity, we do not pursue the proof of this conjecture.

Third, we construct the loss function as the expected squared error that arises in the identity (\ref{reference_SDE}) when the neural network approximations $V_{w_1}(\cdot)$ and $G_{w_2}(\cdot)$ are substituted for the value function $\tilde{V}$ and its gradient $\nabla \tilde{V}$:
\begin{align}
    \ell (w_1,w_2) &= \mathbb{E}\bigg[\bigg(e^{-\tilde{r} T}V_{w_1}(\tilde{W}(T))-V_{w_1}(\tilde{W}(0))  -\int_{0}^{T}e^{-\tilde{r} t}G_{w_2}(\tilde{W}(t))\cdot dB(t)\bigg.\nonumber\\
    &  \bigg.+\int_{0}^{T}e^{-\tilde{r} t}F(G_{w_2}(\tilde{W}(t)))dt+\int_{0}^{T}e^{-\tilde{r} t}%
    \hat{\alpha}\cdot d\tilde{\Upsilon}(t)\bigg)^2\bigg], \label{loss_function_SDE}%
\end{align}
and derive its time-discretized version---the integrals are replaced by the
approximating sums. More specifically, we define empirical loss as follows.
\begin{align}
    \hat{\ell} (w_1,w_2) &= \frac{1}{\bar{B}}\sum_{i=1}^{\bar{B}} \bigg(e^{-\tilde{r} T}V_{w_1}(\tilde{W}^{(i)}(T))-V_{w_1}(\tilde{W}^{(i)}(0))  -\sum_{j=0}^{N-1} e^{-\tilde{r} hj} G_{w_2}(\tilde{W}^{(i)}(hj))\cdot \Delta B^{(i)}(hj)\bigg.\nonumber\\
    &  \bigg.+\sum_{j=0}^{N-1}e^{-\tilde{r} hj}F(G_{w_2}(\tilde{W}^{(i)}(hj)))h+\sum_{j=0}^{N-1}e^{-\tilde{r} hj}%
    \hat{\alpha}\cdot \Delta \tilde{\Upsilon}^{(i)}(hj)\bigg)^2.\label{loss_function_SDE_empirical}%
\end{align}
Lastly, we learn
the weight vectors $(w_{1},w_{2})$ of neural networks $V_{w_1}(\cdot)$ and $G_{w_2}(\cdot)$ by minimizing empirical loss (\ref{loss_function_SDE_empirical}) using a standard optimizer; see Algorithm \ref{alg:discounted_control}.

\begin{algorithm}[t]
\begin{small}
\caption{Computational method for solving equivalent SDE (\ref{reference_SDE}).}
\label{alg:discounted_control}
\KwIn{Number of iteration steps $\bar{M}$, batch size $\bar{B}$, time horizon $T$, discretization step size $h$ with $N := T/h \in \mathbb{N}$, starting point $w$, and an optimization solver (SGD, ADAM, RMSProp, etc.).}
\KwOut{Neural network approximations of the value function $V_{w_1}(\cdot)$ and gradient function $G_{w_2}(\cdot)$.}
Initialize the neural networks $V_{w_1}(\cdot)$ and $G_{w_2}(\cdot)$; set $w^{(i)}_1 = w$ for $i = 1, 2, \ldots, \bar{B}$\;
\For{$k \gets 1$ \KwTo $\bar{M}$}{
    Simulate $\bar{B}$ discretized workload paths, backup singular control increments, and Brownian increments $\{\tilde{W}^{(i)}, \Delta \tilde{\Upsilon}^{(i)}, \Delta B^{(i)}\}$ starting from $\tilde{W}^{(i)}(0) = w^{(i)}_k$ by invoking \textsc{Discretize}($T$, $h$, $w^{(i)}_k$) for $i = 1, \ldots, \bar{B}$\;
    Compute the empirical loss $\hat{\ell}(w_1, w_2)$ based on Equation (\ref{loss_function_SDE_empirical})\;
    Compute the gradients $\partial \hat{\ell}/\partial w_1$ and $\partial \hat{\ell}/\partial w_2$, and update $w_1, w_2$ using the chosen optimization solver\;
    Update $w^{(i)}_{k+1} \gets \tilde{W}^{(i)}(T)$\;
}
\KwRet $V_{w_1}(\cdot)$ and $G_{w_2}(\cdot)$\;
\end{small}
\end{algorithm}

\subsection{Proposed Policy for the Workload Problem\label{sec: computational_method_policy}}

Having solved value function $V_{w_{1}}(\cdot)$ and gradient function $G_{w_{2}}(\cdot)$, we propose the drift and singular controls as follows. First, the proposed drift control is
\begin{align}
    \theta^*(w) = -\frac{1}{2} \tilde{M}H^{-1}\tilde{M}'G_{w_{2}}(w),\quad w\in \mathcal{W}.\label{learned_drift_control_closed_form}%
\end{align}
Second, the proposed singular control is active only at the boundary $\partial \mathcal{W}$ of the state space $\mathcal{W}$. For any interior point $w$ of face $F_j$ ($j\in \mathcal{J}$), the index of proposed singular control is
\begin{align}
    i^*(w) \in \operatorname*{argmin}_{i:\  \tilde{g}_i \cdot \sigma_j > 0} \frac{\tilde{g}_i \cdot G_{w_2}(w) + \tilde{\alpha}_i}{\tilde{g}_i \cdot \sigma_j},\label{learned_boundary_control}
\end{align}
where $\tilde{g}_i$ (resp. $\tilde{\alpha}_i$) is the pushing direction (resp. cost rate) of singular control $\mathcal{U}_i$, and $\sigma_j$ is the unit vector normal to face $F_j$ and pointing towards interior of state space $\mathcal{W}$. The denominator $\tilde{g}_i \cdot \sigma_j$ is the amount of displacement along the inward unit normal vector $\sigma_j$ per unit of pushing along the direction $\tilde{g}_i$, and acts as a normalization factor. Our proposed singular control is $\mathcal{U}_{i^*(w)}$. It seeks to choose the pushing direction with the lowest cost per unit of displacement among those pushing directions pointing to the interior of the state space.

The definition in Equation (\ref{learned_boundary_control}) is restricted in two ways. First, we consider only pure singular controls, as these admit a direct implementation in the original queueing network. Second, Equation (\ref{learned_boundary_control}) specifies a pushing direction only at the interior of each face. For completeness, we define the singular control for the remaining boundary points, i.e., the points that lie at the intersection of multiple faces, by simply fixing an arbitrary pushing direction that points to the interior of the state space. One can show that the existence of such pushing directions follows from Proposition \ref{proposition: existence_singular_control}; see Appendix \ref{appendix: SRBM}.

As will be seen in the next section, the proposed drift and singular controls (\ref{learned_drift_control_closed_form})--(\ref{learned_boundary_control}) have straightforward interpretations in original queueing system. And we build on these to propose an integrated pricing, outsourcing, and scheduling policy.

\section{Proposed Policy\label{sec: proposed_policy}}

In this section, we propose a policy for the original queueing system that consists of three components: outsourcing and (discretionary) idling, i.e., pushing, pricing, and scheduling. We introduce the three components of the proposed policy in Sections \ref{sec: policy_outsourcing}--\ref{sec: policy_scheduling}, respectively, and describe how to choose policy parameters in Section \ref{sec: policy_parameters}.

The pricing, outsourcing, and discretionary idling policies proposed for the original queueing network are derived from the solution to the workload problem. Of these three components, the pricing policy follows relatively directly from the drift-rate control policy for the workload problem, subject to a few nuances discussed below. By contrast, the construction of the outsourcing and discretionary idling policies is more involved. To design these policies, we carefully emulate the singular control policy for the workload problem characterized in Equation (\ref{learned_boundary_control}), introducing the necessary modifications to accommodate the original queueing network. We next describe the resulting policy in detail.

As a preliminary, we fix a system parameter $n$ and introduce the unscaled (or nominal) workload process  $W^n = \{W^n(t): t\geq 0\}$ in the $n$th system. Given $n$ and queue length $Q^n$ in $n$th system, we define $W^n$ by
\begin{equation}
W^n_{i}(t)=\sum_{(k,l)\in \mathcal{C}_{i}}m_{k,l}\sum_{j=1}^{l} Q^n_{k,j}(t), \ t \geq 0, \ i = 1,\dots,p,
\label{nth_workload_process}%
\end{equation}
which can be seen as a prelimit version of \eqref{workload_process}, i.e., $W_i^n(t) \approx \sqrt{n} W_i(t)$ for large $n$.

Then we define the corresponding state space of $W^n$, by first introducing a notation $\mathcal{Q}^n$ representing the state space of queue-length process $Q^n$:
\begin{equation}
\mathcal{Q}^n=\left\{  x\in \mathbb{R}_+^m: \sum_{l=1}^{n_{k}}x_{k,l}\leq q^n_{k},\ k = 1,\dots,K\right\},
\label{nth_Q_state_space}%
\end{equation}
then we write the state space of $W^n$ as
\begin{equation}
\mathcal{W}^n=\left\{  w\in \mathbb{R}_+^p:\exists x\in\mathcal{Q}^n\text{ such that }w_{i}%
=\sum_{(k,l)\in \mathcal{C}_{i}}m_{k,l}\sum_{j=1}^l x_{k,j},\ i=1,\ldots
,p\right\} . \label{nth_W_state_space}%
\end{equation}
One can verify that $\mathcal{W}^n = \sqrt{n}\mathcal{W}$, where $\mathcal{W}$ is the state space of (limiting) workload process $W$ as defined in Equation (\ref{W_state_space}).

Although the throughput-time constraints can be satisfied asymptotically in the heavy-traffic limit, in the prelimit systems they can only be guaranteed to hold with high probability. To achieve such a guarantee, we adopt a two-pronged approach. First, we modify the intended state space $\mathcal{W}^n$, possibly shrinking it by a factor $\rho$, so that the maximum probability of violating a throughput-time constraint is below a prescribed tolerance $\varepsilon$, such as $\varepsilon=1\%$. Second, we introduce a band surrounding $\mathcal{W}^n$ that governs the control policy whenever the process exits the set $\mathcal{W}^n$.

More specifically, we first consider a family of policies parameterized by $\rho>0$, under which the queue-length bounds are scaled as follows: $q^n_k(\rho) = \rho\cdot q^n_k$, $k = 1,\dots,K$. The parameter $\rho$ controls the degree of conservatism of the policy and will be chosen to ensure that the prescribed probability-of-violation requirement is satisfied. 
Both limiting scaled and prelimit unscaled workload state spaces are parameterized by $\rho$:
\begin{align}
    \mathcal{W}_\rho = \rho \mathcal{W}\ \text{and}\  \mathcal{W}^n_\rho = \rho \sqrt{n} \mathcal{W}.\label{nth_W_state_space_rho}
\end{align}
For each value of $\rho$, one therefore formulates a Brownian control problem and solves it following our method in Sections \ref{sec: BCP}--\ref{sec: computational_method} to obtain a gradient function $G_{w_2}(\cdot;\rho)$. A similar treatment on queue-lengths bounds can be found in \cite{Ata_2006_thin} and \cite{Rubino_Ata_2009}. We discuss how to choose parameter $\rho$ in Section \ref{sec: policy_parameters}.

Next, we introduce the band surrounding $\mathcal{W}^n_\rho$. While the workload process $W^n$ remains within this band, the proposed outsourcing and discretionary idling policies emulate the singular policy applied at the boundary of the state space $\mathcal{W}_\rho$ in the workload problem. If $W^n$ moves beyond the band, we switch to an outsourcing-only safeguard policy to push the workload towards the state space.

While the pricing, outsourcing, and discretionary idling policies follow from the limiting workload solution, the scheduling policy does not. Our scheduling policy builds on the maximum pressure policy (see, e.g., \cite{Dai_Lin_2005_maximum,Dai_Lin_2008,Ata_Lin_2008}), a throughput-optimal scheduling rule for queueing networks. To account for the throughput-time constraints, we augment the standard pressure metric with a route-level congestion term. Lastly, we incorporate the notion of safety stocks (see, e.g., \cite{Harrison_1996_bigstep,Harrison_1998_heavy,Maglaras_2000_discrete}) to mitigate unplanned idling. This modification is particularly important for networks such as the Rybko-Stolyar network, where the well-known mutual-blocking phenomenon can lead to severe performance degradation.

\subsection{Proposed Outsourcing and Discretionary Idling Policies\label{sec: policy_outsourcing}}

As a preliminary to introducing our proposed outsourcing and discretionary idling policies, we represent the prelimit state space $\mathcal{W}^n_{\rho}$ as the intersection of half-spaces: 
\begin{align}
    \mathcal{W}^n_{\rho} = \{w\in \mathbb{R}^p : \ \sigma_j \cdot w \geq h^n_j(\rho)\ \text{for}\ j\in \mathcal{J}\},\label{nth_W_cvx_polyhedron}
\end{align}
where $h^n_j(\rho):= \sqrt{n}\rho h_j$; cf. Equation (\ref{W_cvx_polyhedron}). In order to define the aforementioned band, we first introduce an enlarged state space with parameter $\epsilon^n>0$ as
\begin{align}
    \mathcal{W}^{n}_{\rho}(\epsilon^n)
    = \left\{ w\in \mathbb{R}_+^p:\;
    \sigma_j \cdot w \geq h_j^n(\rho)-\epsilon^n
    \text{ for } j\in\mathcal{J}
    \right\}.
    \label{relaxed_state_space_choice3}
\end{align}
Then we define the band as $\mathcal{B}_{\rho}^n(\epsilon^n)=\mathcal{W}_{\rho}^n(\epsilon^n) \setminus \mathcal{W}_{\rho}^n$. The parameter $\epsilon^n$ measures the width of the band $\mathcal{B}_{\rho}^n(\epsilon^n)$, and we discuss how to set it in Section \ref{sec: policy_parameters}. In what follows we refer to the region outside enlarged state space, i.e., $\mathbb{R}_+^p\backslash\mathcal{W}^{n}_{\rho}(\epsilon^n)$, as the region beyond the band. 

\begin{figure}[tbp]
\FIGURE
{\includegraphics[width=0.5\figbase]{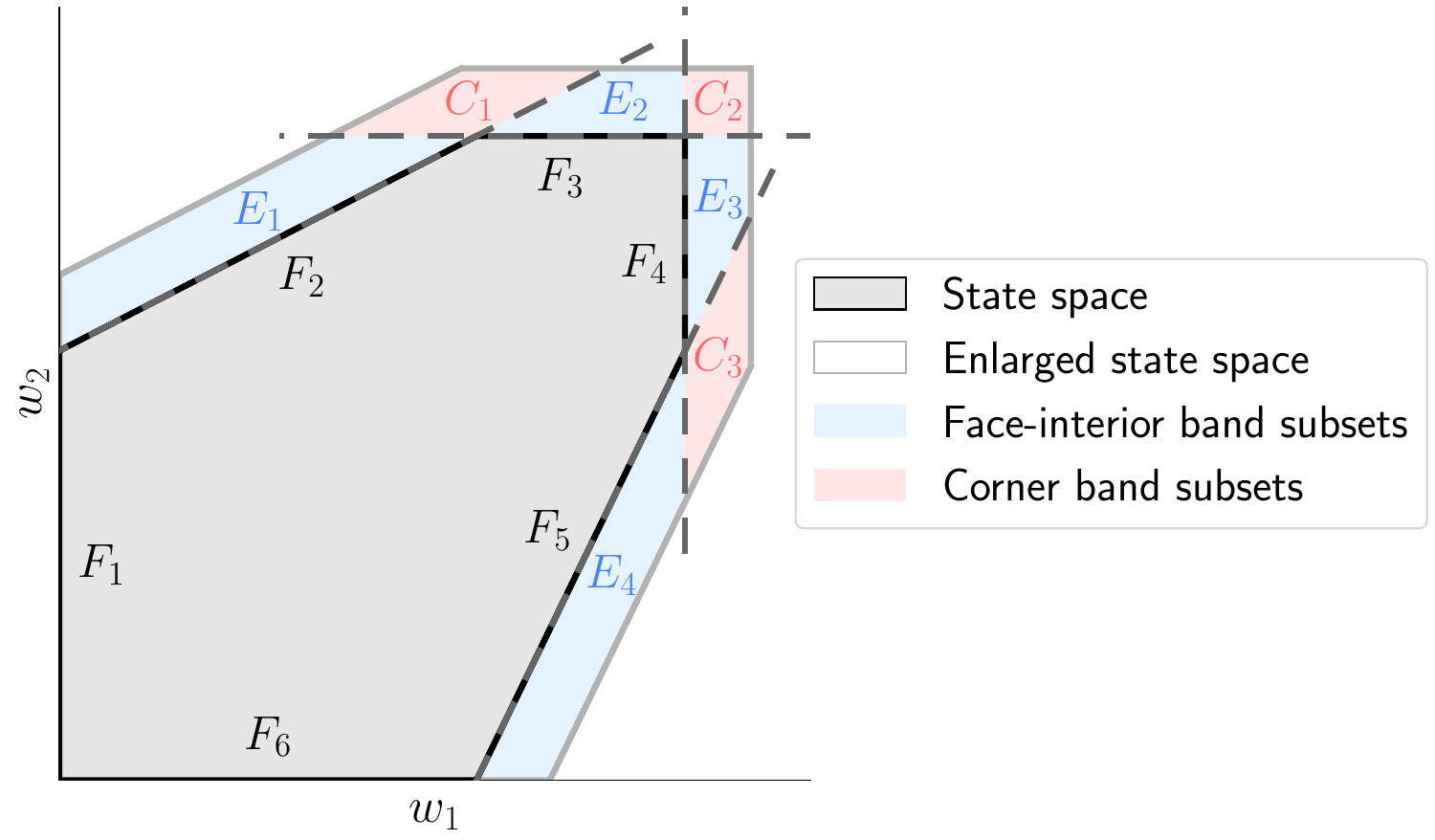}}
{Illustration of prelimit band structure.\label{fig:prelimit_band_illustration}}
{}
\end{figure}

Figure \ref{fig:prelimit_band_illustration} illustrates these definitions for the Rybko-Stolyar example. The gray region corresponds to $\mathcal{W}^n_\rho$, while the larger region enclosing it is $\mathcal{W}^n_\rho(\epsilon^n)$. The difference between these two regions forms the band, which is further partitioned into the subsets $E_1,\ldots,E_4$, shown in blue, and $C_1,\ldots,C_3$, shown in pink. We refer to $E_1,\ldots,E_4$ as the \textit{face-interior band subsets}. For example, $E_4$ encloses the interior points of the face $F_5$ of $\mathcal{W}^n_\rho$, where the faces of $\mathcal{W}^n_\rho$ are denoted by $F_1,\ldots,F_6$. Similarly, we refer to $C_1,C_2$, and $C_3$ as the \textit{corner band subsets}. For example, $C_1$ surrounds the corner formed by the intersection of faces $F_2$ and $F_3$. For points in the corner band subsets, $C_1,\ldots,C_3$, multiple half-space constraints are violated. In contrast, for points in the face-interior band subsets, $E_1,\ldots,E_4$, only one half-space constraint is violated. To define these sets formally, we introduce the set
\begin{align}
    \mathcal{K}_{\text{out}}(w) = \{j\in \mathcal{J}:\ \sigma_j \cdot w < h^n_j(\rho) \},
\end{align}
which denotes the set of violated half-space constraints at workload $w$. Also, its cardinality $| \mathcal{K}_{\mathrm{out}}(w) |$ gives the number of such violations. For sets as in $E_1,\dots,E_4$ we have $| \mathcal{K}_{\mathrm{out}}(w) | =1$ whereas for $w$ in $C_1, C_2$ and $C_3$, $| \mathcal{K}_{\mathrm{out}}(w) | = 2$.

Our outsourcing and discretionary idling policies are applied only when the workload leaves the state space $\mathcal{W}^n_\rho$. We distinguish among three cases, depending on where the workload lies: (i) the face-interior band subsets; (ii) the corner band subsets; and (iii) the region beyond the band.

Before describing our proposed policy, two subtleties are worth mentioning. First, we must specify which gradient value to use for decision-making when the workload vector $w\notin\mathcal{W}_{\rho}^{n}$. This issue arises because the neural network $G_{w_2}(w;\rho)$, which approximates the gradient vector, is trained on $\mathcal{W}_{\rho}^{n}$ and therefore it is more accurate when $w\in\mathcal{W}_{\rho}^{n}$.
When $w\notin\mathcal{W}_{\rho}^{n}$, we instead evaluate the gradient using the most recently observed workload vector in $\mathcal{W}_{\rho}^{n}$. Specifically, let $\tau_1,\tau_2,\ldots$ denote the successive exit times of the workload process $W^n(t)$ from $\mathcal{W}_{\rho}^{n}$. Thus, for each $\ell\geq 1$,
$$ W^n(\tau_\ell)\notin\mathcal{W}_{\rho}^{n}. $$
Between $\tau_\ell$ and $\tau_{\ell+1}$, the process returns to $\mathcal{W}_{\rho}^{n}$ and remains there until its next exit at time $\tau_{\ell+1}$. For any time $t$ at which $W^n(t)\notin\mathcal{W}_{\rho}^{n}$, define
$$ \ell^*(t)=\sup\{\ell\geq 1:\tau_\ell\leq t\}.$$
Then $\tau_{\ell^*(t)}$ is the most recent exit time before $t$, and $W^n(\tau_{\ell^*(t)}^-)$ is the last workload vector observed in $\mathcal{W}_{\rho}^{n}$. Accordingly, we use the gradient estimate $\widetilde{G}(t)$, defined as follows:
\begin{align}
    \widetilde{G}(t) =
\begin{cases}
G_{w_2}\!\left(\frac{W^n(t)}{\sqrt{n}};\rho\right), & \text{if } W^n(t)\in\mathcal{W}_\rho^n,\\[6pt]
G_{w_2}\!\left(\frac{W^n(\tau_{\ell^*(t)}^-)}{\sqrt{n}};\rho\right), & \text{otherwise.}
\end{cases}\label{recent_interior_gradient}
\end{align}

Second, as mentioned earlier, safety stocks will be used to facilitate coordination among servers. Consequently, some servers may be required to operate to maintain the safety stocks of their downstream buffers. This requirement precludes deliberately idling these servers, even when the pushing control associated with the workload problem prescribes doing so. This issue arises under our proposed policy only when the workload process lies within a face-interior band subset and will be discussed further when we examine that case.

More specifically, fix a safety-stock vector $s=(s_{k,l})$. The system manager may exercise discretionary idleness for server $i$, $i=1,\ldots,p$, if the downstream buffer of each of its nonempty buffers contains more jobs than its safety-stock level. Let $S(t)$ denote the set of servers that may be idled at time $t\geq 0$; that is,
\begin{align}
    S(t) = \{i = 1,\dots,p: \forall (k,l)\in \mathcal{C}_i \text{ with } Q_{k,l}(t) > 0,\ Q_{k,l+1}(t) \geq s_{k,l+1}\}.
 \end{align}
In what follows, discretionary idling at time $t$ is permitted only for servers in $S(t)$. We also let $\mathcal{A}=\{p+1,\ldots,p+K\}$ denote the set of pushing directions associated with outsourcing decisions. Next, we describe the proposed outsourcing and discretionary-idling policies for the two types of subsets of the band surrounding the state space of the unscaled workload process.

\textit{Face-interior band subsets.} Recall that when the workload process lies in such a subset, exactly one half-space constraint is violated; that is, \(\mathcal{K}_{\mathrm{out}}(W^n(t))\) is a singleton. For concreteness, suppose that
$ \mathcal{K}_{\mathrm{out}}(W^n(t))=\{j\}$, 
and \(\sigma_j\) is the inward-pointing unit normal to the face \(F_j\). We then choose the index \(i^*\) of the proposed pushing direction as follows:
\begin{align}
    i^* \in \operatorname*{argmin}_{i\in S(t)\cup \mathcal{A}:\  \tilde{g}_i \cdot \sigma_j > 0} \frac{\tilde{g}_i \cdot \widetilde{G}(t) + \tilde{\alpha}_i}{\tilde{g}_i \cdot \sigma_j}.\label{proposed_interior_index}
\end{align}
If $i^*\in S(t)$, the system manager idles server $i^*$. If instead $i^*\in\mathcal{A}$, she outsources the corresponding product type. Under the indexing conventions used in the definition of $\mathcal{A}$ and in (\ref{mathcal_U})--(\ref{workload_process_singular}), the set of outsourced product type is $i^* -p$.

\textit{Corner band subsets and beyond the band.} Recall that when the workload vector lies in a corner-band subset, more than one half-space constraint is violated; that is,
$ \bigl|\mathcal{K}_{\mathrm{out}}(W^n(t))\bigr|>1$. For simplicity, we restrict attention in this case to outsourcing controls. For each $j\in \mathcal{K}_{\mathrm{out}}(W^n(t))$, we choose the index $i^*(j)$ of a pushing direction as follows:
\begin{align}
    i^*(j) \in \operatorname*{argmin}_{i\in \mathcal{A}:\  \tilde{g}_i \cdot \sigma_j > 0} \frac{\tilde{g}_i \cdot \widetilde{G}(t) + \tilde{\alpha}_i}{\tilde{g}_i \cdot \sigma_j},\label{proposed_corner_outside_index}
\end{align}
and we apply the pushing controls
$ \bigl\{i^*(j):j\in \mathcal{K}_{\mathrm{out}}(W^n(t))\bigr\} $
simultaneously, thereby outsourcing the corresponding product types. Under the indexing convention used to define $\mathcal{A}$, the set of outsourced product types is
$ \bigl\{i^*(j)-p:j\in \mathcal{K}_{\mathrm{out}}(W^n(t))\bigr\}$. We proceed in the same manner, when the workload vector is in neither the state space nor the band, i.e., it is beyond the band.

\begin{remark}\label{remark:outsourcing_only_fallback}
The feasible index set in Equations (\ref{proposed_interior_index})--(\ref{proposed_corner_outside_index}) might be empty for some violated face in general. In that case, we define this outsourcing backup policy: at each arrival time $t$, we outsource new arrivals of type-$k$ ($k = 1,\dots,K$) order if its total queue length is at or above bound $q_k^n(\rho)$, i.e., $\sum_{l=1}^{n_k}Q^n_{k,l}(t)\geq q_k^n(\rho)$. This policy remains active until the workload returns to state space $\mathcal{W}_\rho^n$.
This case does not arise in our test problems (see Section \ref{sec: test_problems}).
\end{remark}

\subsection{Proposed Pricing Policy\label{sec: policy_pricing}}

As a preliminary to introducing our proposed pricing policy, we first describe how prices are set dynamically under the drift-rate control policy $\{\zeta(t):t\geq 0\}$ for the Brownian control problem (\ref{BCP}):
\begin{align}
    p^n(t) = \Lambda^{-1}(\lambda^*)+ \frac{\nabla \Lambda^{-1}(\lambda^*)\zeta(t)}{\sqrt{n}}, \ t\geq 0.\label{p_n_Taylor}
\end{align}
Here, $\Lambda^{-1}(\lambda^*)$ is the vector of nominal prices, and $\nabla\Lambda^{-1}(\lambda^*)$ is its Jacobian evaluated at $\lambda^*$. See Appendix \ref{appendix: derivation_policy_pricing} for a derivation of (\ref{p_n_Taylor}) from (\ref{assumption_Lambda_n_mu_n}) and (\ref{lambda_assumption}).

Recall from Lemma \ref{lemma:c_x} and (\ref{learned_drift_control_closed_form}) that the optimal drift-rate control policy for the Brownian control problem is given by
\begin{align}
     -\frac{H^{-1}\tilde{M}'}{2}G_{w_2}(W^n(t)/\sqrt{n};\rho),\ t\geq 0.\label{zeta_ast}
\end{align}
Substituting the proposed choice (\ref{zeta_ast}) for $\zeta(t)$ in (\ref{p_n_Taylor}) yields the proposed pricing policy:
\begin{align}
    p^n(t) = \Lambda^{-1}(\lambda^*) - \frac{1}{2\sqrt{n}} \big(\nabla \Lambda^{-1}(\lambda^*)H^{-1}\tilde{M}'\big) \widetilde{G}(t), \ t\geq 0.\label{pricing_policy}
\end{align}
Here, $\widetilde{G}(t)$ is defined in (\ref{recent_interior_gradient}) and evaluates the gradient on most recently observed workload vector in state space $\mathcal{W}_\rho^n$.

\subsection{Proposed Scheduling Policy\label{sec: policy_scheduling}}

Since scheduling does not have a direct counterpart in the workload problem, we design it separately by building on the maximum-pressure policy (see, e.g., \cite{Dai_Lin_2005_maximum,Dai_Lin_2008,Ata_Lin_2008}). To account for the throughput-time constraints, we augment the standard maximum-pressure policy in two ways. First, we add a route-level congestion term to the pressure of the last buffer on each route, so that jobs on routes close to their throughput-time bounds receive additional priority. Second, when deciding which buffer the server should serve next, we account for safety-stock requirements in the downstream buffer.

More specifically, we let $\Delta_k(t;\tilde{\tau}, \tilde{c})$ represent the congestion level of route $k$ ($k = 1,\dots,K$) at time $t$, where $\tilde{\tau}$ is used in choosing the functional form of such term, and $\tilde{c} \geq 0$ is tuning parameter therein. We provide two specific choices:
\begin{align}
\Delta_k(t;\text{abs}, \tilde{c}) &= \sum_{l'=1}^{n_k} Q^n_{k,l'}(t) - \big(q_k^n(\rho) - \tilde{c}\big),\label{proposed_scheduling_extra_abs} \\
\Delta_k(t;\text{rel}, \tilde{c}) &= \frac{\sum_{l'=1}^{n_k} Q^n_{k,l'}(t)}{q_k^n(\rho)} - \tilde{c}.\label{proposed_scheduling_extra_rel}
\end{align}
For both choices, $\sum_{l'=1}^{n_k} Q^n_{k,l'}(t)$ is the queue length of route $k$ at time $t$ and $q_k^n(\rho)$ is its bound, and $\tilde{c} \geq 0$ is a tuning parameter. The form $\Delta_k(t;\text{abs}, \tilde{c})$ measures route-level congestion by the absolute distance of the route-$k$ queue lengths from its bound, while the form $\Delta_k(t;\text{rel}, \tilde{c})$ measures congestion by the ratio of them. We consider both forms in our policy and treat the choice, i.e., $\tilde{\tau}\in\{\text{abs}, \text{rel}\}$, as a tuning parameter. 

Then, given queue length $Q^n(t)$ at time $t$ in $n$th system and a tuning parameter $\tilde{\lambda}$, we define the pressure term $P_{k,l}(t)$ of buffer $(k,l)$ at time $t$ as
\begin{align}
P_{k,l}(t) = \left\{\begin{array}{ll}  Q^n_{k,l}(t)-Q^n_{k,l+1}(t), & \text{if} \ l < n_k,\\ Q^n_{k,l}(t) +  \tilde{\lambda}\, \Delta_k(t;\tilde{\tau}, \tilde{c}), & \text{if} \ l = n_k,
\end{array}\right.\text{ for } l = 1,\dots,n_k,\ k = 1,\dots,K.\label{scheduling_policy_P}
\end{align}

For each server $i$, the system manager makes prioritization decisions as follows. Suppose that server $i$ does not exercise the discretionary idleness prescribed by the optimal pushing controls for the workload problem. The manager first checks whether the downstream buffer of any nonempty buffer served by server $i$ is below its safety-stock level. If so, server $i$ prioritizes one such buffer, where the choice among such buffers follows a fixed order. Otherwise, for each nonempty buffer $(k,l)\in\mathcal{C}_i$, the system manager computes the index $\mu_{k,l}P_{k,l}(t)$, and server $i$ prioritizes the buffer with the highest index. The ties are broken deterministically according to the fixed order in implementation.

The proposed scheduling policy is a modification of the maximum pressure policy. When $\tilde{\lambda} = 0$ and safety stock $s = 0$ in our setting, the pressure comparison reduces to the standard maximum-pressure policy. The extra term $\Delta_k(t;\tilde{\tau}, \tilde{c})$ for each route $k$ is intended to help meet the associated throughput time constraint. When the system is not congested, this term is small and negative. In this case, meeting the throughput time constraints is not really a concern and the system manager prioritizes the upstream buffers to reduce unnecessary server idleness. However, when the system is congested, this term is larger and positive, reflecting the concern for the throughput time constraints. This leads to prioritizing the last buffer to help meet those constraints.

Tuning parameters in proposed scheduling policy include $\tilde{\tau},\tilde{\lambda}, \tilde{c}$, and $s$. We discuss how to set them, along with other policy parameters, in the next subsection.

\subsection{Choosing Policy Parameters\label{sec: policy_parameters}}

Our proposed policy in Sections \ref{sec: policy_outsourcing}--\ref{sec: policy_scheduling} includes policy parameters $\rho$, $\epsilon^n$, $s$, $\tilde{\tau}$, $\tilde{\lambda}$, $\tilde{c}$. Regarding band-width $\epsilon^n$, we choose it so that
\begin{align}
    \epsilon^n =\epsilon\cdot   n^{\kappa}, \ \kappa \in (0,1/2),\label{epsilon_n_kappa}
\end{align}
where $\kappa$ is a constant and $\epsilon$ is a tuning parameter. The requirement $\kappa \in (0,1/2)$ balances two considerations. First, since $\kappa - 1/2<0$, $\epsilon^n$ is small relative to $\sqrt{n}$, thus the scaled band enclosing the state space is also small. Second, because $\epsilon^n\gg 1$ as $n$ becomes large, the probability that the workload process moves beyond the band is negligible.

Now we have a family of parameters to determine: $\rho$, $\epsilon$, $s$, $\tilde{\tau}$, $\tilde{\lambda}$, $\tilde{c}$. For simplicity, we use a common safety-stock parameter $s$; that is, $s_{k,l}=s$ for all $(k,l)$. We determine the best value of these parameters via simulation-based search, where the objective is to ensure delay probabilities of all routes do not exceed $\varepsilon = 1\%$ while minimizing cost in the original system.
To be specific, for each candidate value of $\rho$, we first solve the corresponding Brownian control problem and obtain the neural-network approximation of the gradient function $G_{w_2}(\cdot;\rho)$ (see Section \ref{sec: computational_method}). We then evaluate the full grid of the remaining five parameters using this gradient $G_{w_2}(\cdot;\rho)$. Repeating the above procedure for every candidate value of $\rho$ amounts to an exhaustive search over the full six-dimensional parameter grid. Although the search is six-dimensional, the candidate sets for several parameters are small (e.g., safety stock $s$ and scheduling form $\tilde{\tau}\in \{\text{abs},\text{rel}\}$), resulting in a total number of parameter combinations on the order of $10^4$ in our test problems, making this brute-force search feasible; see search grid of each parameter in Table \ref{tab:search_grid} of Appendix \ref{appendix: prelimit_implementation_details}.

\section{Test Problems and Benchmark Policies}

\subsection{Test Problems\label{sec: test_problems}}

In this subsection, we introduce two test problems and describe their model primitives. The first is Rybko-Stolyar network \citep{Rybko_Stolyar_1992} whose workload problem is two-dimensional. The second one is the extended criss-cross network, which builds on criss-cross network \citep{Harrison_Wein_1989_scheduling} and provides a set of high-dimensional problems: we consider a $3$-dimensional case and a $10$-dimensional case. For all test problems, we set system number as $n = 400$ and summarize model parameters for $n$th system in Table \ref{tab: Prelimit_parameters}.

\begin{table}[htbp]
\TABLE
{Parameters in the $n$th system where $n = 400$. In the $p$-server extended criss-cross network, numerical values of parameters depend on $p$, the number of servers.\label{tab: Prelimit_parameters}}
{\begin{small}
\setlength{\tabcolsep}{1.2em}
\renewcommand{\arraystretch}{1.0}
\begin{tabular}{lcc}
\toprule
& Rybko--Stolyar network & $p$-server extended criss-cross network \\
\midrule
Demand function $\Lambda^n(x)$    & $(316.66-63.33x_1,\ 316.66-63.33x_2)$ & $(475-95x_1,\dots,475-95x_{p+1})$ \\
Production cost $\delta^n$        & $(1,1)$            & $(1,\dots,1)$ \\
Nominal arrival rate $(\lambda^n)^*$ & $(126.67,126.67)$ & $(190,\dots,190)$ \\
Nominal static price $(p^n)^*$    & $(3,3)$            & $(3,\dots,3)$ \\
Outsourcing cost $\nu^n$          & $(0.05,0.125)$     & $(0.05,\dots,0.05,\ 0.025p)$ \\
Service rate $\mu^n$              & $(400,200,400,200)$ & $(400,\dots,400)$ \\
Average utilization               & $(0.95,0.95)$      & $(0.95,\dots,0.95)$ \\
Throughput time bound $d^n$       & $(0.33,0.33)$      &
  \renewcommand{\arraystretch}{0.85}\begin{tabular}{@{}c@{}}
  $(0.15,\dots,0.15,\ 0.225)$ for $p = 3$\\
  $(0.15,\dots,0.15,\ 0.6)$ for $p = 10$
  \end{tabular} \\
Interest rate $r^n$               & $0.01$             & $0.01$ \\
\bottomrule
\end{tabular}
\end{small}}
{}
\end{table}

\subsubsection{Rybko-Stolyar Network.\label{sec: test_problem_RS_network}%
}

\begin{figure}[htbp]
\FIGURE
{\includegraphics[width=0.65\figbase]{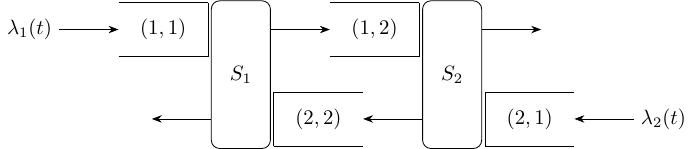}}
{Rybko-Stolyar network. \label{fig:RS_network}}
{}
\end{figure}

As shown in Figure \ref{fig:RS_network}, the Rybko-Stolyar network consists of $p = 2$ servers and $K = 2$ types of products. Demand of product $1$ (resp. $2$) arrives to the system according to a nonhomogeneous Poisson process with instantaneous rate $\lambda^n_1(t)$ (resp. $\lambda^n_2(t)$) for $t \geq 0$. Once arrived, type $1$ (resp. $2$) customer orders need service sequentially at servers $1$ and $2$ (resp. $2$ and $1$), and then leave the system. There are in total $4$ classes of jobs and the jobs of each class reside in the corresponding buffer (represented by open-ended rectangle) both before and during the service. We assume that the service time at each queue is exponentially distributed. Second column of Table \ref{tab: Prelimit_parameters} summarizes all model primitives in Rybko-Stolyar network. The nominal arrival rate and service rate lead to an average utilization of $0.95$ for each server.

Introduced by \cite{Rybko_Stolyar_1992}, the Rybko-Stolyar network provides an example of a system which is unstable under a certain work-conserving static priority scheduling rule and certain parameter region, even when the average utilization of both servers is less than one. Our nominal arrival rate and mean service time in second column of Table \ref{tab: Prelimit_parameters} lie in such a region.

If one regards the time unit in the $n$th system as three months, then the demand rate of each product is around $1.5$ per day, the average total production time for each product is around $2/3$ days, the throughput time bound is about one month, and the annual interest rate is $4\%$.

\subsubsection{Extended Criss-cross
Network.\label{sec: test_problem_extended_CC}}

\begin{figure}[htbp]
\FIGURE
{\includegraphics[width=0.7\figbase]{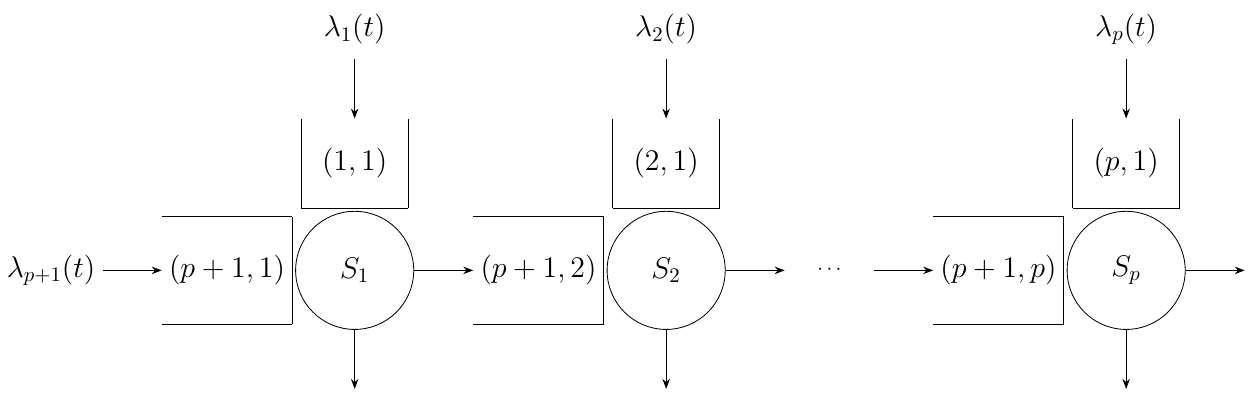}}
{Extended criss-cross network. \label{fig:CC_network}}
{}
\end{figure}

Figure \ref{fig:CC_network} shows a network with $p$ servers and $K = p+1$ routes, which builds on criss-cross network by adding more servers and routes to it, thus we refer to it as extended criss-cross network. In such a network, type $i$ ($i = 1,\dots,p$) product needs service only at server $i$, whereas type $p+1$ product needs service sequentially at all $p$ servers. For convenience, we also refer to route $1,\dots,p$ (resp. $p+1$) as a vertical (resp. the horizontal) route. As we vary $p$ (i.e., the number of servers), we obtain a set of high-dimensional problems. We emphasize the dependence of network on parameter $p$ by calling it $p$-server extended criss-cross network. 

Third column of Table \ref{tab: Prelimit_parameters} summarizes all model primitives in $p$-server extended criss-cross network. As can be seen from the table, our specific choice of model primitives also depends on parameter $p$. Regarding the throughput time bound, we choose the same throughput time bound $d^n_v$ for all vertical routes, and $d^n_h$ for the horizontal route. The values of $d^n_v$ and $d^n_h$ are set to be of moderate magnitude such that the control problem is neither ``trivial'' nor ``infeasible''. Indeed, if the throughput time bounds are too large (resp. small), the control problem becomes less interesting (resp. infeasible), since throughput time requirement will be achieved with nearly zero control (resp. only with an excessive amount of outsourcing). In our computational results we separately consider a $3$-server and a $10$-server case, for which we set $(d^n_v,d^n_h)$ equal to $(0.15,0.225)$ and $(0.15, 0.6)$ respectively. 

If one regards time unit in the $n$th system as three months, then the demand for each product is around $2$ per day, the average production time for each service is around $1/4$ days, the throughput time bound is approximately $2$ weeks for vertical routes and $3$ (resp. $8$) weeks for horizontal route in $3$-server (resp. $10$-server) case, and the annual interest rate is $4\%$.

\subsection{Benchmark Policies\label{sec: benchmark_main_text}}

To assess the performance of our proposed policy, we consider $12$ benchmark policies. More specifically, we consider two alternatives for each of the three policy components that are briefly described in Table \ref{tab: benchmark_three_parts_description} and detailed in Appendices \ref{sec: benchmark_outsourcing}--\ref{sec: benchmark_scheduling}. Regarding the scheduling component, we further include our proposed scheduling policy as an alternative, as it is independent from our neural network-based solution to workload problem. This allows us to distinguish the effect of the scheduling policy from those of the other policy components. Each benchmark policy contains tuning parameters in its three policy components (see Table \ref{tab: benchmark_three_parts} in Appendix \ref{sec: benchmark}), and Appendix \ref{sec: benchmark_combination} explains how these parameters are selected.

\begin{table}[htp]
\TABLE
{Benchmark policies in outsourcing and discretionary idling, pricing, and scheduling. 
\label{tab: benchmark_three_parts_description}}
{\begin{small}
\tabcolsep=0.6em
\renewcommand{\arraystretch}{1.0}
\begin{tabular}{
    p{0.42\figbase}
    p{0.50\figbase}
}
\toprule
Policy & Brief description \\
\midrule

\multicolumn{2}{l}{
\textit{Panel A: Outsourcing and Discretionary Idling Benchmark Policies 
(Appendix \ref{sec: benchmark_outsourcing})}
} \\
\midrule

\textit{Single pushing control per face policy}
&
Assigns a single pushing control to each face of the state space of the workload process based on instantaneous pushing cost, otherwise mimics our proposed outsourcing and discretionary idling policies. \\
\addlinespace[0.4em]

\textit{Threshold-based policy}
&
Outsources an arriving job of a product whenever the total number of jobs on its route exceeds a specified threshold. \\

\midrule
\multicolumn{2}{l}{
\textit{Panel B: Pricing Benchmark Policies 
(Appendix \ref{sec: benchmark_pricing})}
} \\
\midrule

\textit{Linear pricing policy}
&
Follows the dynamic pricing policy (\ref{p_n_Taylor}) prescribed by the Brownian control problem, and assumes a linear dependence of drift control $\zeta(t)$ in (\ref{p_n_Taylor}) on queue lengths in original system. \\
\addlinespace[0.4em]

\textit{High-low pricing policy}
&
Follows the dynamic pricing policy (\ref{p_n_Taylor}) prescribed by the Brownian control problem, and sets different values for drift control $\zeta(t)$ in (\ref{p_n_Taylor}) depending on the number of jobs in each route relative to a specified threshold. \\

\midrule
\multicolumn{2}{l}{
\textit{Panel C: Scheduling Benchmark Policies 
(Appendix \ref{sec: benchmark_scheduling})}
} \\
\midrule

\textit{Proposed scheduling policy}
&
See Section \ref{sec: policy_scheduling}. \\
\addlinespace[0.4em]

\textit{Least-relative-slack policy}
&
Defines an index of each buffer as the age of the oldest job relative to its throughput-time bound, prioritizes the buffer with the highest index, and further incorporates safety stocks. \\
\addlinespace[0.4em]

\textit{Least-relative-slack policy based on queue lengths}
&
Rests on a similar idea as the least-relative-slack policy, but approximates the age using queueing lengths in current and upstream buffers (normalized by arrival rate of its route), and further incorporates safety stocks.  \\

\bottomrule
\end{tabular}
\end{small}}
{}
\end{table}

\section{Computational Results\label{sec: simulation}}

In this section, we present our computational results. We first present our proposed policy in Section \ref{sec: simulation_our_proposed_policy}, including our solution to the workload problem and the parameters of the proposed policy for the original system. Then we compare the performance of our proposed policy with the benchmark policies in Section \ref{sec: simulation_results}, using Monte Carlo simulation. For all test problems, we provide details of the workload problem, the neural-network training, the Monte-Carlo simulation setup, and the computational burden in Appendix \ref{appendix: test_problem_details}.

\subsection{Proposed Policy\label{sec: simulation_our_proposed_policy}}

We first present our proposed policy for the workload problem separately for two test problems, and interpret their direct implication for the outsourcing, discretionary idling, and pricing policy in the original system. Then we describe the parameters of our proposed policy, which are introduced in Section \ref{sec: proposed_policy} and tuned as described in Section \ref{sec: policy_parameters}.

\subsubsection{Proposed Policy for the Workload Problem for the Rybko-Stolyar Network.\label{sec: simulation_our_proposed_policy_workload_RS}}

As a preliminary to presenting our proposed policy for the workload problem, we illustrate the (limiting) workload state space. It is a two-dimensional polyhedron 
\begin{align}
\mathcal{W} = \left\{ w\in \mathbb{R}_+^{2}:  w_1,w_2 \leq 6.67,\ 2w_1 - 8.89 \leq w_2 \leq 0.5w_1 +4.44\right\},\label{W_state_space_RS}
\end{align}
represented by the grey polygon in Figure \ref{fig:RS_state_space}. There are four singular controls $L_1$, $L_2$, $U_1$, and $U_2$, with pushing matrix $\tilde{G}$ and costs vector $\tilde{\alpha}$ given by
\begin{align}
    \tilde{G} = \begin{pmatrix}1&0&-1.05&-2.10\\
0&1&-2.10&-1.05\end{pmatrix}\ \  \text{and} \ \  \tilde{\alpha} = \begin{pmatrix}
    0 & 0 & 1 & 2.5
\end{pmatrix}^\top.\label{singular_direction_cost_RS}
\end{align}
Here, the pushing direction of $i$th $(i = 1,\dots,4)$ singular control is given by $i$th column of matrix $\tilde{G}$. Figure \ref{fig:RS_state_space} plots the pushing directions of four singular controls as arrows with different colors. Note that all singular controls can be applied at any point of the state space, not restricted to the boundary points or the edge visually close to the plotted arrow. 

\begin{figure}[htbp]
\FIGURE
{\includegraphics[width=0.35\figbase]{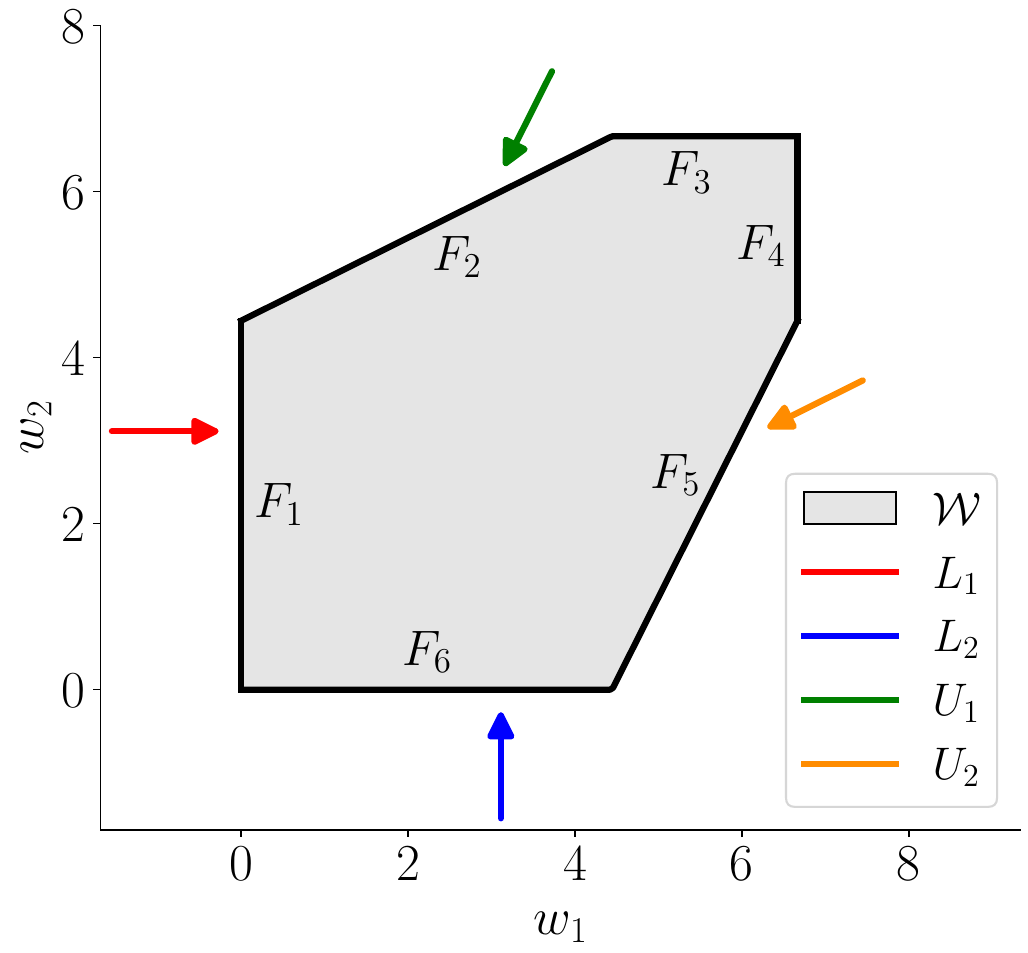}}
{Workload problem of Rybko-Stolyar network. The polygon represents the state space $\mathcal{W}$. Arrows with different colors represent the pushing directions of four singular controls, including $L_i \ (i = 1,2)$ and $U_k \ (k = 1,2)$.\label{fig:RS_state_space}}
{}
\end{figure}

It is worth mentioning that, at each face of state space multiple singular controls might be feasible (i.e., pointing towards the interior of state space), offering the flexibility in choosing singular controls at each boundary point. To facilitate the analysis, we let $F_j$ ($j = 1,\dots,6$) denote six edges of state space $\mathcal{W}$ in the clockwise order starting from the origin (see Figure \ref{fig:RS_state_space}), and identify the set of feasible singular controls on (interior of) each face $F_j$ ($j = 1,\dots,6$):
\begin{align}
    \left(
        \{L_1\} \ \{L_1,U_1\} \ \{U_1,U_2\} \ \{U_1,U_2\} \ \{L_2,U_2\} \ \{L_2\}
    \right). \label{W_cvx_polyhedron_RS_available}
\end{align}
For example, both $L_1$ and $U_1$ can be used on face $F_2$ (i.e., the upper-left edge) when it's necessary to confine the workload process to the state space. 

Recall that our proposed singular control (\ref{learned_boundary_control}) selects one singular control among a feasible set (\ref{W_cvx_polyhedron_RS_available}) for each boundary point. Figure \ref{fig:RS_EWF_solu_singular} draws the proposed active region of each singular control on the boundary of the state space. Two features stand out. First, the singular control $U_2$ with high cost (see Equation (\ref{singular_direction_cost_RS})) is never used, and $L_1$, $L_2$, and $U_1$ together partition the boundary. Second, a single face can host more than one singular control: on face $F_2$ (i.e., the upper-left edge), $L_1$ is used on the left segment and $U_1$ on the right. Intuitively, when both discretionary idling and outsourcing controls are feasible on the same face, free discretionary idling is preferred at small workloads and costly outsourcing at large ones; we will see a similar pattern in next test problem.

\begin{figure}[htbp]
\FIGURE
{\includegraphics[width=0.35\figbase]{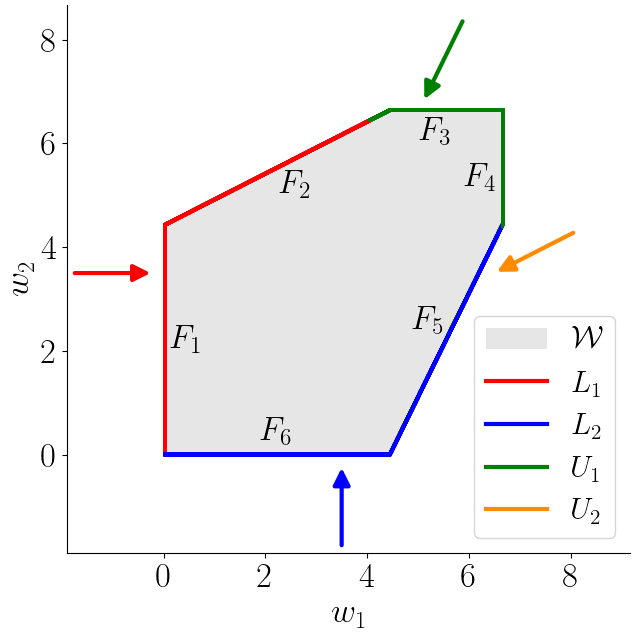}}
{Proposed singular control in workload problem of Rybko-Stolyar network. The red (resp. blue, green) region on the boundary of state space $\mathcal{W}$ represents the proposed active region of singular control $L_1$ (resp. $L_2$, $U_1$), using the same color as the arrow representing its pushing direction. \label{fig:RS_EWF_solu_singular}}
{}
\end{figure}

On face $F_2$ (i.e., the upper-left edge), the active segment of $U_1$ shows the long-term effect incorporated in our proposed singular control (\ref{learned_boundary_control}), i.e., the gradient $G_{w_2}(w)$. Indeed, for a policy that neglects such long-term effect (e.g., the single-pushing benchmark policy in Appendix \ref{sec: benchmark_outsourcing}), the singular control $L_1$ can be used on whole range of face $F_2$. 

Next, we consider the implications of our proposed singular control for the original queueing network. First, we notice that a face in the workload problem represents a specific congestion scenario in the original system. Take face $F_2$ as an example, it represents a scenario that the workload of two servers is not necessarily high but imbalanced, and route $1$ is congested. In fact, on face $F_2$ (i.e., $w_2 =  0.5w_1+ 4.44$), the workload of server $2$ is higher than (half of) that of server $1$ by a significant margin, and only route-$1$ jobs positively contribute to such margin. Second, given a congestion scenario, different feasible singular controls map to different ways in dealing with it. For example, on face $F_2$, an active $U_1$ (resp. $L_1$) suggests that outsourcing route-$1$ product (resp. allocating more service capacity to route $1$) is more economical in reducing congestion in route $1$ when workload is high (resp. low).  

Now we present our proposed drift control for the workload problem. Recall that the proposed drift control (\ref{learned_drift_control_closed_form}) is a two-dimensional vector defined in the (limiting) state space $\mathcal{W}$. We use the left and right panels of Figure \ref{fig:RS_EWF_solu_drift} to show its first and second coordinates across the workload space $(w_1,w_2)$. Three features of proposed drift control stand out. First, each coordinate takes mostly negative values. Second, each coordinate becomes more negative as the corresponding workload component increases. Intuitively, when $w$ lies close to the upper boundary of its state space, stronger drift control is favored because it steers the workload away from the boundary, where costly singular control would otherwise be used. Third, each coordinate is non-linear with respect to workload and is steeper for larger workloads. 

\begin{figure}[htbp]
\FIGURE
{\includegraphics[width=0.6\figbase]{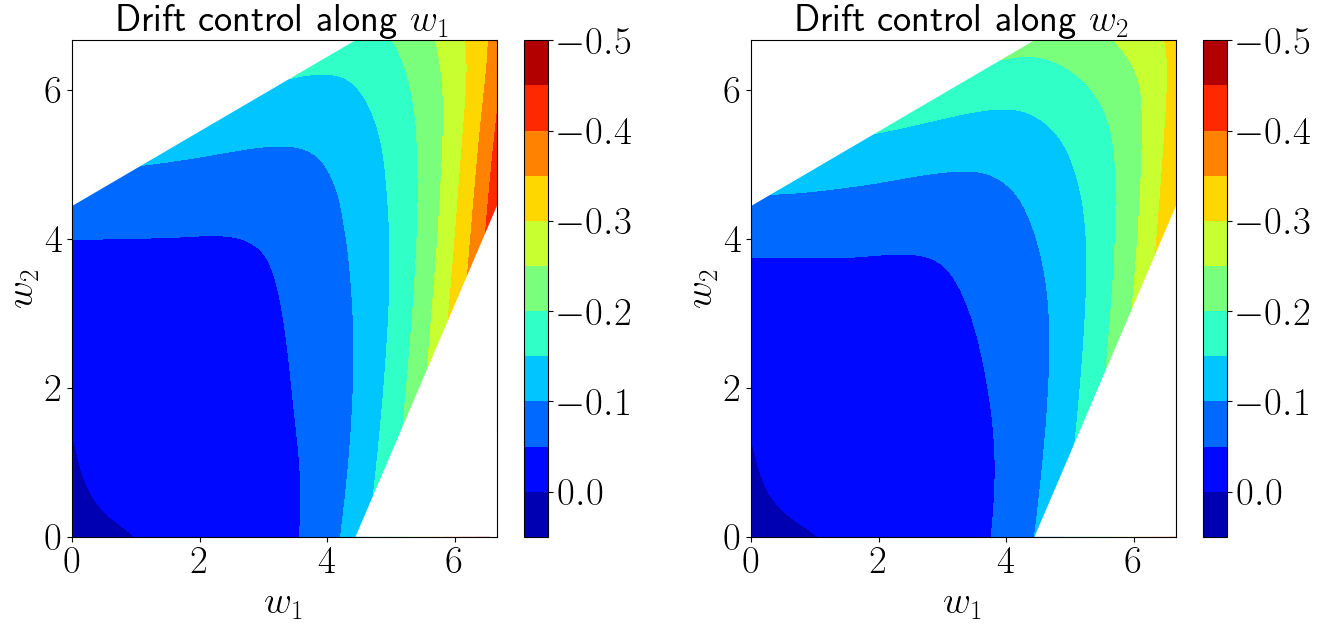}}
{Proposed drift control in workload problem of Rybko-Stolyar network. The left (resp. right) panel plots first (resp. second) coordinate of proposed drift control $\theta^*(w)$ across the workload space $(w_1,w_2)$. \label{fig:RS_EWF_solu_drift}}
{}
\end{figure}

Aforementioned three features in Figure \ref{fig:RS_EWF_solu_drift} suggest the following properties of our proposed pricing control in the original system. First, the pricing control in our problem is mainly used to reduce the congestion level of the system. Second, this role is stronger as the congestion level increases. Third, the pricing control depends non-linearly on queue lengths in the original system.

\subsubsection{Proposed Policy for the Workload Problem for the Extended Criss-cross Network.\label{sec: simulation_our_proposed_policy_workload_ECC}}

As a preliminary to presenting our proposed policy for the workload problem, we describe the state space and singular controls. The state space is a $p$-dimensional bounded convex polyhedron
\begin{align}
    \mathcal{W} = \{ w\in \mathbb{R}_+^{p}:\ w_i \leq w_j + 10d^n_v, \ 1\leq i < j \leq p;\ \ 
    w_i \leq 
    10d^n_v+10d^n_h,\ 1\leq i \leq p\},\label{W_state_space_ECC}
\end{align}
where $d^n_v$ and $d^n_h$ are throughput time bounds for vertical and horizontal routes respectively. There are in total $2p+1$ singular controls, including $p$ idling controls $L_1,\dots,L_{p}$ and $p+1$ outsourcing controls $U_1,\dots,U_{p+1}$. The matrix $\tilde{G}$ of pushing directions and the associated cost rate vector $\tilde{\alpha}$ are given as follows:
$$\tilde{G} = \begin{pmatrix}
    1 & & & -1.05 & &  &-1.05\\
    & \ddots & & & \ddots & &\vdots \\
    & & 1 & & & -1.05 & -1.05
\end{pmatrix}\ \  \text{and} \ \  \tilde{\alpha} = \begin{pmatrix}
    0 & \cdots & 0 & 1 & \cdots & 1 & 0.5p
\end{pmatrix}^\top,$$
where $i$th column of matrix $\tilde{G}$ represents pushing direction of $i$th singular control ($i = 1,\dots,2p+1$). As seen in Equation (\ref{W_state_space_ECC}), the faces of state space $\mathcal{W}$ can be classified into three types; see Table \ref{tab:faces_ECC}. For each type of face, one can identify the set of feasible (i.e., pointing towards the interior of state space) singular controls on (interior of) each face. For the first (resp. second and third) type of faces, there is only one (resp. two) feasible singular control.

\begin{table}[htbp]
\TABLE
{Face classification of state space $\mathcal{W}$ in $p$-server extended criss-cross network.\label{tab:faces_ECC}}
{\begin{small}
\setlength{\tabcolsep}{0.5em}
\renewcommand{\arraystretch}{0.9}
\begin{tabular}{lcccccc}
\toprule
Type & & Face & & Feasible singular controls & & Range \\
\midrule
i   & & $w_i = 0$                   & & $\{L_i\}$          & & $1 \leq i \leq p$ \\
ii  & & $w_i = w_j + 10 d_v^n$      & & $\{L_j, U_i\}$     & & $1 \leq i < j \leq p$ \\
iii & & $w_i = 10 d_v^n + 10 d_h^n$ & & $\{U_i, U_{p+1}\}$ & & $1 \leq i \leq p$ \\
\bottomrule
\end{tabular}
\end{small}}
{}
\end{table}

Next, we describe our proposed policy for the workload problem in $3$-server case (i.e., when $p = 3$). Figure \ref{fig:CC_state_space} shows the state space and the singular controls, where the grey polyhedron represents the $3$-dimensional state space $\mathcal{W}$, and seven arrows with different colors represent pushing directions of singular controls $L_i \ (i = 1,\dots,3)$ and $U_k \ (k = 1,\dots,4)$. As discussed earlier and shown in Table \ref{tab:faces_ECC}, the faces of the state space can be divided into three types. For the first type, there is only one feasible pushing direction for the faces. Thus, the singular-control decision is trivial for those faces. For the second (resp. third) type, the proposed singular control is shown in the first (resp. second) row of graphs in Figure \ref{fig:CC_EWF_solu_singular}. For each of the six figures, two pushing controls are used in different parts of the face. More specifically, for the first row of graphs, an idleness control is used virtually on the entire face and an outsourcing control is used only in a small part of the face. In contrast, for the graphs in the second row, each face is somewhat evenly split between two outsourcing controls.

\begin{figure}[htbp]
\FIGURE
{\includegraphics[width=0.35\figbase]{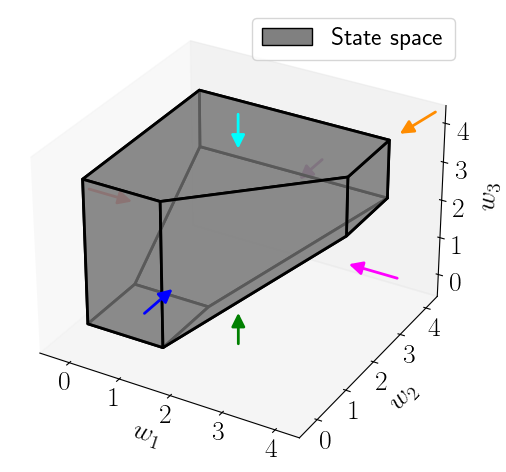}}
{Workload problem of $3$-server extended criss-cross network. The grey polyhedron with black edges represents the state space $\mathcal{W}$. Seven arrows with different colors represent pushing directions of singular controls, including $L_i\ (i = 1,\dots,3)$ and $U_k\ (k = 1,\dots,4)$.\label{fig:CC_state_space}}
{}
\end{figure}

\begin{figure}[htbp]
\FIGURE
{\includegraphics[width=0.6\figbase]{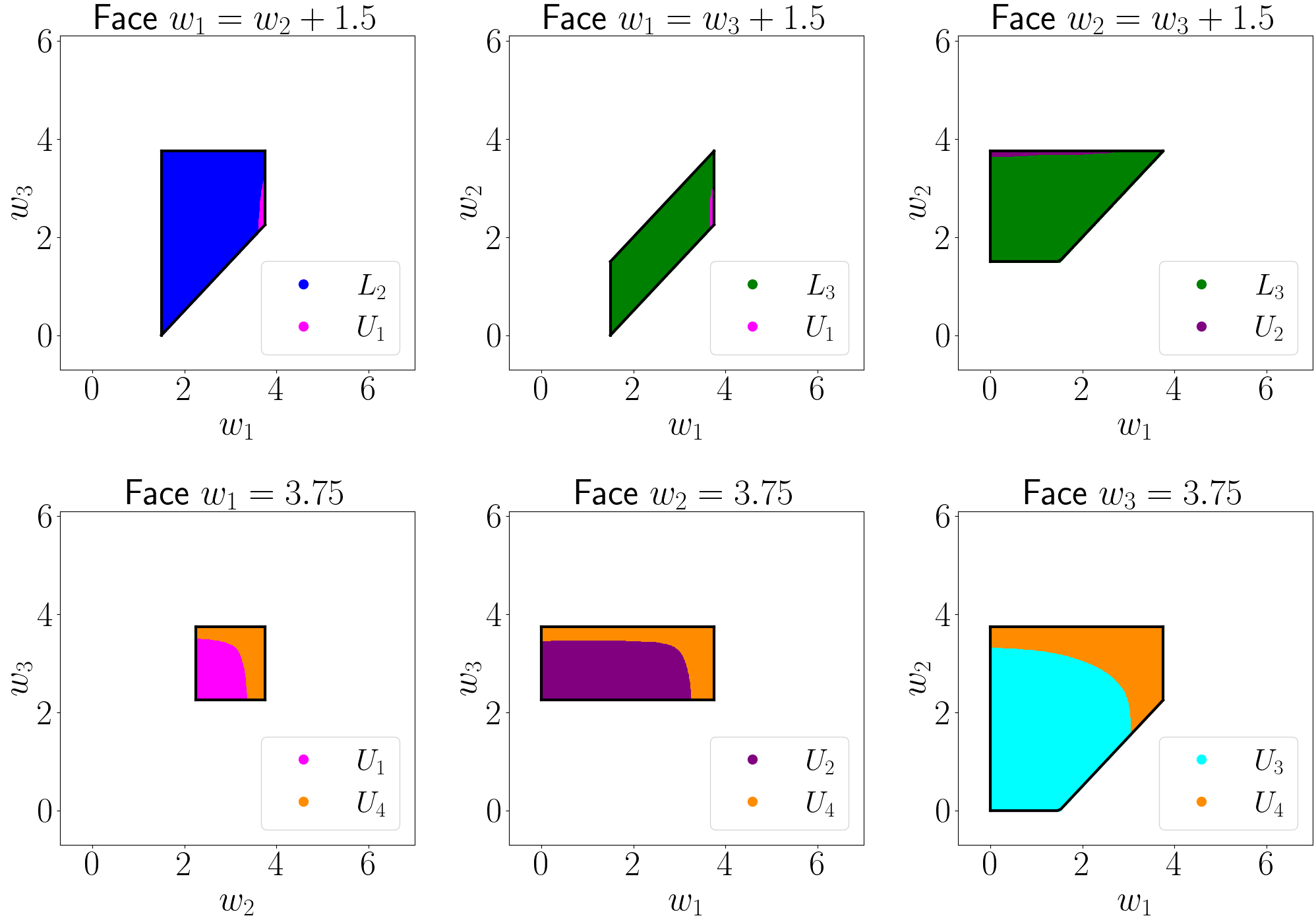}}
{Proposed singular control in workload problem of $3$-server extended criss-cross network. Each panel shows proposed active region of singular controls on each second- or third-type face of state space $\mathcal{W}$ (see Table \ref{tab:faces_ECC}).\label{fig:CC_EWF_solu_singular}}
{}
\end{figure}

To study the performance of our proposed singular control displayed in Figure \ref{fig:CC_EWF_solu_singular}, we compare it with single-pushing benchmark policy (see Appendix \ref{sec: benchmark_outsourcing}), which greedily minimizes instantaneous cost and would use idling controls $L_2,L_3,L_3$ (resp. outsourcing controls $U_1,U_2,U_3$) for three second-type (resp. third-type) faces.

Lastly, we show our proposed drift control in $3$-server problem. Figure \ref{fig:CC_EWF_solu_drift} shows the three coordinates of it in three panels, each plotted as colored contours inside aforementioned three-dimensional polyhedral state space (see Figure \ref{fig:CC_state_space}). Similar to the Rybko-Stolyar problem (see Figure \ref{fig:RS_EWF_solu_drift}), the proposed drift control in Figure \ref{fig:CC_EWF_solu_drift} takes mostly negative values, has each coordinate growing in magnitude with corresponding workload component, and shows non-linearity with respect to workload. 

\begin{figure}[htbp]
\FIGURE
{\includegraphics[width=0.8\figbase]{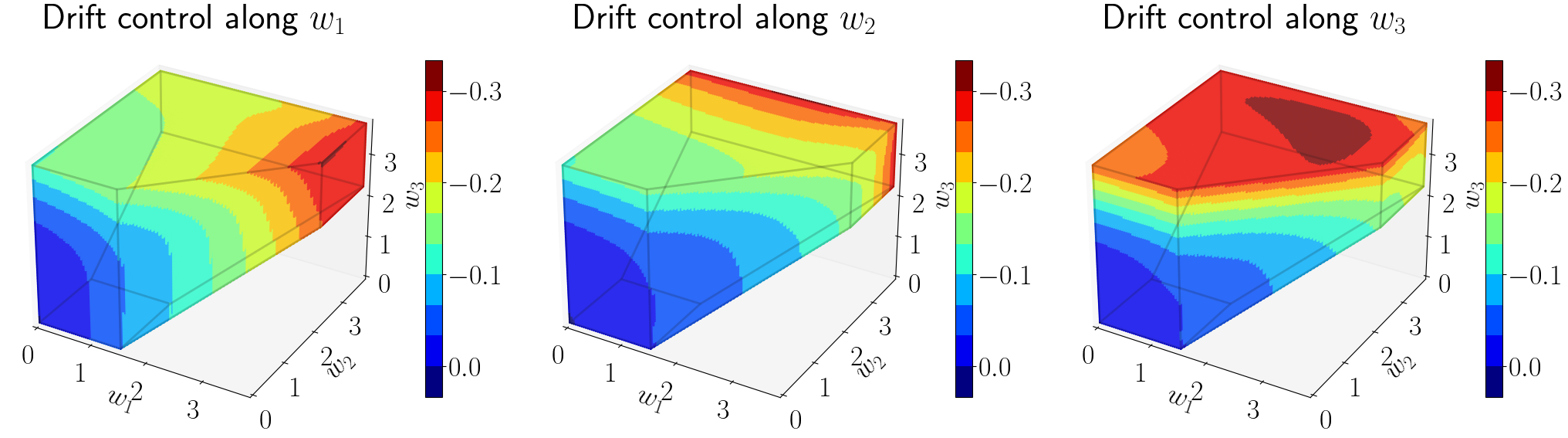}}
{Proposed drift control in workload problem of $3$-server extended criss-cross network. The left (resp. middle, right) panel shows the first (resp. second, third) coordinate of proposed drift control $\theta^*(w)$ on three-dimensional state space $\mathcal{W}$. \label{fig:CC_EWF_solu_drift}}
{}
\end{figure}

For the workload problem in the $10$-server case (i.e., when $p = 10$), a direct visualization of the proposed policy is not feasible due to the high dimensionality of the state space. Since both the proposed singular and drift controls are determined by the gradient of the value function (i.e., $G_{w_2}(\cdot)$), we instead examine two-dimensional slices of the $10$-dimensional gradient. These slices exhibit qualitative patterns similar to those in the $3$-server case; see Appendix \ref{appendix: test_problem_NN_training} for the gradient slices for both cases.

\subsubsection{Tuning the Parameters of Our Proposed Policy for the Original Queueing Network.\label{sec: simulation_our_proposed_policy_original}}

Recall that our proposed policy in the original system (see Section \ref{sec: proposed_policy}) involves tuning parameters $\rho,\epsilon,s,\tilde{\tau},\tilde{\lambda},\tilde{c}$. To choose them, we conduct simulation-based grid search in the original system; see search procedure in Section \ref{sec: policy_parameters} and implementation details in Appendix \ref{appendix: prelimit_implementation_details}. We set $\tilde{\tau}=\text{abs}$ in the Rybko-Stolyar problem (see Section \ref{sec: test_problem_RS_network}) to avoid redundant computation, because two scheduling form choices $\tilde{\tau}\in\{\text{abs},\text{rel}\}$ for that test problem are equivalent subject to a parameter transformation in $\tilde{\lambda}$ and $\tilde{c}$.

Table \ref{tab:optimal_parameter_proposed} summarizes the tuning parameters for our proposed policy. Having $\rho \approx 1$ in the Rybko-Stolyar problem means directly replacing throughput-time bound with queue lengths bound would lead to delay probabilities around $\varepsilon = 1\%$. Having $\rho < 1$ in two extended criss-cross problems suggests that a tighter queue-length bound thus a stricter outsourcing policy is required to achieve delay-probability constraint. In our grid search we set $\epsilon^n= \epsilon\cdot n^{1/4}$, i.e., $\kappa = 1/4$ in Equation (\ref{epsilon_n_kappa}), and search over $\epsilon$. The reported value of $\epsilon$ with order $1$ implies a small band width $\epsilon^n$ compared to (unscaled) state space $\mathcal{W}_\rho^n$ with order $\sqrt{n}$. The safety stock is $s = 1$ for the Rybko-Stolyar problem and $s = 0$ for both extended criss-cross problems. 

\begin{table}[htbp]
\TABLE
{Tuning parameters for our proposed policy in the original system.\label{tab:optimal_parameter_proposed}}
{\begin{small}
\setlength{\tabcolsep}{0.3em}
\renewcommand{\arraystretch}{0.9}
\begin{tabular}{lcccccccc}
\toprule
Test problem & & Bound scale $\rho$ & & Band width $\epsilon$ & & Safety stock $s$ & & Scheduling ($\tilde{\tau}$, $\tilde{\lambda}$, $\tilde{c}$) \\
\midrule
Rybko-Stolyar   
& & $1.02$ & & $1$ & & $1$          
& & $(\text{abs},20, 13)$   \\
$3$-server extended criss-cross  
& & $0.91$ & & $1$ & & $0$          
& & $(\text{rel},80, 0.87)$   \\
$10$-server extended criss-cross 
& & $0.93$ & & $1.5$ & & $0$          
& & $(\text{abs},6, 6)$   \\
\bottomrule
\end{tabular}
\end{small}}
{}
\end{table}

\subsection{Comparison with Benchmark Policies in Prelimit Simulation\label{sec: simulation_results}}

Next, we examine the performance of our proposed policy by comparing it with the $12$ benchmark policies (see brief description in Section \ref{sec: benchmark_main_text} and details in Appendix \ref{sec: benchmark}) using Monte Carlo simulation in the original system. For our proposed policy, the tuning parameters were given in Table \ref{tab:optimal_parameter_proposed}. For benchmark policies, we conduct similar simulation-based grid search to tune their parameters in the original system; see search procedure in Appendix \ref{sec: benchmark_combination} and implementation details in Appendix \ref{appendix: prelimit_implementation_details}. The resulting parameter values for each benchmark policy are given in Appendix \ref{appendix: test_problem_optimal_parameters}.

We report the cost for each policy in Table \ref{tab: prelimit_simulation}, with standard errors shown in parentheses. Three panels correspond to results in three test problems. In each panel, the cost of proposed policy is given in first row with bold font, and each benchmark is represented by its outsourcing and discretionary idling, pricing, and scheduling components. We highlight the cost of the best benchmark (i.e., the benchmark with lowest cost) for each test problem by underlining it. For the delay performance of each policy included in Table \ref{tab: prelimit_simulation}, see Appendix \ref{appendix: test_problem_delay_performance} for further details. 

\begin{table}[t]
\TABLE
{Simulation results in the original system. Least-relative-slack (queues) denotes least-relative-slack policy based on queue lengths. Each policy uses the policy parameters given in Appendix \ref{appendix: test_problem_optimal_parameters} that are obtained by a simulation-based search.\label{tab: prelimit_simulation}}
{\begin{small}
\setlength{\tabcolsep}{0.5em}
\renewcommand{\arraystretch}{0.9}
\begin{tabular}{lccccccc}
\toprule
\multicolumn{8}{l}{\textit{Panel A: Rybko-Stolyar network}} \\
\midrule
Proposed policy & & & \multicolumn{5}{c}{$\mathbf{2.93\ (0.03)}$} \\
\midrule
 & & & 
 \multicolumn{2}{c}{Single-pushing} & & \multicolumn{2}{c}{Threshold-based} \\
\cmidrule(lr){4-5} \cmidrule(lr){7-8}
Benchmark policies & & & Linear pricing & High-low pricing
           & & Linear pricing & High-low pricing \\
\midrule
Proposed scheduling   & & & \uline{$3.24\ (0.04)$}  & $3.25\ (0.04)$  & & $3.64\ (0.04)$  & $3.59\ (0.04)$  \\
Least-relative-slack  & & & $4.35\ (0.04)$  & $4.42\ (0.04)$  & & $5.92\ (0.05)$  & $6.01\ (0.05)$  \\
Least-relative-slack (queues)
    & & & $4.03\ (0.04)$  & $4.15\ (0.04)$  & & $5.66\ (0.04)$  & $5.80\ (0.05)$  \\
\midrule
\multicolumn{8}{l}{\textit{Panel B: $3$-server extended criss-cross network}} \\
\midrule
Proposed policy & & & \multicolumn{5}{c}{$\mathbf{13.37\ (0.14)}$} \\
\midrule
 & & & 
 \multicolumn{2}{c}{Single-pushing} & & \multicolumn{2}{c}{Threshold-based} \\
\cmidrule(lr){4-5} \cmidrule(lr){7-8}
Benchmark policies & & & Linear pricing & High-low pricing
           & & Linear pricing & High-low pricing \\
\midrule
Proposed scheduling   & & & \uline{$14.60\ (0.15)$} & $14.84\ (0.16)$ & & $15.52\ (0.15)$ & $15.63\ (0.15)$ \\
Least-relative-slack  & & & $16.60\ (0.16)$ & $16.73\ (0.16)$ & & $17.77\ (0.15)$ 
& $18.25\ (0.15)$ \\
Least-relative-slack (queues)
& & & $23.59\ (0.17)$ & $24.20\ (0.17)$ & & $20.44\ (0.16)$             
& $20.71\ (0.17)$ \\
\midrule
\multicolumn{8}{l}{\textit{Panel C: $10$-server extended criss-cross network}} \\
\midrule
Proposed policy & & & \multicolumn{5}{c}{$\mathbf{17.37\ (0.45)}$} \\
\midrule
 & & & 
 \multicolumn{2}{c}{Single-pushing} & & \multicolumn{2}{c}{Threshold-based} \\
\cmidrule(lr){4-5} \cmidrule(lr){7-8}
Benchmark policies & & & Linear pricing & High-low pricing
           & & Linear pricing & High-low pricing \\
\midrule
Proposed scheduling   
& & & $18.93\ (0.46)$ & \uline{$18.92\ (0.44)$} & & $23.11\ (0.56)$ & $25.80\ (0.58)$ \\
Least-relative-slack  
& & & $28.43\ (0.62)$ & $28.49\ (0.58)$ & & $36.20\ (0.54)$ & $37.35\ (0.59)$ \\
Least-relative-slack (queues)
 & & & $43.97\ (0.49)$ & $46.28\ (0.62)$ & & $45.85\ (0.62)$ & $48.46\ (0.70)$ \\
\bottomrule
\end{tabular}
\end{small}}
{}
\end{table}

Table \ref{tab: prelimit_simulation} shows that our proposed policy significantly outperforms all $12$ benchmark policies. Among the benchmark policies, the one that combines single-pushing, linear pricing, and the proposed scheduling policy performs the best in the Rybko-Stolyar and $3$-server test problems, whereas the one that combines single-pushing, high-low pricing, and the proposed scheduling policy performs the best in the $10$-server test problem. In the Rybko-Stolyar test problem (see Panel A), our proposed policy has cost $2.93$ and reduces the cost by $9.6\%$ in comparison with the best benchmark with cost $3.24$. For the $3$-server test problem (see Panel B), the best benchmark has cost of $14.60$ whereas our proposed policy has a cost of $13.37$, an $8.4\%$ improvement. Similarly, for the $10$-server test problem (see Panel C), the best benchmark has a cost of $18.92$ whereas that for our proposed policy is $17.37$, an $8.2\%$ improvement.

Additionally, the comparison of different benchmark policies presented in Table \ref{tab: prelimit_simulation} shows that the proposed scheduling policy (inspired by the maximum pressure policy) is superior to the other two scheduling rules across three test problems. 

As an illustration, we present delay time distribution under different policies. Across our test problems, we find that the shape of delay time distribution mainly depends on the scheduling policy used. Therefore, we consider one representative policy for each scheduling policy. Figure \ref{fig: delay_time_RS} shows the results in the Rybko-Stolyar problem: the left panel for proposed policy with $\tilde{\tau} = \text{abs}$, the middle (resp. right) panel for the benchmark policy using single-pushing, linear pricing, and the least-relative-slack policy (resp. least-relative-slack policy based on queue lengths). Figure \ref{fig: delay_time_CC3} repeats the same representative policies for $3$-server problem, but separately considers two forms of proposed policy (i.e., $\tilde{\tau} \in \{\text{abs}, \text{rel}\}$). The figure for $10$-server problem is deferred to Appendix \ref{appendix: test_problem_delay_performance}, due to its similarity to the $3$-server one.

\begin{figure}[tbp]
\FIGURE
{\includegraphics[scale=0.8]{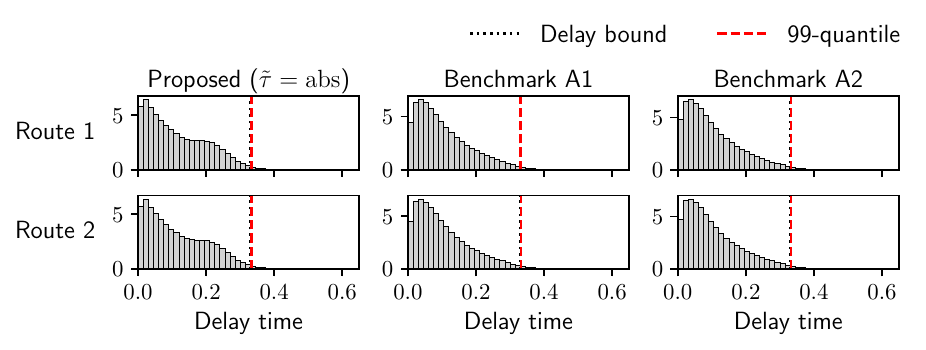}}
{Delay time histogram under the proposed policy and representative benchmark policies in the Rybko-Stolyar network. Benchmark A1 (resp. A2): the benchmark policy using single-pushing, linear pricing, and least-relative-slack policy (resp. least-relative-slack policy based on queue lengths). \label{fig: delay_time_RS}}
{}
\end{figure}

\begin{figure}[htbp]
\FIGURE
{\includegraphics[scale=0.8]{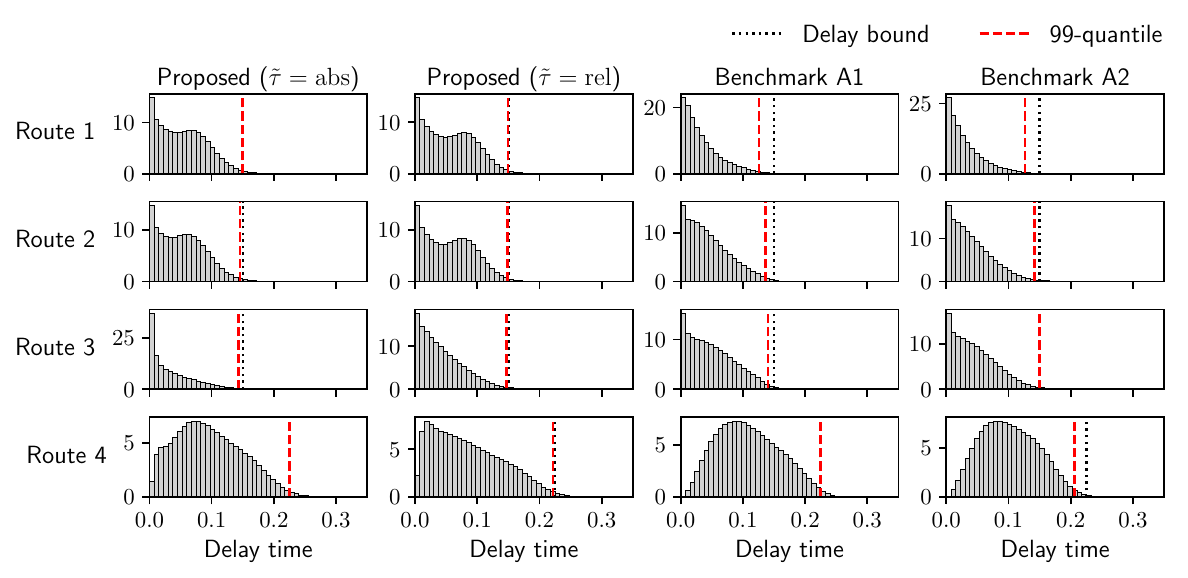}}
{Delay time histogram under the proposed policy and representative benchmark policies in the 3-server extended criss-cross network. Benchmark A1 (resp. A2): the benchmark policy using single-pushing, linear pricing, and least-relative-slack policy (resp. least-relative-slack policy based on queue lengths).\label{fig: delay_time_CC3}}
{}
\end{figure}

In each panel of Figures \ref{fig: delay_time_RS}--\ref{fig: delay_time_CC3}, the dotted line and the dashed red line mark the delay time bound and the 99-quantile of the delay time distribution respectively. The 99-quantile lying to the left of the bound indicates that the delay probability constraint $\varepsilon = 1\%$ is satisfied; a larger gap between the two lines indicates that the constraint is less binding on that route. Three observations stand out. First, the delay-probability constraints are overall more binding for routes in our proposed policy than those in two benchmark policies. Second, for jobs not violating delay-time bound, jobs of certain route (e.g., both routes in Figure \ref{fig: delay_time_RS} and first two routes in Figure \ref{fig: delay_time_CC3}) under our proposed policy spend on average longer time in system than those under two benchmark policies. Lastly, two forms of our proposed scheduling policy, i.e., $\tilde{\tau}\in \{\text{abs},\text{rel}\}$, have different behaviors in balancing last buffers of different routes; see routes $3$ and $4$ in Figure \ref{fig: delay_time_CC3}.

\section{Concluding Remarks\label{sec: conclusion}}

We study the dynamic control problem of multiclass queueing networks with throughput-time constraints, and develop an effective policy in high-dimensional settings. Our approach approximates the original problem by a Brownian control problem, reduces it to a lower-dimensional workload problem, and solves the resulting HJB equation using neural networks. This reduction provides both computational tractability and structural interpretability. Computationally, we are able to extend existing computational methods in \cite{Ata_Harrison_Si_2024_drift, Ata_Harrison_Si_2024_singular} to the bounded convex polyhedral state spaces arising in our setting. 

Despite its tractability and interpretability, translating the solution of the workload problem into implementable policy in the original system requires an intricate policy design. First, we select the proposed singular control only among those pushing towards interior of the state space.
Second, we introduce necessary modifications in the prelimit policy, such as bands and safety stocks, which allow the workload solution to be implemented in the original queueing network. Third, we propose a scheduling policy that builds on maximum-pressure policy and is tailored for our setting with throughput-time constraints. Overall, our policy consistently outperforms the benchmark policies by $5\%$--$10\%$ in our test problems.

In extending the computational method in \cite{Ata_Harrison_Si_2024_drift, Ata_Harrison_Si_2024_singular} to our bounded-state-space setting, we observe an additional tradeoff in reference-policy design. Besides adequately covering regions frequently visited under the optimal policy, the reference process should also generate sufficiently many samples near the boundary, particularly the upper faces. Such coverage is important for two reasons. First, the boundary conditions play a major role in shaping the gradient function throughout the state space. Second, costly singular controls are activated around the boundary, making accurate approximation there essential for control decisions. Our numerical results illustrates this tradeoff: moderately positive reference drifts lead to successful training, whereas the proposed drift controls are mostly negative. Intuitively, the former generates more visits to the upper faces, while the latter tends to avoid such costly visits.

\begingroup
\parindent 0pt \parskip 0.0ex
\def\enotesize{\normalsize}
\endgroup

\bibliographystyle{informs2014}
\bibliography{RBM_TPT}

\newpage
\appendix
\setcounter{equation}{0}
\renewcommand{\theequation}{\Alph{section}.\arabic{equation}}

\section{Notations in Brownian Control Problem and Workload Problem \label{appendix: notation_derivation_BCP}}

\subsection{Notations in Approximating Brownian Control
Problem\label{appendix: notation_BCP}}

In Equations (\ref{BCP_Z_k_1})--(\ref{BCP_Z_k_l}) defining (scaled) queue length process $Z$, $X_{k}=(X_{k,l})$ is an $n_{k}$-dimensional Brownian motion
with drift vector $\varphi_{k}$ and covariance matrix
$\Sigma_{k}$, for which $\varphi
_{k}=(\varphi_{k,l})$ is a $n_k$-dimensional vector with
\begin{align}
\varphi_{k,1}  &  =-\eta_{k,1}x_{k,1}^{\ast},\label{BCP_BM_drift_1}\\
\varphi_{k,l}  &  =\eta_{k,l-1}x_{k,l-1}^{\ast}-\eta_{k,l}%
x_{k,l}^{\ast},\quad l=2,\ldots,n_{k}, \label{BCP_BM_drift_l}%
\end{align}
and $\Sigma_{k}=(\Sigma_{k,l_1,l_2})$ is a $n_k\times n_k$ dimensional matrix with entry
\begin{equation}
\Sigma_{k,l_1,l_2}=\left\{
\begin{array}
[c]{ll}%
\lambda_{k}^{\ast}\big(1+v_{k,1}^{2}\big), & \text{if }l_1=l_2=1,\\
\lambda_{k}^{\ast}\big(v_{k,l_1-1}^{2}+v_{k,l_1}^{2}\big), & \text{if }l_1=l_2>1,\\
-\lambda_{k}^{\ast}v_{k,l_2}^{2}, & \text{if }l_2=l_1-1,\\
-\lambda_{k}^{\ast}v_{k,l_1}^{2}, & \text{if }l_2=l_1+1,\\
0, & \text{otherwise,}%
\end{array}
\right.\quad \text{for } l_1,l_2=1,\ldots,n_{k}.  \label{BCP_BM_variance}%
\end{equation}
We then use $\varphi$ and $\Sigma$ to represent drift and variance of $m$-dimensional Brownian motion $X = (X_{k,l})$:
\begin{align}
\varphi = \begin{pmatrix}
    \varphi_{1}\\
    \vdots \\ 
    \varphi_{K}
\end{pmatrix},\quad 
\Sigma = \begin{pmatrix}
    \Sigma_{1} & &\\
    & \ddots  &\\
    & & \Sigma_{K}
\end{pmatrix},\label{X_BM_drift_variance}
\end{align}
where $\varphi_k$ and $\Sigma_k$ ($k = 1,\dots,K$) are defined in (\ref{BCP_BM_drift_1})--(\ref{BCP_BM_variance}).

\subsection{Notations in Workload Problem\label{appendix: notation_EWF}}

In Equation (\ref{workload_process_diffusion}) defining workload process $W$, $\xi = (\xi_i)$ is a $p$-dimensional Brownian motion with drift $\gamma$ and variance $\Gamma$ defined as
\begin{equation}
\gamma=M\varphi,\quad \Gamma=M\Sigma M^{\top}.
\label{gamma_Gamma_def}%
\end{equation}
Here, $\varphi$ and $\Sigma$ are defined in (\ref{X_BM_drift_variance}). The matrix $M = (M_{i,(k,l)})$ is a $p\times m$ matrix with each
component defined as%
\begin{equation}
M_{i,(k,l)}=\sum_{l':(k,l')\in \mathcal{C}_{i},l'\geq l}m_{k,l'},\label{M_matrix}%
\end{equation}
for $i=1,\ldots,p$, $k=1,\ldots,K$, and $l=1,\ldots,n_{k}$. Indeed, $M$ is the
matrix such that the definition (\ref{workload_process_diffusion}) of workload process $W$ and the definition (\ref{workload_BM}) of Brownian motion $\xi$ can be written into the matrix forms:
\begin{equation}
W(t)=MZ(t), \quad \xi(t)=MX(t). \label{W_Z_matrix_B_X}%
\end{equation}

\section{Derivations and Proofs}\label{appendix: proofs}

\subsection{Derivations of Approximating Brownian Control
Problem\label{appendix: derivation_BCP}}

In what follows, we formally derive the Brownian approximation of (scaled) queue
length process and cumulative cost process in the heavy-traffic limit.

\subsubsection{(Scaled) Queue Length Process.}

For (scaled) queue length process $Z^{n}$, we separately consider $Z^n_{k,1}(t)$ and $Z^n_{k,l}(t)$ ($l = 2,\dots,n_k$). First, regarding $Z_{k,1}^{n}(t)$ defined by Equations (\ref{Q_k_1}) and (\ref{Z_n_U_n}), i.e.,
\begin{align}
    Z_{k,1}^{n}(t) = Z_{k,1}^{n}(0) + \frac{1}{\sqrt{n}} \left(N_{k}\left(  \int_{0}^{t}\lambda^n_{k}(s)ds\right)
-S^n_{k,1}(T^n_{k,1}(t))-O^n_{k}(t)\right),
\end{align}
we decompose it into
\begin{align}
Z_{k,1}^{n}(t)  &  
 =Z_{k,1}^{n}(0) +\frac{1}{\sqrt{n}}\left(  N_{k}\left(  \int_{0}^{t}\lambda_{k}%
^{n}(s)ds\right)  -\int_{0}^{t}\lambda_{k}^{n}(s)ds\right)
  +\frac{1}{\sqrt{n}}\left(  \int_{0}^{t}\lambda_{k}^{n}(s)ds-\mu_{k,1}%
^{n}x_{k,1}^{\ast}t\right) \nonumber\\
&  +\frac{1}{\sqrt{n}}\left(  \mu_{k,1}^{n}x_{k,1}^{\ast}t-\mu_{k,1}%
^{n}T_{k,1}^{n}(t)\right)  -\frac{1}{\sqrt{n}}\left(  S_{k,1}^{n}(T_{k,1}^{n}(t))-\mu_{k,1}%
^{n}T_{k,1}^{n}(t)\right)  -\frac{1}{\sqrt{n}}O_{k}^{n}(t). \label{Z_n_k_1}%
\end{align}
In what follows, we separately derive the limit of five terms (except initial value $Z_{k,1}^{n}(0)$) on right-hand side of (\ref{Z_n_k_1}), after which we will combine them and obtain limit of $Z_{k,1}^n$. 
\begin{itemize}
    \item[1)] For the first term, we have
\begin{equation}
\frac{1}{\sqrt{n}}\left(  N_{k}\left(  \int_{0}^{t}\lambda_{k}^{n}%
(s)ds\right)  -\int_{0}^{t}\lambda_{k}^{n}(s)ds\right)  =\sqrt{\lambda
_{k}^{\ast}}B_{k}(t)+o(1), \label{Z_n_k_1_part1_derivation}%
\end{equation}
where $B_{k}$ is one-dimensional standard Brownian motion. It follows from the functional strong
approximations 
\begin{equation}
N(nt)=nt+\sqrt{n}B(t)+o(\sqrt{n})\text{ a.s.},\nonumber
\end{equation}
and assumption (\ref{lambda_assumption}) on demand rate $\lambda^n(t)$ that 
$$\lambda_k^n(t) = n\lambda_k^* + \mathcal{O}(\sqrt{n}).$$
\item[2)] For the second term, it comes from assumptions (\ref{assumption_Lambda_n_mu_n}) and (\ref{lambda_assumption}) on service rate $\mu^n$ and demand rate $\lambda^n(t)$, and definition $x^*_{k,l} = \lambda_k^*/\mu_{k,l}$ that 
\begin{align}
\frac{1}{\sqrt{n}}\left(  \int_{0}^{t}\lambda_{k}^{n}(s)ds-\mu_{k,1}%
^{n}x_{k,1}^{\ast}t\right) &  =\frac{1}{\sqrt{n}}\left(  n\lambda_{k}^{\ast}t+\sqrt{n}\int_{0}^{t}%
\zeta_{k}(s)ds-(n\mu_{k,1}+\sqrt{n}\eta_{k,1})x_{k,1}^{\ast}t\right)\nonumber\\
&  =\int_{0}^{t}\zeta_{k}(s)ds-\eta_{k,1}x_{k,1}^{\ast}t.
\label{Z_n_k_1_part2_derivation}%
\end{align}
\item[3)] For the third term, definition (\ref{scaled_scheduling_process}) of centered and scaled allocation process $Y^n(t)$ and assumption (\ref{assumption_Lambda_n_mu_n}) on service rate $\mu^n$ give that
\begin{equation}
\frac{1}{\sqrt{n}}\left(  \mu_{k,1}^{n}x_{k,1}^{\ast}t-\mu_{k,1}%
^{n}T_{k,1}^{n}(t)\right)  = \frac{1}{n} \mu_{k,1}^n Y_{k,1}^n(t) = \mu_{k,1}Y_{k,1}^n(t) + o(1). \label{Z_n_k_1_part3_derivation}%
\end{equation}
\item[4)] For the fourth term, as $x_{k,1}^*$ serves as a first-order approximation to $T_{k,1}^n(t)$, we have
\begin{equation}
\frac{1}{\sqrt{n}}\left(  S_{k,1}^{n}(T_{k,1}^{n}(t))-\mu_{k,1}%
^{n}T_{k,1}^{n}(t)\right)   \sim \frac{1}{\sqrt{n}}\left(  S_{k,1}%
^{n}(x_{k,1}^{\ast}t)-\mu_{k,1}^{n}x_{k,1}^{\ast}t\right) \Rightarrow  \operatorname*{BM}\left(0,\lambda_{k}^{\ast}v_{k,1}^{2}\right).
\label{Z_n_k_1_part4_derivation}%
\end{equation}
\item[5)] For the fifth term, Equation (\ref{Z_n_U_n}) leads to 
\begin{equation}
\frac{1}{\sqrt{n}}O_{k}^{n}(t) = U_k^n(t).
\label{Z_n_k_1_part5_derivation}%
\end{equation}
\end{itemize}
Combining Equations (\ref{Z_n_k_1_part1_derivation})--(\ref{Z_n_k_1_part5_derivation}) and let $Y^n \Rightarrow Y$, $U^n \Rightarrow U$, we have $Z_{k,1}^{n}  \Rightarrow Z_{k,1}$, where $Z_{k,1}$ is as defined in (\ref{BCP_Z_k_1}) with $X_{k,1}(t)$ defined by
\begin{equation}
    X_{k,1}(t) = \sqrt{\lambda_{k}^{\ast}}B_{k}(t) -\eta_{k,1}x_{k,1}^{\ast}t - \operatorname*{BM}\left(0,\lambda_{k}^{\ast}v_{k,1}^{2}\right).\label{X_k_1_appendix}
\end{equation}
Similarly, we can decompose $Z_{k,l}^{n}(t)$ ($l = 2,\dots,n_k$) as follows based on Equations (\ref{Q_k_l}) and (\ref{Z_n_U_n}):
\begin{align}
Z_{k,l}^{n}(t)  &  =Z_{k,l}^{n}(0) + \frac{1}{\sqrt{n}}\left(  S_{k,l-1}^{n}(T_{k,l-1}^{n}(t))-\mu
_{k,l-1}^{n}T_{k,l-1}^{n}(t)\right)  -\frac{1}{\sqrt{n}}\left(  S_{k,l}^{n}(T_{k,l}^{n}(t))-\mu_{k,l}%
^{n}T_{k,l}^{n}(t)\right) \nonumber \\
&  -\frac{1}{\sqrt{n}}\left(  \mu_{k,l-1}^{n}x_{k,l-1}^{\ast}%
t-\mu_{k,l-1}^{n}T_{k,l-1}^{n}(t)\right)  +\frac{1}{\sqrt{n}}\left(  \mu_{k,l}^{n}x_{k,l}^{\ast}t-\mu_{k,l}%
^{n}T_{k,l}^{n}(t)\right) \nonumber \\
&+\frac{1}{\sqrt{n}}\left(  \mu_{k,l-1}^{n}x_{k,l-1}^{\ast}t-\mu
_{k,l}^{n}x_{k,l}^{\ast}t\right)  , \text{ for } l = 2,\dots,n_k.  \label{Z_n_k_l}%
\end{align}
Taking each term on right-hand side of Equation (\ref{Z_n_k_l}) to limit, we have $Z_{k,l}^n \Rightarrow Z_{k,l}$, where 
\begin{align}
Z_{k,l}(t) &=  Z_{k,l}(0) + \operatorname*{BM}\left(0,\lambda_{k}^{\ast
}v_{k,l-1}^{2}\right)-\operatorname*{BM}\left(0,\lambda_{k}^{\ast}v_{k,l}^{2}%
\right)-\mu_{k,l-1}Y_{k,l-1}(t)+\mu_{k,l}Y_{k,l}(t)\nonumber\\
&  +(\eta_{k,l-1}x_{k,l-1}^{\ast}t-\eta_{k,l}x_{k,l}^{\ast}t).
\label{Z_n_k_l_limit}%
\end{align}
Here, five terms on right hand side of (\ref{Z_n_k_l_limit}) are limit of 
five terms in  (\ref{Z_n_k_l}). Let 
\begin{align}
    X_{k,l}(t) = \operatorname*{BM}\left(0,\lambda_{k}^{\ast
}v_{k,l-1}^{2}\right)-\operatorname*{BM}\left(0,\lambda_{k}^{\ast}v_{k,l}^{2}%
\right) +(\eta_{k,l-1}x_{k,l-1}^{\ast}t-\eta_{k,l}x_{k,l}^{\ast}t),\label{X_k_l_appendix}
\end{align}
we arrive at definition (\ref{BCP_Z_k_l}) for $Z_{k,l}(t)$. 

Lastly, based on expressions (\ref{X_k_1_appendix}) and (\ref{X_k_l_appendix}) of $X_{k,l}(t)$, $l = 1,\dots,n_k$, where $B_k(t)$ and $ \operatorname*{BM}\big(0,\lambda_{k}^{\ast}v_{k,l}^{2}%
\big)$, $l = 1,\dots,n_k$ are $n_k+1$ mutually independent Brownian motions, we obtain drift $\varphi_k$ and variance $\Sigma_k$ of $n_k$-dimensional Brownian motion $X_k$ given in (\ref{BCP_BM_drift_1})--(\ref{BCP_BM_variance}).

\subsubsection{Cumulative Cost Process.}

The cumulative cost process defined in (\ref{cumulative_cost_n}) can be
written as%
\begin{equation}
C^{n}(t)=n\Pi(\lambda^{\ast})t-\left(  \int_{0}^{t}\Pi^{n}(\lambda
^{n}(s))ds-\nu^{n}\cdot O^{n}(t)\right)  . \label{xi_n_t}%
\end{equation}
The assumption (\ref{lambda_assumption}) on demand rate $\lambda^{n}(t)$ and the property
(\ref{property_revenue}) of profit rate function $\Pi^{n}$ imply that
\begin{align}
\Pi^{n}(\lambda^{n}(t))    =\Pi^{n}(n\lambda^{\ast}+\sqrt{n}\zeta
(t)) =n\Pi\left(  \lambda^{\ast}+\frac{1}{\sqrt{n}}\zeta(t)\right). \nonumber
\end{align}
And then Assumptions \ref{assumption: profit_function} and \ref{assumption: static_planning} leads to a Taylor expansion of $\Pi$ around $\lambda^*$:
\begin{align}
\Pi^{n}(\lambda^{n}(t)) = n\left(  \Pi(\lambda^{\ast})+\frac{1}{2n}\zeta(t)^{\prime}\nabla^{2} \Pi(\lambda^{\ast})\zeta(t)\right)  +o(1).\label{revenue_n_expand}
\end{align}
Plugging (\ref{revenue_n_expand}) into (\ref{xi_n_t}), and
noting that $\nu^{n}=\alpha/\sqrt{n}$ and $O^{n}/\sqrt
{n}=U^n \Rightarrow U$, we have $C^n \Rightarrow C$, where $C$ is as defined in Equation (\ref{BCP_xi}).

\subsection{Proof of Lemma \ref{lemma:c_x}}\label{appendix: proof_lemma_c_x}
The cost function $c(x)$ defined in (\ref{c_x}) can be written into matrix form:
\begin{align}
c(x)  &  =\min\left\{  \zeta^{\prime}H\zeta:\tilde{M}\zeta=x\right\}  ,\ x \in \Theta,\label{c_x_matrix}
\end{align}
where $\tilde{M}$ is as defined in (\ref{c_x_direction}). Then we use Lagrangian method with multiplier $\bar{\lambda}$:
\begin{align}
    \mathcal{L}(\zeta,\bar{\lambda}) = \zeta'H\zeta - 2\bar{\lambda}'(\tilde{M}\zeta - x),
\end{align}
and have first-order conditions
\begin{align}
    &\nabla_{\zeta} \mathcal{L}(\zeta,\bar{\lambda}) = 2H\zeta - 2\tilde{M}'\bar{\lambda} =0,\label{FOC_1}\\
    &\nabla_{\bar{\lambda}} \mathcal{L}(\zeta,\bar{\lambda}) = -2(\tilde{M}\zeta - x) =  0. \label{FOC_2}
\end{align}
Solving (\ref{FOC_1}), we have 
\begin{align}
    \zeta = H^{-1}\tilde{M}'\bar{\lambda},\label{zeta_lambda}
\end{align}
and substituting (\ref{zeta_lambda}) into (\ref{FOC_2}) leads to the linear system
\begin{align}
    (\tilde{M}H^{-1} \tilde{M}')\bar{\lambda} = x. \label{linear_system}
\end{align}
Assumption \ref{assumption: profit_function} and Equation (\ref{Hessian}) imply that matrix $H$ is positive definite, thus matrix $\tilde{M}H^{-1} \tilde{M}'$ is positive semidefinite but not necessarily positive definite (e.g., when $p>K$). Thus, we solve Equation (\ref{linear_system}) by
\begin{align}
    \bar{\lambda} = (\tilde{M}H^{-1} \tilde{M}')^{\dagger} x,\label{lambda}
\end{align}
where $\dagger$ denotes pseudo-inverse. When matrix $\tilde{M}H^{-1} \tilde{M}'$ is full-rank, we can replace the pseudo-inverse by an inverse. Combining (\ref{zeta_lambda}) and (\ref{lambda}) we have the expression of $\zeta^*(x)$ and then the expression of $c(x)$ in Lemma \ref{lemma:c_x}.

\subsection{Proof of Proposition \ref{proposition: equivalent_BCP_EWF}}\label{appendix: proof_equivalent_BCP_EWF}

The proof follows from Proposition EC.1 of \cite{Ata_Barjesteh_2023}, Proposition 3 of \cite{Rubino_Ata_2009}, and Propositions 3--4 and Theorem 2 of \cite{Harrison_Van_Mieghem_1997}. 

First, suppose that $(U,Y,\zeta)$ is an admissible control for Brownian control problem (\ref{BCP}) with state process $Z$, then we define workload process $W(t) = MZ(t)$ for $t\geq 0$, drift control process $\theta(t) = \tilde{M}\zeta(t)$ for $t\geq 0$, and idling process $L$ by Equation (\ref{BCP_I_i}). $(L,U,\theta)$ is an admissible policy for the workload problem (\ref{EWF}) with state process $W$, and its cost is less than or equal to that of the policy $(U,Y,\zeta)$ for the Brownian control problem (\ref{BCP}).

Next, suppose that $(L,U,\theta)$ is an admissible control for the workload problem (\ref{EWF}) with state process $W$. By the definition of the workload state space $\mathcal{W}$ (see Equation (\ref{W_state_space})), for each $W(t)\in \mathcal{W}$ there exists at least one scaled queue-length vector $Z(t)\in \mathcal{Z}$ such that $MZ(t) = W(t)$. Also let drift control $\zeta(t)=H^{-1}\tilde{M}^{\prime}(\tilde{M}H^{-1}\tilde{M}^{\prime})^{\dagger}\theta(t)$ for $t\geq 0$, and the centered and scaled allocation process $Y$ given by Equations (\ref{BCP_Z_k_1})--(\ref{BCP_Z_k_l}). $(U,Y,\zeta)$ is an admissible control for Brownian control problem (\ref{BCP}); and two policies have the same cost.

\subsection{Proof of Proposition \ref{proposition: equivalent_SDE}}\label{Appendix: proof: equivalent_SDE}

Applying It\^{o}'s formula to $e^{-\tilde{r} t} \tilde{V}(\tilde{W}(t))$, we have 
\begin{align}
    e^{-\tilde{r} T} \tilde{V} (\tilde{W}(T)) - \tilde{V} (\tilde{W}(0)) & = \int_0^T e^{-\tilde{r} t}\big((\mathcal{L} - \tilde{r}) \tilde{V} (\tilde{W}(t)) + \nabla \tilde{V} (\tilde{W}(t)) \cdot  \tilde{\theta}  + \tilde{G}^\top \nabla \tilde{V} (\tilde{W}(t)) \cdot \tilde{\psi}\big)dt \nonumber \\
    &+ \int_0^T e^{-\tilde{r} t}\nabla \tilde{V} (\tilde{W}(t)) \cdot dB(t) + \int_0^T e^{-\tilde{r} t} \mathcal{D} \tilde{V} (\tilde{W}(t)) \cdot d\tilde{\Upsilon}(t).\label{proof_equivalent_SDE_Ito}
\end{align}
Using boundary condition (\ref{HJB_BC}) and complementarity condition (\ref{backup_face_assign}), we have
\begin{align}
    \int_0^T e^{-\tilde{r} t} \mathcal{D} \tilde{V} (\tilde{W}(t)) \cdot d\tilde{\Upsilon}(t)  =  \int_0^T e^{-\tilde{r} t} (-\hat{\alpha}) \cdot d\tilde{\Upsilon}(t).
    \label{proof_equivalent_SDE_1}
\end{align}
Using HJB equation (\ref{HJB}), we have
\begin{align}
    &\int_0^T e^{-\tilde{r} t}\big((\mathcal{L} - \tilde{r}) \tilde{V} (\tilde{W}(t)) +  \nabla \tilde{V} (\tilde{W}(t)) \cdot \tilde{\theta} + \tilde{G}^\top \nabla\tilde{V} (\tilde{W}(t)) \cdot \tilde{\psi}\big)dt \nonumber \\
    &= -\int_0^T e^{-\tilde{r} t} F(\nabla \tilde{V}(\tilde{W}(t)))dt,\label{proof_equivalent_SDE_2}
\end{align}
where function $F$ is defined in (\ref{F_function}). Plugging Equations (\ref{proof_equivalent_SDE_1})--(\ref{proof_equivalent_SDE_2}) into (\ref{proof_equivalent_SDE_Ito}), we arrive at the identity (\ref{reference_SDE}).

\subsection{Derivation of Dynamic Pricing Policy (\ref{p_n_Taylor}) in Brownian Control Problem\label{appendix: derivation_policy_pricing}}

Recall the assumption (\ref{lambda_assumption}) on arrival rate $\lambda^n$ that
\begin{align}
    \lambda^{n}(t)=n\lambda^{\ast}+\sqrt{n}\zeta(t),\  t\geq 0.
\end{align}
Given heavy-traffic assumption (\ref{assumption_Lambda_n_mu_n}), we have
\begin{align}
    p^n(t) = (\Lambda^n)^{-1}\left(\lambda^n(t)\right) = \Lambda^{-1}\left(\lambda^* + \frac{\zeta(t)}{\sqrt{n}} \right).
\end{align}
Then its first-order Taylor expansion gives
\begin{align}
    p^n(t) = \Lambda^{-1}(\lambda^*) + \frac{\nabla \Lambda^{-1}(\lambda^*)\zeta(t)}{\sqrt{n}} + o\left(\frac{1}{\sqrt{n}}\right).\label{p_n_Taylor_appendix}
\end{align}

\section{Additional Details of Computational Method}

\subsection{Existence and Uniqueness of SRBMs in our Setting for the Reference Process\label{appendix: SRBM}}

To facilitate the analysis, we introduce the auxiliary definitions as follows. For a subset $\mathcal{K}\subset \mathcal{J}$, we define the intersection of faces $\{F_j,\ j \in \mathcal{K}\}$ as 
\begin{align}
    F_{\mathcal{K}} = \{w: \ w \in F_j \text{ for all } j \in \mathcal{K}\}, \ \mathcal{K}\subset \mathcal{J}.\label{F_K}
\end{align}
A subset $\mathcal{K}\subset \mathcal{J}$ is maximal if $\mathcal{K}\neq \emptyset$, $F_{\mathcal{K}}\neq \emptyset$, and $F_{\mathcal{K}} \neq F_{\tilde{\mathcal{K}}}$ for any $\tilde{\mathcal{K}}\supsetneq \mathcal{K}$. 

To verify that one can choose the pushing directions of the backup singular control so that the reference process is well defined for our problem, we utilize the theory of existence and uniqueness for semimartingale reflecting Brownian motions (SRBMs) in convex polyhedrons developed by \cite{Dai_Williams_1996_existence}; also see \cite{Reiman_Williams_1988}. Assumption 1.1 of \cite{Dai_Williams_1996_existence} provides sufficient conditions for the existence and uniqueness of SRBMs in convex polyhedrons, which subsumes our setting for the reference process. Under that assumption, Theorem 1.3 of \cite{Dai_Williams_1996_existence} establishes the existence and uniqueness of the SRBM. We state their Assumption 1.1 below using our notation for completeness.

\begin{assumption}[Assumption 1.1 in \cite{Dai_Williams_1996_existence}]\label{assumption:dai_williams} For each maximal $\mathcal{K} \subset \mathcal{J}$,
\begin{itemize}
    \item[(i)] there is a positive linear combination $\hat{g} = \sum_{j\in \mathcal{K}} u_j\hat{g}_j$ ($u_j>0 \ \forall j \in \mathcal{K}$) of the $\{\hat{g}_j,\  j \in \mathcal{K}\}$ such that $\sigma_j\cdot \hat{g}   > 0$ for all $j \in \mathcal{K}$; 
    \item[(ii)] there is a positive linear combination $\sigma = \sum_{j\in \mathcal{K}} y_j\sigma_j$ ($y_j>0 \ \forall j \in \mathcal{K}$) of the $\{\sigma_j, \ j \in \mathcal{K}\}$ such that $\sigma\cdot \hat{g}_j > 0$ for all $j \in \mathcal{K}$. 
\end{itemize}
\end{assumption}

Proposition \ref{proposition: existence_singular_control} below establishes that one can choose the reflection matrix $\hat{G}$ for the backup singular control problem so that Assumption \ref{assumption:dai_williams} holds. Thus, by Theorem 1.3 of \cite{Dai_Williams_1996_existence}, the reference process is well defined. As a preparation for Proposition \ref{proposition: existence_singular_control}, we first prove Lemma \ref{lemma:simplified_dai_williams_assumption} and Lemma \ref{lemma:pushing_direction_span} as follows.

\begin{lemma}\label{lemma:simplified_dai_williams_assumption}
For each maximal $\mathcal{K} \subset \mathcal{J}$, there exists a positive linear combination $g = \sum_{j\in \mathcal{K}} u_j\sigma_j$ ($u_j>0, \ \forall j \in \mathcal{K}$) of the $\{\sigma_j,\ j \in \mathcal{K}\}$ such that $\sigma_j \cdot g > 0$ for all $j\in \mathcal{K}$.
\end{lemma}

\begin{proof}[Proof of Lemma \ref{lemma:simplified_dai_williams_assumption}]\label{proof:simplified_dai_williams_assumption}
For each maximal $\mathcal{K}\subset \mathcal{J}$ we denote a cone and linear space generated by unit normal vector $\{\sigma_j,\ j\in \mathcal{K}\}$ respectively as follows:
\begin{align*}
    C = \operatorname{cone}(\sigma_j, \ j \in \mathcal{K})\quad \text{and} \quad S = \operatorname{span}(\sigma_j, \ j \in \mathcal{K}).
\end{align*}
And denote the dual cone of $C$ restricted to space $S$ as
\begin{align*}
    C^* = \{x \in S: \ x\cdot y \geq 0 \text{ for } \forall y \in C\}.
\end{align*}
Based on the notations above, now $g \in \operatorname{int} (C)$ implies that $g$ can be represented as a positive linear combination $g = \sum_{j\in \mathcal{K}} u_j\sigma_j$ ($u_j>0, \ \forall j \in \mathcal{K}$), and $g \in \operatorname{int} (C^*)$ implies that $g\cdot \sigma_j > 0$ for all $j \in \mathcal{K}$. Then any $g \in \operatorname{int} (C \cap C^*)$ satisfies the conditions.
    
We then prove that the set $\operatorname{int} (C \cap C^*)$ is not empty. By contradiction, if $\operatorname{int} (C \cap C^*)$ is an empty set, then since $C$ and $C^*$ are both convex sets, there exists a hyperplane $\{a^\top x = 0\}$ with coefficient $a \in S \backslash \{0\}$, such that 
\begin{align}
    &a^\top x \geq 0 \quad \text{for} \quad \forall x \in C,\label{hyperplane_C}\\
    &a^\top x \leq 0  \quad \text{for} \quad \forall x \in C^*.\label{hyperplane_C_ast}
\end{align} 
Equation (\ref{hyperplane_C}) implies that $a \in C^*$, then let $x = a$ in Equation (\ref{hyperplane_C_ast}), we have $a^\top a \leq 0$, which contradicts with non-zero property of coefficient vector $a$.\Halmos
\end{proof}

\begin{lemma}\label{lemma:pushing_direction_span}
For any vector $g \in \mathbb{R}^p$, there exists $a \in \mathbb{R}^{p+K}_+$ such that $g = \tilde{G}a$.
\end{lemma}

\begin{proof}[Proof of Lemma \ref{lemma:pushing_direction_span}]\label{proof:pushing_direction_span}
Recall that $\tilde{G} = (
I_p \ -\tilde{M})$. Representing columns of $\tilde{G}$ as $\tilde{g}_1,\dots,\tilde{g}_{p+K}$, we have $\tilde{g}_i = e_i$ for $i = 1,\dots,p$ (i.e., $i$th entry of $\tilde{g}_{i}$ is $1$ and others are zero). The proof resorts to showing that $\{-e_i,i=1,\dots,p\}$ (i.e., $i$th entry is $-1$ and others are zero) can be represented by nonnegative linear combinations of $\tilde{g}_1,\dots,\tilde{g}_{p+K}$, after which we can use $\{e_i,-e_i,i = 1,\dots,p\}$ to construct any vector $g \in \mathbb{R}^p$. 

To construct $-e_i$ for $i = 1,\dots,p$, we proceed as follows. For any $i = 1,\dots, p$, according to definition (\ref{c_x_direction}) of $\tilde{M}$, there exists a $k$ such that $\tilde{M}_{ik}\neq 0$. Then we construct a vector as
\begin{align}
    \tilde{g}_{-i} = (\tilde{g}_{p+k} + \sum_{j: 1\leq j \leq p, j\neq i} \tilde{g}_j \tilde{M}_{jk} ) / \tilde{M}_{ik}. \label{g_minus_i}
\end{align}
It's straightforward to check that $\tilde{g}_{-i} = -e_i$.\Halmos

\end{proof}

\begin{proposition}\label{proposition: existence_singular_control}
    There exists a reflection matrix $\hat{G}$ satisfying Assumption \ref{assumption:dai_williams} for our problem.
\end{proposition}

\begin{proof}[Proof of Proposition \ref{proposition: existence_singular_control}]\label{proof:existence_singular_control}

We construct a reflection matrix $\hat{G}$ by setting the pushing direction for each face as
\begin{align}
    \hat{g}_j = \sigma_j, \ j \in \mathcal{J},\label{construct_direction_sigma}
\end{align}
where $\sigma_j$ is the unit vector normal to face $F_j$ of state space. That is, each pushing direction is orthogonal to its associated face.

Then we prove that such choice of reflection matrix $\hat{G}$ satisfies Assumption \ref{assumption:dai_williams}, and each pushing direction $\hat{g}_j$ can be represented as nonnegative linear combinations of columns of pushing direction matrix $\tilde{G}$ (that is, the pushing direction matrix of singular control $\mathcal{U}$) in our problem. The first half is proved by noting that Assumption \ref{assumption:dai_williams} under choice (\ref{construct_direction_sigma}) can be simplified to Lemma \ref{lemma:simplified_dai_williams_assumption}. And the second half is implied by Lemma \ref{lemma:pushing_direction_span}.\Halmos
\end{proof}

Recall from Section \ref{sec: computational_method_policy}, the workload problem under the proposed policy has two additional features. First, its drift rate is state dependent. This model feature is not covered by the theory of \cite{Dai_Williams_1996_existence}. However, one can combine that theory with Girsanov's theorem to handle the case of state-dependent drifts under certain conditions. The second feature is that under the proposed policy more than one pushing direction might be used on each face $F_j$ of the state space. Such policies have been observed in other papers as the optimal singular control too, e.g., \cite{Ata_Barjesteh_Kumar_2020}. However, we are not aware of existing theory that handles the question of existence and uniqueness of the SRBM is this context. Thus, we take a formal, i.e., nonrigorous, approach in Sections \ref{sec: computational_method_SDE}--\ref{sec: computational_method_policy}, leaving the rigorous analysis of existence and uniqueness questions for future research.

\subsection{Simulation under the Reference Process\label{appendix: computational_description}}

\begin{algorithm}[t]
\begin{small}
\caption{Euler Discretization Scheme.}
\label{alg:euler_discretization}
\KwIn{Drift vectors $\tilde{\theta}, \tilde{\psi}$; drift vector $\gamma$ and covariance matrix $\Gamma$ of Brownian motion $\xi$; reflection matrix $\hat{G}$ for backup control $\Upsilon$; state space $\mathcal{W}$; time horizon $T$; step size $h$ with $N := T/h \in \mathbb{N}$; initial point $\tilde{W}(0) = w$.}
\KwOut{Discretized workload $\tilde{W}(t)$, pushing increments $\Delta\tilde{\Upsilon}(t)$, and Brownian increments $\Delta B(t)$ at $t = 0, h, \ldots, Nh$.}
\Fn{\textsc{Discretize($T$, $h$, $w$)}}{
    Construct the time partition $0 = t_0 < t_1 < \cdots < t_N = T$ with $\Delta t_n = t_{n+1}-t_n =  h$ for $n = 0,\dots,N-1$\;
    Generate i.i.d. Gaussian increments $\Delta B(kh) \sim \mathcal{N}(0, h\Gamma)$ for $k = 0, \ldots, N-1$\;
    \For{$k \gets 0$ \KwTo $N-1$}{
        $x \gets \tilde{W}(kh) + \gamma h + \Delta B(kh) + \tilde{\theta} h + \tilde{G}\tilde{\psi}h$\;
        $\tilde{W}((k+1)h), \Delta \tilde{\Upsilon}(kh) \gets $ \textsc{Skorokhod}($x$)\;
    }
    \KwRet $\{\tilde{W}(kh)\}_{k=1}^{N}$, $\{\Delta \tilde{\Upsilon}(kh)\}_{k=0}^{N-1}$, $\{\Delta B(kh)\}_{k=0}^{N-1}$\;
}
\end{small}
\end{algorithm}

Algorithm \ref{alg:euler_discretization} describes the Euler discretization scheme in the first step of our computational method. The reference policy $(\tilde{\theta},\tilde{\psi})$ specifies the drift parameters of the reference process $\tilde{W}$. Additionally, when $\tilde{W}$ reaches the boundary $\partial \mathcal{W}$ of the state space, it is ``pushed back'' to the interior of the state space. This boundary behavior is formalized using the Skorokhod map that, in turn, can be viewed as a solution of a certain linear complementarity problem. To be specific, the function \textsc{Skorokhod}($x$) (see Algorithm \ref{alg:skorokhod}) in Algorithm \ref{alg:euler_discretization} solves the Skorokhod problem given the bounded convex polyhedron state space $\mathcal{W}$, current workload $x$, and the reflection matrix $\hat{G}$. Its output is the updated workload vector $\tilde{W}((k+1)h)$ and the pushing process increment $\Delta \tilde{\Upsilon}(kh)$. 

In principle any algorithm for solving linear complementarity problem can be used in Algorithm \ref{alg:skorokhod}, however, we need it to be efficient as it is invoked repeatedly in Algorithm \ref{alg:euler_discretization} for simulating the reference process. \cite{Ata_Harrison_Si_2024_drift} provides an efficient method for their setting, i.e., the state space is the positive orthant and the reflection matrix is a Minkowski matrix. Because our reflection matrix $\hat{G}$ for the backup process need not be a Minkowski matrix in general, one needs to solve the linear complementarity problem from first principles.
For our test problems (see Section \ref{sec: test_problems}), it is worth mentioning that using the problem-specific special structure, we derive simple yet efficient implementations of the Skorokhod map; see Appendix \ref{appendix: implementation_reflection_mapping}.

\begin{algorithm}[t]
\begin{small}
\caption{Solve the Skorokhod Problem (Linear Complementarity Problem).}
\label{alg:skorokhod}
\KwIn{The state space $\mathcal{W} = \{w\in \mathbb{R}^p: \sigma_j \cdot w \geq h_j \text{ for all } j \in \mathcal{J}\}$, current workload $x$, reflection matrix $\hat{G} = (\hat{g}_1,\dots,\hat{g}_J)$.}
\KwOut{Updated workload vector $y$, pushing process increment $u$.}
\Fn{\textsc{Skorokhod($x$)}}{
    \If{$x \in \mathcal{W}$}{
    \KwRet $x$, $0_J$\;
    }
    \Else{
    Find $\{u_j\}_{j=1}^J$ and $y$ such that
    $y = x + \sum_{j=1}^J u_j \hat{g}_j$, $\sigma_j \cdot y \geq h_j$, $u_j \geq 0$, $u_j(\sigma_j \cdot y - h_j) = 0$ for all $j$\;
        $u \gets (u_1, \ldots, u_J)'$\;
        \KwRet $y$, $u$\;
    }
}
\end{small}
\end{algorithm}

\section{Detailed Definition of Benchmark Policies\label{sec: benchmark}}

\subsection{Outsourcing and Discretionary Idling Policies\label{sec: benchmark_outsourcing}}

In this subsection, we introduce two alternative benchmark policies for outsourcing and discretionary idling. We refer to these alternatives as the single-pushing-control-per-face policy and the threshold-based policy. The former assigns a single pushing control to each face of the state space of the workload process but otherwise mimics our proposed policy. Thus, unlike our proposed policy, it does not partition each face into subsets and assign a different pushing direction to each subset. The threshold-based policy outsources an arriving job of a product whenever the total number of jobs of that product in the system exceeds a specified threshold.

\subsubsection{Single Pushing Control per Face Policy.} We start by defining the singular control policy for the workload problem in single-pushing-control-per-face policy (single-pushing policy hereafter). Recall that the limiting state space $\mathcal{W}$ has a set of faces $F_j$ ($j \in \mathcal{J}$); see Equation (\ref{W_face}). We assign a single singular control to each face $F_j$ ($j \in \mathcal{J}$) of the state space by index
\begin{align}
    i^{\text{SP}}(j) \in \operatorname*{argmin}_{i:\ \tilde{g}_i \cdot \sigma_j>0}\ \frac{\tilde{\alpha}_i}{\tilde{g}_i \cdot \sigma_j}.\label{benchmark_outsourcing_single_push_index}
\end{align}
Here, $\tilde{g}_i$ (resp. $\tilde{\alpha}_i$) is the pushing direction (resp. cost rate) of singular control $\mathcal{U}_i$, and $\sigma_j$ is the inward-pointing unit normal to the face $F_j$. The policy is to apply singular control $\mathcal{U}_{i^{\text{SP}}(j)}$ at any interior point $w$ of face $F_j$. It greedily minimizes the instantaneous pushing cost in the workload problem by using the cheapest pushing control at the boundary.

It is noteworthy that Equation (\ref{benchmark_outsourcing_single_push_index}) and our proposed singular control (\ref{learned_boundary_control}) both provide a singular-control solution to limiting workload problem (at interior of faces), but differ in their objectives for selecting singular control. To be specific, Equation (\ref{benchmark_outsourcing_single_push_index}) is solely based on instantaneous cost $\tilde{\alpha}_i$, whereas our proposed one (\ref{learned_boundary_control}) further include a gradient term $G_{w_2}(w)$ characterizing long-term effect. In our computational results (see Section \ref{sec: simulation}), we identify long-term effect in our proposed policy by comparing it with single-pushing policy. Furthermore, different objectives in selecting index lead to different structures of limiting solution in two policies. In Equation (\ref{benchmark_outsourcing_single_push_index}) only one singular control is assigned to one face, while in our proposed one (\ref{learned_boundary_control}) singular control is state-dependent and can vary among different states on one face.

Based on Equation (\ref{benchmark_outsourcing_single_push_index}), we form an outsourcing and discretionary idling policy in the original system using the same approach used for our proposed one (see Section \ref{sec: policy_outsourcing}). To be specific, we apply the policy only when the workload process $W^n(t)$ leaves the state space $\mathcal{W}^n_\rho$, and distinguish among three cases, namely, the face-interior band subsets, the corner band subsets, and the region beyond the band. We define the policy for three cases in what follows, where the notation introduced in Section \ref{sec: policy_outsourcing} is reused.

\textit{Face-interior band subsets.} In this case, exactly one half-space constraint is violated, that is, \(\mathcal{K}_{\mathrm{out}}(W^n(t))\) is a singleton. We suppose that $\mathcal{K}_{\mathrm{out}}(W^n(t))=\{j\}$,
and recall that \(\sigma_j\) is the inward-pointing unit normal to the face \(F_j\). We then choose the index \(i^{\text{SP}}\) of the pushing direction as follows:
\begin{align}
    i^{\text{SP}} \in \operatorname*{argmin}_{i\in S(t)\cup \mathcal{A}:\  \tilde{g}_i \cdot \sigma_j > 0} \frac{ \tilde{\alpha}_i}{\tilde{g}_i \cdot \sigma_j}.\label{SP_interior_index}
\end{align}
If $i^{\text{SP}}\in S(t)$, the system manager idles server $i^{\text{SP}}$. If instead $i^{\text{SP}}\in\mathcal{A}$, she outsources the corresponding product type. Under the indexing conventions used in the definition of $\mathcal{A}$ and in (\ref{mathcal_U})--(\ref{workload_process_singular}), the set of outsourced product type is $i^{\text{SP}} -p$.

\textit{Corner band subsets and beyond the band.} In this case, we restrict attention to outsourcing controls only. For each violated constraint $j\in \mathcal{K}_{\mathrm{out}}(W^n(t))$, we choose the index $i^{\text{SP}}(j)$ of a pushing direction as follows:
\begin{align}
    i^{\text{SP}}(j) \in \operatorname*{argmin}_{i\in \mathcal{A}:\  \tilde{g}_i \cdot \sigma_j > 0} \frac{\tilde{\alpha}_i}{\tilde{g}_i \cdot \sigma_j},\label{SP_corner_outside_index}
\end{align}
and we apply the pushing controls $\bigl\{i^{\text{SP}}(j):j\in \mathcal{K}_{\mathrm{out}}(W^n(t))\bigr\}$
simultaneously, thereby outsourcing the corresponding product types. Under the indexing convention used to define $\mathcal{A}$, the set of outsourced product types is $\bigl\{i^{\text{SP}}(j)-p:j\in \mathcal{K}_{\mathrm{out}}(W^n(t))\bigr\}$.

Compared to our proposed outsourcing and discretionary idling policies (see Section \ref{sec: policy_outsourcing}), the single-pushing policy differ only in its objective of selecting the index (see Equations (\ref{SP_interior_index})--(\ref{SP_corner_outside_index})), which are solely based on instantaneous cost $\tilde{\alpha}_i$, and does not include gradient $\widetilde{G}(t)$. Similar to our proposed policy, its tuning parameters include bound scale $\rho$, band width $\epsilon$, and safety stock $s$.

\begin{remark}
    In general, one might regard mapping $i^{\text{SP}}(\cdot)$ as a tuning parameter beyond the specific choice given in Equation (\ref{benchmark_outsourcing_single_push_index}). Indeed, one can use any mapping $i^{\text{SP}}(\cdot):\mathcal{J}\to \{1,\dots,p+K\}$ that assigns one \textit{feasible} singular control $\tilde{g}_{i^{\text{SP}}(j)}$ to each face $F_j$ of state space, i.e., $\tilde{g}_{i^{\text{SP}}(j)} \cdot \sigma_j > 0$. However, the number of possible mapping $i^{\text{SP}}(\cdot)$ can grow exponentially with the number of faces, making a thorough search computationally intractable. Thus in our benchmark policy, we consider only the choice given in Equation (\ref{benchmark_outsourcing_single_push_index}).
\end{remark}

\subsubsection{Threshold-based Policy.}

Given queue length vector $Q^n(t)$ at time $t$, the threshold-based policy outsources new arrivals of type $k$ ($k = 1,\dots,K$) order if its total number in system exceeds its bound $q^n_k(\rho)$, i.e.,
\begin{align}
    \sum_{l=1}^{n_k} Q_{k,l}^n(t) \geq q^n_k(\rho).\label{threshold_based_outsourcing_policy}
\end{align}
Here, the queue-length bound $q^n_k(\rho)$ is parameterized by a tuning parameter $\rho$ to fit into delay-probability constraint setting. There is no discretionary idling in threshold-based policy. That is, server $i$ ($i = 1,\dots, p$) will idle only if all of its associated buffers are empty.

\subsection{Pricing Policies\label{sec: benchmark_pricing}} 

We introduce two alternative pricing policies. Both policies follow the dynamic pricing policy (\ref{p_n_Taylor}) prescribed by the Brownian control problem (\ref{BCP}), but specify different choices of drift-rate control policy $\{\zeta(t),\ t\geq 0\}$ adapted from literature. \textit{Linear pricing policy} adapts from a benchmark policy in \cite{Ata_Harrison_Si_2024_drift}, which assumes that drift-rate control depends linearly on the state vector. \textit{High-low pricing policy} adapts from $N$-price benchmark policy in \cite{Ata_Barjesteh_2023}.

\subsubsection{Linear Pricing Policy.} 

The linear pricing policy assumes a linear dependence of drift control $\zeta(t)$ in Equation (\ref{p_n_Taylor}) on queue lengths $Q^n(t)$ in original system. More specifically, its choice of $\zeta(t)$ in Equation (\ref{p_n_Taylor}) is
\begin{align}
\zeta^{\text{Lin}}_k(t) = \phi_k \cdot \left(\sum_{l=1}^{n_k} Q^n_{k,l}(t)\right) + \phi_{-k} \cdot \left(\sum_{k'\neq k} \sum_{l = 1}^{n_{k'}} Q^n_{k',l}(t)\right), \quad k = 1,\dots,K.\label{benchmark_pricing_linear_zeta}
\end{align}
Here, $\phi_k$ (resp. $\phi_{-k}$) is the tuning parameter representing the linear dependence of $k$th entry of drift control $\zeta^{\text{Lin}}(t)$ on total queue lengths in route $k$ (resp. all other routes). Substituting the choice (\ref{benchmark_pricing_linear_zeta}) for $\zeta(t)$ in (\ref{p_n_Taylor}) yields the following linear pricing policy: Given queue length $Q^n(t)$ at time $t$, choose the price vector
\begin{align}
    p^n(t)=\Lambda^{-1}(\lambda^*) + \frac{1}{\sqrt{n}}\nabla \Lambda^{-1}(\lambda^*)\zeta^{\text{Lin}}(t).\label{benchmark_pricing_linear}
\end{align}
Tuning parameters in linear pricing policy (\ref{benchmark_pricing_linear}) includes $\{(\phi_k,\phi_{-k}),\ k= 1,\dots,K\}$, denoted as $(\phi,\phi_{-})$ for brevity.

\subsubsection{High-low Pricing Policy.}

The high-low pricing policy sets different values for drift control $\zeta(t)$ in Equation (\ref{p_n_Taylor}), depending on the number of jobs in each route. More specifically, its choice of $\zeta(t)$ in Equation (\ref{p_n_Taylor}) is
\begin{align}
\zeta^{\text{HL}}_k(t) = \left\{\begin{array}{ll} -\zeta^L_k &\text{if} \ \sum_{l=1}^{n_k} Q^n_{k,l}(t) < \tilde{T}_k, \\  -\zeta^H_k &\text{otherwise,} \end{array}\right.\quad \text{for } k = 1,\dots,K.\label{benchmark_pricing_highlow_zeta}
\end{align}
Here, $\tilde{T}_k$, $\zeta_k^H \geq \zeta_k^L \geq 0 $ are tuning parameters. The parameter $\zeta_k^L$ (resp. $\zeta_k^H$) represents the magnitude of the pre-set $k$th entry of drift control vector $\zeta^{\text{HL}}(t)$ if total queue lengths of route $k$ is below (resp. above) certain threshold $\tilde{T}_k$. Substituting the choice (\ref{benchmark_pricing_highlow_zeta}) for $\zeta(t)$ in (\ref{p_n_Taylor}) yields the high-low pricing policy, defined as follows: Given queue length $Q^n(t)$ at time $t$, we choose the price vector 
\begin{align}
    p^n(t)=\Lambda^{-1}(\lambda^*) + \frac{1}{\sqrt{n}}\nabla \Lambda^{-1}(\lambda^*)\zeta^{\text{HL}}(t).\label{benchmark_pricing_highlow}
\end{align}
Tuning parameters in high-low pricing policy (\ref{benchmark_pricing_highlow}) includes $\{(\tilde{T}_k, \zeta_k^L, \zeta_k^H),\ k = 1,\dots,K\}$, denoted as $(\tilde{T},\zeta^L,\zeta^H)$ for brevity.

\begin{remark}\label{remark: pricing-shared-parameters}
    The linear and high-low pricing benchmark policies have $2K$ and $3K$ tuning parameters, respectively. Combined with tuning parameters in other policy components, a joint search is intractable even in our low-dimensional test problems. To make the search tractable, we impose a shared-parameter assumption across routes in our test problems; see Appendix \ref{appendix: prelimit_implementation_details}. It reduces number of tuning parameters to $2$ and $3$ in linear and high-low pricing benchmark policies, respectively.
\end{remark}

\begin{remark}
    One can also consider the nominal (resp. static) pricing policy adapted from \cite{Ata_Barjesteh_2023}, where the drift control rate in workload problem is assumed to be $\theta(t) = 0$ for $t \geq 0$ (resp. $\theta(t) = \bar{\theta}$ for $t \geq 0$ with tuning parameter $\bar{\theta} \in \Theta$). Since they are subsumed by both linear pricing and high-low pricing and lead to inferior numerical performance, we do not include them for brevity.
\end{remark}

\subsection{Scheduling Policies\label{sec: benchmark_scheduling}}

We introduce two alternative {scheduling} policies motivated by the least-relative-slack-type policies in literature (see \cite{Plambeck_Kumar_Harrison_2001}, \cite{Ata_2006_thin}, \cite{Celik_Maglaras_2008}). As mentioned earlier, we further incorporate safety stocks in the policy definition. Two policies have the same structure as our proposed scheduling policy (see Section \ref{sec: policy_scheduling}), and only differ in the choice of their index (i.e., $\mu_{k,l}P_{k,l}(t)$ in our proposed scheduling policy), which we specify next.

\subsubsection{Least-relative-slack Policy.}

The least-relative-slack policy lets the server
prioritize the job with the largest age relative to its throughput-time bound. In other words, it prioritizes the job with the least relative slack. Thus, the index for an nonempty buffer $(k,l)$ is defined as
\begin{align}
    \frac{a_{k,l}(t)}{d^n_k},\label{benchmark_scheduling_LRS}
\end{align}
where $a_{k,l}(t)$ denote the age of the oldest job in buffer $(k,l)$ at time $t$, and $d_k^n$ is the throughput-time bound for route $k$. The policy prioritizes the job with the highest index.

\subsubsection{Least-relative-slack Policy Based on Queue Lengths.}

The least-relative-slack policy based on queue lengths rests on a similar idea as the least-relative-slack policy, but it approximates the age using queue lengths. More specifically, we approximate the age $a_{k,l}(t)$ as follows:
\begin{align}
    a_{k,l}(t) \approx \frac{\sum_{j=1}^l Q^n_{k,j}(t)}{\lambda^n_k},\label{age_approx}
\end{align}
where the numerator counts number of jobs in buffer $(k,l)$ and its upstream buffers, and the denominator $\lambda_k^n$ is the arrival rate for route $k$. Plugging (\ref{age_approx}) into least-relative-slack index (\ref{benchmark_scheduling_LRS}) and noting that $d_k^n \lambda^n_k$ equals queue length bound $q_k^n$ on route $k$, we obtain the queue-length based index for an nonempty buffer $(k,l)$:
\begin{align}
\frac{\sum_{j=1}^l Q^n_{k,j}(t)}{q^n_k}.\label{benchmark_scheduling_LRS_queue_length}
\end{align}

\subsection{Choosing Policy Parameters\label{sec: benchmark_combination}}

\begin{table}[tbp]
\TABLE
{Policy parameters of benchmark policies in outsourcing and discretionary idling, pricing, and scheduling. \label{tab: benchmark_three_parts}}
{\begin{small}
\tabcolsep= 0.8em
\renewcommand{\arraystretch}{1.0}
\begin{tabular}{lcccc}
\toprule
 & &  & & Policy parameters \\
\midrule
\multicolumn{5}{l}{\textit{Panel A: Outsourcing and Discretionary Idling Benchmark Policies (Appendix \ref{sec: benchmark_outsourcing})}} \\ 
\midrule
Single pushing control per face policy
& & 
 & & 
$\epsilon, s, \rho \in \mathbb{R}$ \\
Threshold-based policy 
 & & 
 & &  $\rho \in \mathbb{R}$
 \\
\midrule
\multicolumn{5}{l}{\textit{Panel B: Pricing Benchmark Policies (Appendix \ref{sec: benchmark_pricing})}} \\
\midrule
Linear pricing policy
& &  
& &  
$(\phi,\phi_{-})\in \mathbb{R}^{K\times 2}$ \\
High-low pricing policy
& &  
& &  $(\tilde{T}, \zeta^L, \zeta^H)\in \mathbb{R}^{K\times 3}$
  \\
\midrule
\multicolumn{5}{l}{\textit{Panel C: Scheduling Benchmark Policies (Appendix \ref{sec: benchmark_scheduling})}} \\ 
\midrule
Proposed scheduling policy (Section \ref{sec: policy_scheduling})
 & &  & &  
$\tilde{\tau}\in \{\text{abs}, \text{rel}\},\tilde{\lambda}, \tilde{c}, s \in \mathbb{R}$
 \\
Least-relative-slack policy
& &  & & 
$s\in \mathbb{R}$ \\
Least-relative-slack policy based on queue lengths
& & & &  
$s\in \mathbb{R}$  \\
\bottomrule
\end{tabular}
\end{small}}
{}
\end{table}

The policy parameters for each of our $12$ benchmark policies include policy parameters for its three components; see Table \ref{tab: benchmark_three_parts}. We determine these parameters via simulation-based search, where we seek to minimize costs while ensuring that the probability of violating the delay constraints does not exceed $\varepsilon = 1\%$. Let us elaborate by separately considering two groups of benchmark policies.

The first group includes $8$ benchmark policies using least-relative-slack policy or least-relative-slack policy based on queue lengths. In this group, the policy parameters include $\rho$, $s$, $\epsilon$ (if using single-pushing policy), and pricing parameters (see Panel B of Table \ref{tab: benchmark_three_parts}). For these benchmark policies, we perform a joint search for all policy parameters, and select the one with the lowest cost subject to the delay probability constraint $\varepsilon = 1\%$.

\begin{algorithm}[htbp]
\begin{small}
\caption{Parameter Search for Benchmark Policies Using Proposed Scheduling Policy.}
\label{alg:parameter_search}
\KwIn{Search grids $\Pi_o$, $\Pi_p$, $\Pi_s$ for $\pi_o$, $\pi_p$, and $\pi_s$; delay probability constraint $\varepsilon$; a simulator returning the cost $C(\pi)$ and the delay probabilities $D(\pi)$ of any parameter triple $\pi = (\pi_o, \pi_p, \pi_s)$.}
\KwOut{Selected parameter triple $(\pi_o^*, \pi_p^*, \pi_s^*)$.}
\Fn{\textsc{ParameterSearch($\Pi_o$, $\Pi_p$, $\Pi_s$, $\varepsilon$)}}{
    $\mathcal{F} \gets \emptyset$\;
    \For{$\pi_o \in \Pi_o$}{
        $(\pi_p^*(\pi_o), \pi_s^*(\pi_o)) \gets$ \textsc{InnerLoop}($\pi_o,\Pi_p,\Pi_s,\varepsilon$)\tcp*{see Algorithm \ref{alg:inner_loop}}
        \lIf{$(\pi_p^*(\pi_o), \pi_s^*(\pi_o)) \neq \emptyset$}{
            $\mathcal{F} \gets \mathcal{F} \cup \{(\pi_o, \pi_p^*(\pi_o), \pi_s^*(\pi_o))\}$
        }
    }
    \KwRet $\argmin_{(\pi_o,\pi_p,\pi_s) \in \mathcal{F}} C(\pi_o, \pi_p, \pi_s)$
}
\end{small}
\end{algorithm}

\begin{algorithm}[tbp]
\begin{small}
\caption{Inner Loop: Alternating Pricing and Scheduling Search for a Given $\pi_o$.}
\label{alg:inner_loop}
\KwIn{Outer-loop parameter $\pi_o$; search grids $\Pi_p$, $\Pi_s$; delay probability target $\varepsilon$; initial scheduling parameter $\pi_s^{\mathrm{init}}$; maximum number of iterations $\bar{K}$.}
\KwOut{Pricing and scheduling parameters $(\pi_p^*(\pi_o), \pi_s^*(\pi_o))$, or $\emptyset$ if no feasible choice.}
\Fn{\textsc{InnerLoop($\pi_o,\Pi_p,\Pi_s,\varepsilon$)}}{
    $\pi_s^{(0)} \gets \pi_s^{\mathrm{init}}$;\quad $k \gets 0$\;
    \Repeat{$\pi_s^{(k)} = \pi_s^{(k-1)}$ \textbf{\emph{or}} $\mathcal{S} = \emptyset$ \textbf{\emph{or}} $k = \bar{K}$\tcp*[f]{stop if: converged / infeasible / max iter}}
    {
        $k \gets k + 1$\;
        $\pi_p^{(k)} \gets \argmin_{\pi_p \in \Pi_p} C(\pi_o, \pi_p, \pi_s^{(k-1)})$\tcp*{pricing search; without delay-prob constraint$^{\dagger}$}
        $\mathcal{S} \gets \{\pi_s \in \Pi_s : D(\pi_o, \pi_p^{(k)}, \pi_s) \leq \varepsilon\}$\tcp*{scheduling search; with delay-prob constraint}
        \lIf{$\mathcal{S} \neq \emptyset$}{$\pi_s^{(k)} \gets \argmin_{\pi_s \in \mathcal{S}} C(\pi_o, \pi_p^{(k)}, \pi_s)$}
    }
    \lIf{$\mathcal{S} = \emptyset$}{\KwRet $\emptyset$}
    \KwRet $(\pi_p^{(k)}, \pi_s^{(k)})$\;
}
\vspace{0.5ex}
\hspace*{-\algomargin}%
\begin{minipage}{\dimexpr\linewidth}
\renewcommand{\baselinestretch}{1}\footnotesize
$^{\dagger}$ The delay probability constraint is not imposed in the pricing search, because delay probabilities are primarily controlled by the outsourcing/idling and scheduling decisions; imposing the constraint here would force $\pi_p$ to take unreasonably large values to compensate for limitations in those decisions.
\end{minipage}
\end{small}
\end{algorithm}

The second group includes $4$ benchmark policies that use the proposed scheduling policy. In this group, the policy parameters include $\rho$, $s$, $\epsilon$ (if using single-pushing policy), pricing parameters (see Panel B of Table \ref{tab: benchmark_three_parts}), and scheduling parameters $\tilde{\tau}$, $\tilde{\lambda}$, $\tilde{c}$. Due to the large number of policy parameters, a joint search over all parameters is computationally intractable. We therefore adopt the sequential approach. More specifically, we do the parameter search in two nested loops, which inherently classifies policy parameters into three types. The outer loop enumerates parameters $\pi_o: = (\rho,s,\epsilon,\tilde{\tau})$, while the inner loop alternatively search over two sets of parameters, namely pricing $\pi_p:=(\text{pricing parameters in Panel B of Table \ref{tab: benchmark_three_parts}})$ and scheduling $\pi_s:=(\tilde{\lambda},\tilde{c})$, until a stopping criterion is met. Details of parameter search in this group are provided in Algorithm \ref{alg:parameter_search}, where the inner-loop function is given in Algorithm \ref{alg:inner_loop}.

\section{Details in Test Problems\label{appendix: test_problem_details}}

\subsection{Problem Parameters in Brownian Control Problem and Workload Problem}\label{appendix: test_problem_numerical_parameters}

This section starts from the model primitives in $n$th system (see Table \ref{tab: Prelimit_parameters}), and provides derived Brownian control problem (\ref{BCP}) and workload problem (\ref{EWF}) for two test problems. Values are given in exact fractional form here, whereas round to two decimals in Section \ref{sec: test_problems}.

\subsubsection{Rybko-Stolyar
Network.\label{appendix: test_problem_numerical_parameter_RS}}

With system parameter $n=400$, the model primitives in $n$th system are as follows. The demand function
$$\Lambda^n(x) = (950/3 - 190x_1/3, 950/3 - 190x_2/3),$$
the production cost $\delta^n = (1, 1)$, the (nominal) instantaneous demand rate $n\lambda^* = (380/3,380/3)$, the outsourcing cost $\nu^n = (0.05, 0.125)$, and the service rate $\mu^n = (400,200,400,200)$. Besides, the throughput time bound is $d^n = (1/3, 1/3)$, thus the queue-length bound is $q^n = (n\lambda^*)d^n = (380/9, 380/9)$. And the interest rate is $r^n = 0.01$.

\textit{Problem Parameters in Brownian Control Problem.} First, according to Assumption \ref{assumption: heavy_traffic} we have $\mu = (0.95, 0.475, 0.95, 0.475)$, $\eta = (1,0.5,1,0.5)$, and in Equation (\ref{lambda_assumption}) for $\lambda^n(t)$ we have $\lambda^* = 0.95\cdot (1/3,1/3)$. Then, the long-run average fraction of service time $x^* = (1/3,2/3,1/3,2/3)$. In scaled queue length process
(\ref{BCP_Z_k_1})--(\ref{BCP_Z_k_l}), $X$ is $4$-dimensional Brownian motion
with drift $\varphi$ and variance $\Sigma$,
which are defined in (\ref{X_BM_drift_variance}) and have
values
\begin{equation}
\varphi=\begin{pmatrix}
    -1/3\\
0\\
-1/3\\
0
\end{pmatrix}
,\quad\Sigma=0.95\begin{pmatrix}
2/3 & -1/3 &  & \\
-1/3 & 2/3 &  & \\
&  & 2/3 & -1/3\\
&  & -1/3 & 2/3
\end{pmatrix} . \label{vartheta_Sigma_RS}%
\end{equation}
The (scaled) queue-length bounds vector is $\beta =q^n/\sqrt{n} = (19/9,19/9)$. The outsourcing cost for each (scaled) unit of product is $\alpha = \nu^n \sqrt{n} = (1,2.5)$. And the interest rate $\tilde{r} = r^n = 0.01$.

\textit{Problem Parameters in Workload Problem.} First, in the definition (\ref{workload_process}) of workload process $W$, we have $m_{1,1} = m_{2,1} = 20/19$, $m_{1,2} = m_{2,2} = 40/19$, and the drift $\gamma$ and variance $\Gamma$ of Brownian motion $\xi$ are%
\begin{equation}
\gamma=\frac{20}{19}\begin{pmatrix}
-1\\
-1%
\end{pmatrix}
,\quad\Gamma=\frac{20}{19}\begin{pmatrix}
10/3 & 4/3\\
4/3 & 10/3
\end{pmatrix}, \label{gamma_Gamma_RS}%
\end{equation}
which are calculated based on definition of $\gamma$ and $\Gamma$ in
(\ref{gamma_Gamma_def}), numerical values of $\varphi$ and $\Sigma$ in (\ref{vartheta_Sigma_RS}), and
the numerical value of $M$,%
\[
M=\frac{20}{19}\begin{pmatrix}
1 & 0 & 2 & 2\\
2 & 2 & 1 & 0
\end{pmatrix}.
\]
Second, the state space $\mathcal{W}$ for
workload process $W$ defined in (\ref{W_state_space}) is%
\begin{equation}
\mathcal{W}=\left\{  w\in \mathbb{R}_+^2:\ 
w_{1},w_{2}\leq 20/3,\ 
w_{2}\geq2w_{1}-80/9,\ 
w_{1}\geq2w_{2}-80/9
\right\}  . \label{W_state_space_RS_appendix}%
\end{equation}
Third, given the profit rate function $\Pi(x)$ defined via demand function $\Lambda(x)$ and production cost $\delta$, where $\Lambda(x) = \Lambda^n(x)/n$ is given in Equation (\ref{assumption_Lambda_n_mu_n}), the matrix $H$ is 
$$H = \begin{pmatrix}
    120/19 & 0 \\
    0 & 120/19
\end{pmatrix},$$
the cost function $c(x)$ for effective drift rate $x\in\mathbb{R}^{2}$
and the corresponding optimal drift rate $\zeta^{\ast}(x)$ defined in
(\ref{c_x}) with closed-form expression in Lemma \ref{lemma:c_x} are%
\begin{equation}
c(x)=x^{\top} \begin{pmatrix}
    19/6 & -38/15\\
    -38/15 & 19/6
\end{pmatrix} x\quad \text{and}\quad \zeta^{\ast}(x)=\begin{pmatrix}
-19/60 & 19/30\\
19/30 & -19/60
\end{pmatrix}
x.\quad\label{c_zeta_RS}%
\end{equation}
Fourth, the singular control $\mathcal{U}$ has pushing matrix $\tilde{G}$ and costs vector $\tilde{\alpha}$ given as:
\begin{align}
    \tilde{G} = \begin{pmatrix}1&0&-20/19&-40/19\\
0&1&-40/19&-20/19\end{pmatrix},\quad \tilde{\alpha} = \begin{pmatrix}
    0 & 0 & 1 & 2.5
\end{pmatrix}^\top.\label{singular_direction_cost_RS_appendix}
\end{align}
Lastly, the interest rate is $\tilde{r} = 0.01$.

\subsubsection{Extended Criss-cross
Network.\label{appendix: test_problem_numerical_parameter_CC}}

With system parameter $n=400$, the model primitives in $n$th system are as follows. The demand function 
$$\Lambda^n(x) = (475 - 95x_1,\dots,475 - 95x_{p+1}),$$
the production cost $\delta^n = (1,\dots, 1)$, the (nominal) instantaneous demand rate $n\lambda^* = (190,\dots,190)$, the outsourcing cost $\nu^n = (0.05,\dots,0.05, 0.025p)$, and the service rate $\mu^n = (400,\dots,400)$. Besides, the throughput time bound is $d^n = (d^n_v, \dots, d^n_v, d^n_h)$, where we set $(d^n_v,d^n_h)$ equal to $(0.15,0.225)$ and $(0.15, 0.6)$ for the $3$-server and $10$-server cases respectively; we keep general form of $(d^n_v,d^n_h)$ in the following related quantities. The queue-length bounds vector is $q^n = (n\lambda^*)d^n =  (190d^n_v,\dots, 190d^n_v, 190d^n_h)$. And the interest rate is $r^n = 0.01$.

\textit{Problem Parameters in Brownian Control Problem.} First, according to Assumption \ref{assumption: heavy_traffic} we have $\mu = (0.95,\dots,0.95)$, $\eta = (1,\dots,1)$, and in Equation (\ref{lambda_assumption}) for $\lambda^n(t)$ we have $\lambda^* = 0.95\cdot (1/2,\dots,1/2)$. Then, the long-run average fraction of service time $x^* = (1/2,\dots,1/2)$. In scaled queue length process
(\ref{BCP_Z_k_1})--(\ref{BCP_Z_k_l}), $X$ is $2p$-dimensional Brownian
motion with drift $\varphi$ and variance
$\Sigma$, which are defined in (\ref{X_BM_drift_variance})
and have values
\begin{equation}
\varphi=\begin{pmatrix}
-1/2\\
\vdots\\
-1/2\\
-1/2\\
0\\
\vdots\\
0
\end{pmatrix}
,\quad\Sigma=0.95\begin{pmatrix}
1 &  &  &  &  &  & \\
& \ddots &  &  &  &  & \\
&  & 1 &  &  &  & \\
&  &  & 1 & -1/2 &  & \\
&  &  & -1/2 & \ddots & \ddots & \\
&  &  &  & \ddots & \ddots & -1/2\\
&  &  &  &  & -1/2 & 1
\end{pmatrix}. \label{vartheta_Sigma_CC}%
\end{equation}
The (scaled) queue-length bounds vector is $\beta = q^n/\sqrt{n} = (19d^n_v/2,\dots,19d^n_v/2,19d^n_h/2)$. The outsourcing cost for each (scaled) unit of product is $\alpha = \nu^n \sqrt{n} =(1,\dots,1,0.5p)$. And the interest rate is $\tilde{r} = r^n = 0.01$.

\textit{Problem Parameters in Workload Problem.}
First, in definition (\ref{workload_process}) of workload process $W$, we have $m_{k,l} = 20/19$ for $k = 1,\dots,K$, $l = 1,\dots,n_k$, and the drift $\gamma$ and 
covariance $\Gamma$ of Brownian motion $\xi$ are%
\begin{equation}
\gamma=\frac{20}{19}\begin{pmatrix}
-1\\
\vdots\\
-1%
\end{pmatrix}
,\quad\Gamma=\frac{20}{19}\begin{pmatrix}
2 &1/2&\dots&1/2\\
1/2&\ddots&\ddots&\vdots\\
\vdots&\ddots&\ddots&1/2\\
1/2&\dots&1/2&2
\end{pmatrix}
,
\label{gamma_Gamma_CC}%
\end{equation}
which are calculated based on definition of $\gamma$ and $\Gamma$ in
(\ref{gamma_Gamma_def}), numerical values of $\varphi$ and $\Sigma$ in (\ref{vartheta_Sigma_CC}), and
the numerical value of $p\times(2p)$ dimensional matrix $M$,%
\[
M=\frac{20}{19}\begin{pmatrix}
1 &  &  & 1 &  & \\
& \ddots &  & \vdots & \ddots & \\
&  & 1 & 1 & \cdots & 1%
\end{pmatrix}.
\]
Second, the state space $\mathcal{W}$ as defined in (\ref{W_state_space}) for
workload process $W$ is%
\begin{equation}
\mathcal{W}=\left\{  w\in \mathbb{R}_+^{p}:\ w_i\leq w_j +10d^n_v,\ 1\leq i<j\leq p;\ w_i \leq 10d^n_v + 10d^n_h,\ i = 1,\dots,p \right\}. \label{W_state_space_CC_appendix}%
\end{equation}
Third, given the profit rate function $\Pi(x)$ defined via demand function $\Lambda(x)$ and production cost $\delta$, where $\Lambda(x) = \Lambda^n(x)/n$ is given in Equation (\ref{assumption_Lambda_n_mu_n}), the matrix $H$ is 
$$H = \operatorname{diag} \begin{pmatrix}
    80/19 & \cdots & 80/19
\end{pmatrix}.$$
The cost function $c(x)$ for effective drift rate $x\in\mathbb{R}%
^{p}$ and the corresponding optimal drift rate $\zeta^{\ast}(x)$ defined in
(\ref{c_x}) with closed-form expression in Lemma \ref{lemma:c_x} are%
\begin{equation}
c(x)=x^{\prime}(\tilde{M}H^{-1}\tilde{M}^{\prime})^{-1}x,\qquad\zeta^{\ast
}(x)=H^{-1}\tilde{M}^{\prime}(\tilde{M}H^{-1}\tilde{M}^{\prime})^{-1}x,
\label{c_zeta_CC}%
\end{equation}
where the $p\times (p+1)$ dimensional matrix $\tilde{M}$ is 
\begin{align}
\tilde{M}=\frac{20}{19}\begin{pmatrix}
1 &  &  & 1 \\
& \ddots &  & \vdots  \\
&  & 1 & 1 
\end{pmatrix}.\label{M_tilde_ECC}
\end{align}
Fourth, the singular control $\mathcal{U}$ has direction matrix $\tilde{G}$ and cost vector $\tilde{\alpha}$:
$$\tilde{G} = \begin{pmatrix}
    I_p & -\tilde{M}
\end{pmatrix},\quad \tilde{\alpha} = \begin{pmatrix}
    0 & \cdots & 0 & 1 & \cdots & 1 & 0.5p
\end{pmatrix}^\top.$$
Lastly, the interest rate is $\tilde{r} = 0.01$.

\subsection{Implementation Details of Computational Method\label{appendix: computational_implementation_detail}}

This section provides implementation details of our computation method, including backup controls and the algorithms for solving the Skorokhod problem in Appendix \ref{appendix: implementation_reflection_mapping}, and neural network training setup in Appendix \ref{appendix: implementation_parameters_NN}.

\subsubsection{Backup Controls and Algorithm for Solving the Skorokhod Problem.\label{appendix: implementation_reflection_mapping}}

For each test problem, we specify the backup controls, including their pushing directions and associated costs, and the algorithm used to solve the resulting Skorokhod problem. The amounts of pushing reported below are aggregated over faces sharing the same backup-control direction and cost coefficient.

\paragraph{Rybko-Stolyar Network.}

We set the backup controls for the six faces $F_1,\ldots,F_6$ (in clockwise order from the origin; see Figure \ref{fig:RS_state_space}) as $\hat{L}_1,\hat{L}_1,\hat{U}_1,\hat{U}_1,\hat{L}_2,\hat{L}_2$, where $\hat{L}_1$ (resp. $\hat{L}_2$, $\hat{U}_1$) has the same pushing directions as $L_1$ (resp. $L_2$, $U_1$), but has cost $0.05$ (resp. $0.05$, $1.2$).
Thus, the reflection matrix is
$\hat{G} = \left(\tilde{g}_1,\ \tilde{g}_1, \ \tilde{g}_3,\ \tilde{g}_3,\ \tilde{g}_2,\  \tilde{g}_2\right)$.
Under this choice, the Skorokhod mapping admits a closed-form
expression, bypassing the need for an iterative solver.
Given a point $w=(w_1,w_2)\in\mathbb{R}^2 \backslash \mathcal{W}$, where $\mathcal{W}$ is given by (\ref{W_state_space_RS_appendix}), the projected boundary point
$\tilde{w}\in\partial\mathcal{W}$ is obtained sequentially by
\begin{equation}
\begin{aligned}
w^\circ
&= w+\tilde{g}_3
   \max\left\{
       0,\frac{w_1-b_1}{\tilde{M}_{11}},
       \frac{w_2-b_1}{\tilde{M}_{21}}
   \right\},\\
\tilde{w}_1
&= \max\left\{w^\circ_1, 0,\,2w^\circ_2-b_2\right\},\\
\tilde{w}_2 
&= \max\left\{w^\circ_2, 0,\,2\tilde{w}_1-b_2\right\}.
\end{aligned}\label{RS_skorokhod_explicit}
\end{equation}
Here, $w^\circ=(w^\circ_1,w^\circ_2)$ is an intermediate point, $b_1:=20/3$, $b_2:=80/9$, $\tilde{M}$ is given by (\ref{c_x_direction}) with $\tilde{M}_{11} = 20/19$, $\tilde{M}_{21} = 40/19$. Intuitively, the three steps in Equation (\ref{RS_skorokhod_explicit}) sequentially apply $\hat{U}_1$, $\hat{L}_1$, and $\hat{L}_2$. The corresponding push amount for each backup control is
$$\hat{L}_i:(\tilde{w}_i-w^\circ_i)^+,\ i=1,2,\quad \text{and}\quad \hat{U}_1:(w_1-w^\circ_1)^+/\tilde{M}_{11}.$$

\paragraph{Extended Criss-cross Network.}

We set the backup controls for each type of faces (see Table \ref{tab:faces_ECC}): 
\begin{equation} 
\begin{aligned} 
\text{i)}& \quad w_i = 0,\quad  1\leq i \leq p; && \text{ use }\hat{L}_i ;\\
\text{ii)}& \quad w_i = w_j + 10d^n_v,\quad  1\leq i<j\leq p; && \text{ use }\hat{L}_j;\\
\text{iii)}& \quad w_i = 10d^n_v + 10d^n_h,\quad 1\leq i \leq p;&& \text{ use }\hat{U}_i. 
\end{aligned}\label{extended_criss_cross_faces_backup} 
\end{equation} 
Here, $\hat{L}_i$ and $\hat{U}_i$ ($i = 1,\dots,p$) have the same pushing directions as $L_i$ and $U_i$, respectively; the costs of $\hat{L}_i$ are $0.05$ and $0.1$ in the 3-server and 10-server problems, respectively, and the cost of $\hat{U}_i$ is $24/19$ in both problems.

Under this choice, the Skorokhod mapping admits a closed-form expression, bypassing the need for an iterative solver. Given a point $w = (w_1,\dots,w_{p})\in \mathbb{R}^p\backslash \mathcal{W}$, the projected boundary point $\tilde{w} \in \partial\mathcal{W}$ is obtained sequentially by
\begin{align} 
\tilde{w}_1 = w_1|_{[0,b_v+b_h]},\quad
\tilde{w}_i = w_i |_{[\max(0,\max_{j\leq i-1} \tilde{w}_j - b_v),b_v + b_h]},\ i = 2,\dots,p,\nonumber
\end{align}
where $b_v:=10d^n_v$, $b_h:=10d^n_h$, and $a|_{[b,c]} := \min(\max(a,b),c)$. The corresponding amount of pushing for each backup control is
\begin{align}
\hat{L}_i:(\tilde{w}_i - w_i)^+\quad \text{and} \quad \hat{U}_i : (w_i - \tilde{w}_i)^+/\tilde{M}_{ii}\quad (i=1,\dots,p).\nonumber
\end{align}
Here, the value of $\tilde{M}_{ii}$ is given in Equation (\ref{M_tilde_ECC}).

\subsubsection{Neural Network Training Setup.\label{appendix: implementation_parameters_NN}}

\paragraph{Common Parameters Across All Test Problems.} Batch size $\bar{B} = 256$, time horizon $T = 0.1$, discretization step size $h = 0.1/64$, reference policy $\tilde{\psi} = 0$. The optimizer is Adam, and the activation function is elu.

\paragraph{Parameters Specific to Each Test Problem.}

Table \ref{tab: NN_parameters} lists the hyperparameters used for each test problem. These include the problem dimension and network architecture (\#Hidden layers and \#Neurons in each layer), the training schedule (\#Iterations, \#Epochs, and learning rate scheme), the upper bound $b$ on the drift control $\psi$ together with its growth rate (pace; see Equation (\ref{trick_increase_bound}) below), and the reference policy $\tilde{\theta}$. 

For each training run, we first simulate the reference process and discard the first 1{,}000 iterations as a warm-up period. We then save the remaining simulated iterations as the training data. One epoch consists of one complete pass through these retained data, and the networks are trained for \#Epochs epochs.

Recall that we train a separate neural network for each value of the policy parameter $\rho$. Within each test problem, the hyperparameters in Table \ref{tab: NN_parameters} are shared across all $\rho$, except for the reference policy $\tilde{\theta}$, which we choose adaptively to $\rho$. In reporting reference policy $\tilde{\theta}$, we parameterize it through total drift vector $\gamma + \tilde{\theta}$ under reference policy, assume identical value for all entries, and denote this common value by a scalar. For each $\rho$, we select this scalar from a short candidate list, discarding values that yield visually undertrained networks. The last row of Table \ref{tab: NN_parameters} reports the selected $\gamma + \tilde{\theta}$ for $11$ values of $\rho$ around the optimum.

\begin{table}[t]
\TABLE
{Parameters used in neural network training. ECC denotes extended criss-cross network. We assume the total drift $\gamma + \tilde{\theta} \in \mathbb{R}^p$ under the reference policy has identical entries and denote it by a scalar.\label{tab: NN_parameters}}
{\begin{small}
\tabcolsep= 0.2em
\renewcommand{\arraystretch}{0.9}
\begin{tabular}{lllllll}
\toprule
 & & Rybko-Stolyar & &
$3$-server ECC & &
$10$-server ECC \\
\midrule
Dimension & & $2$ & & $3$ & & $10$ \\
\#Hidden layers & & $3$ & & $3$ & & $4$ \\
\#Neurons in each layer & & $100$ & & $100$ & & $200$ \\
\#Iterations & & $6000$ & & $6000$ & & $9000$ \\
\#Epochs & & $40$ & & $50$ & & $100$ \\
\addlinespace[4pt]
Learning rate scheme & &
{\footnotesize\begin{tabular}{@{}l@{}}
5e-4 (0,2.8e4), \\ 3e-4 (2.8e4,6e4), \\
1e-4 (6e4,1e5), \\ 3e-5 (1e5,$\infty$)
\end{tabular}}
& &
{\footnotesize\begin{tabular}{@{}l@{}}
5e-4 (0, 1.9e4), \\ 3e-4 (1.9e4, 4.4e4), \\
1e-4 (4.4e4, 7e4), \\ 3e-5 (7e4, 1.4e5), \\
1e-5 (1.4e5, $\infty$)
\end{tabular}}
& &
{\footnotesize\begin{tabular}{@{}l@{}}
5e-4 (0, 1.9e4), \\ 3e-4 (1.9e4, 4.4e4), \\
1e-4 (4.4e4, 7e4), \\ 5e-5 (7e4, 1.4e5), \\
3e-5 (1.4e5, 2.7e5), \\ 1e-5 (2.7e5, $\infty$)
\end{tabular}} \\
\addlinespace[4pt]
Upper bound $b$ & & $100$ & & $200$ & & $100$ \\
Pace & & $10$ & & $10$ & & $20$ \\

\addlinespace[4pt]
{\begin{tabular}{@{}l@{}}
Total drift $\rho\mapsto \gamma + \tilde{\theta}$\\
under reference policy
\end{tabular}}
& &
{\footnotesize\begin{tabular}{@{}l@{}}
$0.95 \mapsto -0.8$, $0.96 \mapsto -0.8$,\\
$0.97 \mapsto -0.7$, $0.98 \mapsto -1.0$,\\
$0.99 \mapsto -0.9$, $1.00 \mapsto -0.9$,\\
$1.01 \mapsto -0.8$, $1.02 \mapsto -1.0$,\\
$1.03 \mapsto -1.0$, $1.04 \mapsto -0.9$,\\
$1.05 \mapsto -1.0$
\end{tabular}}
& &
{\footnotesize\begin{tabular}{@{}l@{}}
$0.85 \mapsto -0.8$, $0.86 \mapsto -0.8$,\\
$0.87 \mapsto -0.8$, $0.88 \mapsto -0.8$,\\
$0.89 \mapsto -0.8$, $0.90 \mapsto -0.8$,\\
$0.91 \mapsto -0.8$, $0.92 \mapsto -0.8$,\\
$0.93 \mapsto -0.8$, $0.94 \mapsto -0.8$,\\
$0.95 \mapsto -0.8$
\end{tabular}}
& &
{\footnotesize\begin{tabular}{@{}l@{}}
$0.85 \mapsto -0.4$, $0.86 \mapsto -0.4$,\\
$0.87 \mapsto -0.3$, $0.88 \mapsto -0.4$,\\
$0.89 \mapsto -0.5$, $0.90 \mapsto -0.6$,\\
$0.91 \mapsto -0.4$, $0.92 \mapsto -0.5$,\\
$0.93 \mapsto -0.6$, $0.94 \mapsto -0.4$,\\
$0.95 \mapsto -0.4$
\end{tabular}} \\
\bottomrule
\end{tabular}
\end{small}}
{}
\end{table}

In neural network training, we also adopt the following implementation choices introduced in \cite{Ata_Harrison_Si_2024_drift, Ata_Harrison_Si_2024_singular}:

\textit{1) Gradual increase of the upper bound $b$.} For all test problems, we gradually increase upper bound $b$ on drift control $\psi$ during training: at each iteration, $b$ is replaced by
\begin{align}
    \hat{b} = 
    \min\left(b,\  0.2b +\frac{\text{Iteration}}{40\cdot \text{Dim} \cdot \text{Pace}}\right),\label{trick_increase_bound}
\end{align}
where Dim is the dimension of problem, and Pace (listed in Table \ref{tab: NN_parameters}) controls the rate of increase.

\textit{2) Penalty term for the Rybko-Stolyar network.} 
Larger workload leads to higher future outsourcing costs, so the value function of the original singular-control problem should be monotone: $\nabla V(\cdot) \geq 0$. To encourage this during training, we add the penalty term $\Vert G_{w_2}(\cdot)^-\Vert^2$
to the loss function. 

\textit{3) Variance loss function.} For all test problems, we use a variance loss function in training rather than mean squared error. 
It is introduced in Appendix E.2 of \cite{Ata_Harrison_Si_2024_drift} to accelerate the training speed when interest rate is small. 
To be specific, we parameterize the value function as $V_{w_1}(x) = \hat{V}_{w_1}(x) + \xi_0$. Note that $\partial V_{w_1}(x)/ \partial x = \partial \hat{V}_{w_1}(x)/\partial x$. Rewriting the loss function (\ref{loss_function_SDE}), we have
\begin{align}
    \ell (w_1,w_2) &= \mathbb{E}\bigg[\bigg(e^{-\tilde{r} T}(\hat{V}_{w_1}(\tilde{W}(T))+\xi_0)-(\hat{V}_{w_1}(\tilde{W}(0))+\xi_0)  -\int_{0}^{T}e^{-\tilde{r} t}G_{w_2}(\tilde{W}(t))\cdot dB(t)\bigg.\nonumber\\
    &  \bigg.+\int_{0}^{T}e^{-\tilde{r} t}F(G_{w_2}(\tilde{W}(t)))dt+\int_{0}^{T}e^{-\tilde{r} t}%
    \hat{\alpha}\cdot d\tilde{\Upsilon}(t)\bigg)^2\bigg], \label{loss_function_SDE_rewrite}%
\end{align}
By optimizing $\xi_0$, we obtain the following variance loss function
\begin{align}
    \tilde{\ell} (w_1,w_2) &= \operatorname{Var}\bigg[e^{-\tilde{r} T}\hat{V}_{w_1}(\tilde{W}(T))-\hat{V}_{w_1}(\tilde{W}(0))  -\int_{0}^{T}e^{-\tilde{r} t}G_{w_2}(\tilde{W}(t))\cdot dB(t)\bigg.\nonumber\\
    &  \bigg.+\int_{0}^{T}e^{-\tilde{r} t}F(G_{w_2}(\tilde{W}(t)))dt+\int_{0}^{T}e^{-\tilde{r} t}%
    \hat{\alpha}\cdot d\tilde{\Upsilon}(t)\bigg]. \label{loss_function_SDE_variance}%
\end{align}
Note that $\hat{V}_{w_1}(\cdot)$ obtained via loss function (\ref{loss_function_SDE_variance}) determines the value function $V_{w_1}(\cdot)$ only up to an additive constant, and we only need gradient $G_{w_2}(\cdot)$ in determining our proposed policy.

\subsection{Implementation Details of Prelimit Simulation\label{appendix: prelimit_implementation_details}}

In this section, we provide implementation details of prelimit simulation, including our rationale in setting parameter search magnitude for pricing benchmark policies in Appendix \ref{appendix: prelimit_implementation_details_magnitude}, the shared-parameter assumption in pricing benchmark policies in Appendix \ref{appendix: prelimit_implementation_details_shared}, parameter-search setup in Appendix \ref{appendix: prelimit_implementation_details_search}, and simulation setup in Appendix \ref{appendix: prelimit_implementation_details_simulation}.

\subsubsection{Reference Parameter Magnitude for Pricing Benchmark Policies.\label{appendix: prelimit_implementation_details_magnitude}}
To provide a reasonable grid for the pricing parameters search, we note that the appropriate magnitude of pricing control in our problem is closely related to the outsourcing control. Indeed, since the alternative outsourcing option exists, a rational decision maker would not use pricing controls whose cost is dominated by using outsourcing. We utilize such rationale to provide a reference magnitude for the pricing parameters search, as follows.

Assume in the $n$th system arrival rate $\lambda^n(t) = n\lambda^* - \sqrt{n}\Delta\lambda(t)$. The arrival-rate deduction measured by $\sqrt{n}\Delta\lambda(t)$ can be achieved by either outsourcing or pricing. In approximating Brownian control problem, the outsourcing cost to achieve it is linear in $\Delta \lambda(t)$ with marginal cost $\alpha$, while the pricing cost rate for it is quadratic, $\Delta \lambda(t)^{\prime}H\Delta \lambda(t)$, with marginal cost $2H\Delta \lambda(t)$ strictly increasing in $|\Delta \lambda(t)|$. Equating the two marginal costs identifies the upper bound beyond which pricing is dominated by outsourcing:
\begin{equation}
\label{eq:mc-equality}
\alpha = 2H\Delta \lambda(t) \quad\Rightarrow\quad \Delta \lambda(t) = \frac{1}{2}H^{-1}\alpha.
\end{equation}
Equation~\eqref{eq:mc-equality} provides a reference magnitude for the drift control rate in Brownian control problem, based on which we define the following reference magnitude for high-low pricing:
\begin{equation}
\label{eq:zeta-ref}
\tilde\zeta := \frac{1}{2}H^{-1}\alpha.
\end{equation}
For linear pricing, assuming that the drift control $\zeta_k(t)$ equals $-\tilde\zeta_k$ in Equation (\ref{eq:zeta-ref}) when queue lengths of route $k$ hits its bound $q^n_k$ and $\phi_{-k} = 0$ in Equation (\ref{benchmark_pricing_linear_zeta}), we obtain a reference magnitude for the linear coefficient $\phi_k$:
\begin{equation}
\label{eq:phi-ref}
\tilde\phi_k = -\frac{\tilde\zeta_k}{q^n_k} = -\frac{1}{2}\frac{(H^{-1}\alpha)_k}{q^n_k}, \quad k = 1, \ldots, K.
\end{equation}
The numerical reference values of $\tilde{\zeta}$ and $\tilde{\phi}$ (see Equations (\ref{eq:zeta-ref})--(\ref{eq:phi-ref})) in our test problems are summarized in Table~\ref{tab:ref-magnitude}. 

\begin{table}[htbp]
\TABLE
{Reference magnitudes for high-low and linear pricing parameters. Rybko-Stolyar entries are for (route~1, route~2); ECC represents extended criss-cross network, whose entries are for (vertical route, horizontal route). \label{tab:ref-magnitude}}
{\begin{tabular}{lccc}
\toprule
 & Rybko-Stolyar & 3-server ECC & 10-server ECC \\
\midrule
High-low: $\tilde\zeta$ & $(0.08,\,0.20)$ & $(0.12,\,0.18)$ & $(0.12,\,0.59)$ \\
Linear: $\tilde\phi$    & $(-0.002,\,-0.0045)$ & $(-0.004,\,-0.004)$ & $(-0.004,\,-0.005)$\\
\bottomrule
\end{tabular}}
{}
\end{table}

It is worth mentioning that the derivation and resulting numerical values of $\tilde{\zeta}$ and $\tilde{\phi}$ only serves as a reference magnitude of the pricing parameters in our search grid, rather than a preset choice or a strict upper bound of parameter. In fact, while the rationale that pricing controls dominated by outsourcing would not be used holds broadly in the original system we consider, the specific numerical values derived above can depend on the modeling assumptions (e.g., the quadratic form of pricing cost). The detailed setup of the parameter search grid will be described later in this section.

\subsubsection{Shared-parameter Assumption in Pricing Benchmark Policies.\label{appendix: prelimit_implementation_details_shared}}

As said in Remark \ref{remark: pricing-shared-parameters}, we impose a shared-parameter assumption across routes for pricing benchmark parameters. We consider two variants.

The idea of first variant lies in applying a common percentage price change relative to the nominal price across routes during congestion periods. For high-low pricing, this means the congestion-period drift-control magnitude $\zeta_k^H$ is shared across routes; we also assume the non-congestion drift-control magnitude $\zeta_k^L$ and the threshold ratio $\tilde T_k / q^n_k$ are shared. For linear pricing, the analogous condition is that $\phi_k q^n_k$, which equals the drift control when route $k$'s queue length hits its bound $q^n_k$, is shared across routes. Since Equation (\ref{benchmark_pricing_linear_zeta}) also involves a cross-route coefficient $\phi_{-k}$ acting on the other routes' queue lengths (with scale $q^n_{-k} := \sum_{k' \neq k} q^n_{k'}$), we also assume $\phi_{-k} q^n_{-k}$ to be shared. Formally, for $k = 1, \ldots, K$:
\begin{align}
    & \zeta_k^H \propto 1, \quad \zeta_k^L \propto 1, \quad \frac{\tilde T_k}{q^n_k} \propto 1, \quad\text{(high-low pricing)}\label{eq:shared-A-hl}\\
    &\phi_k q^n_k \propto 1, \quad \phi_{-k} q^n_{-k} \propto 1. \quad \text{(linear pricing)} \label{eq:shared-A-lin}
\end{align}

The second variant builds on the first by further assuming that the pricing parameters are proportional to the route-specific outsourcing cost $\nu_k$ ($k = 1,\dots,K$), which is motivated by Equation (\ref{eq:mc-equality}). Formally, for $k = 1,\dots,K$:
\begin{align}
& \zeta_k^H \propto \nu_k, \quad \zeta_k^L \propto \nu_k, \quad \frac{\tilde T_k}{q^n_k} \propto 1, \quad \text{(high-low pricing)}\label{eq:shared-B-hl}\\
& \phi_k q^n_k \propto \nu_k, \quad \phi_{-k} q^n_{-k} \propto \nu_k. \quad \text{(linear pricing)}\label{eq:shared-B-lin}
\end{align}
This variant may yield parameters closer to the optimal one under our problem structure; however, as noted earlier, it relies on our model assumptions.

In our pricing benchmark parameter search, we consider both variants, rather than committing to one of them. More specifically, we pool the candidate parameters from the two variants into one single search list, so that the search procedure automatically selects the better-performing variant.

\subsubsection{Parameter Search Setup.\label{appendix: prelimit_implementation_details_search}}

Recall that each of our proposed policy and 12 benchmark policies consists of three components, i.e., outsourcing and discretionary idling, pricing, and scheduling. For each choice in each component, we set the same search grid for all related policies to ensure a fair comparison. The parameter search grid is summarized in Table \ref{tab:search_grid} and detailed in what follows. 

\begin{table}[tbp]
\TABLE
{Search grid of tuning parameters in proposed policy and 12 benchmark policies. ECC denotes extended criss-cross network. \label{tab:search_grid}}
{\tabcolsep= 0.88em
\renewcommand{\arraystretch}{0.9}
\begin{tabular}{llllll}
\toprule
\multicolumn{2}{l}{Parameter} & Spacing & Rybko-Stolyar & 3-server ECC & 10-server ECC \\
\midrule
\multicolumn{2}{l}{$\phi_1$}
 & uniform$\times\tilde{\phi}_1$ &
$[0, 1]$ \#11 &
$[0, 1]$ \#11 &
$[0, 1]$ \#11
\\
\multicolumn{2}{l}{$\phi_{-1}$}
 & uniform$\times\tilde{\phi}_1q^n_1/q^n_{-1}$ &
$[-0.5, 0.5]$ \#11 &
$[-0.5, 0.5]$ \#11 &
$[-0.5, 0.5]$ \#11
 \\
 \multicolumn{2}{l}{$\tilde{T}_1$}
 & uniform$\times q^n_1(\rho)$ &
$[0.6, 0.9]$ \#4 &
$[0.6, 0.9]$ \#4 &
$[0.6, 0.9]$ \#4 \\
\multicolumn{2}{l}{$\zeta_1^H, \zeta_1^L$}
 & uniform$\times\tilde{\zeta}_1$ &
$[0, 1]$ \#11 &
$[0, 1]$ \#11 &
$[0, 1]$ \#11 \\
\midrule
$\tilde{\lambda}$
&
\begin{tabular}{@{}l@{}}
$\tilde{\tau}=\text{abs}$ \\
$\tilde{\tau}=\text{rel}$
\end{tabular}
 &
\begin{tabular}{@{}l@{}}
uniform \\
uniform$\times\sqrt{n}$
\end{tabular}
&
\begin{tabular}{@{}l@{}}
$[2,20]$ \#10 \\
--
\end{tabular}
&
\begin{tabular}{@{}l@{}}
$[1,7]$ \#7 \\
$[2,14]$ \#7
\end{tabular} &
\begin{tabular}{@{}l@{}}
$[1,7]$ \#7 \\
$[2,14]$ \#7
\end{tabular}
\\
$\tilde{c}$
&
\begin{tabular}{@{}l@{}}
$\tilde{\tau}=\text{abs}$ \\
$\tilde{\tau}=\text{rel}$
\end{tabular}
 &
\begin{tabular}{@{}l@{}}
uniform$\times q^n_1(\rho)$ \\
uniform
\end{tabular}
&
\begin{tabular}{@{}l@{}}
$[0.1, 0.4]$ \#11 \\
--
\end{tabular}
 &
 \begin{tabular}{@{}l@{}}
$[0.1, 0.4]$ \#11 \\
$[0.6, 0.9]$ \#11
\end{tabular} &
 \begin{tabular}{@{}l@{}}
$[0.1, 0.4]$ \#11 \\
$[0.6, 0.9]$ \#11
\end{tabular}
\\
\midrule
\multicolumn{2}{l}{$\rho$}
 & uniform & step $0.01$ \#7 & step $0.01$ \#7 & step $0.01$ \#7 \\
\multicolumn{2}{l}{$s$}
 & integer
 & $\{0, 1, 2\}$ \#3
 & $\{0,1\}$ \#2
 & $\{0,1\}$ \#2 \\
\multicolumn{2}{l}{$\epsilon$}
 & uniform
 & $[0, 2]$ \#5
 & $[0, 1.5]$ \#4
 & $[0, 1.5]$ \#4 \\
\bottomrule
\end{tabular}}
{\textit{Note.} ``uniform$\times\tilde{\phi}_1$'' (resp.\ ``uniform$\times q^n_1(\rho)$'', etc.) denotes a uniform grid scaled by $\tilde{\phi}_1$ (resp.\ $q^n_1(\rho)$); the reference magnitudes $\tilde{\phi}_1$ and $\tilde{\zeta}_1$ are given in Table \ref{tab:ref-magnitude}. The notation ``$[a, b]\ \#N$'' indicates a grid of $N$ values on $[a, b]$ generated by the corresponding spacing rule; ``$\{\cdot\}\ \#N$'' indicates an integer set with $N$ elements. Under the shared-parameter assumption, only route-1 pricing parameters are searched; the rest are determined by the proportionality relations in (\ref{eq:shared-A-hl})--(\ref{eq:shared-B-lin}). The joint search over $\zeta_1^H$ and $\zeta_1^L$ is further restricted to $\zeta_1^H\geq \zeta_1^L \geq 0$. The parameter $\tilde{\tau}$ uses only $\tilde{\tau} = \text{abs}$ in Rybko-Stolyar network, indicated by ``--'', and takes values in $\{\text{abs}, \text{rel}\}$ for the extended criss-cross networks. For $\rho$, the grid is policy-specific and centered at each policy's approximate optimum (located via a coarse search) with 7 values of step $0.01$.}
\end{table}

\textit{1) Pricing parameters.}
Under both variants of shared-parameter assumption (see Appendix \ref{appendix: prelimit_implementation_details_shared}), only the route-1 pricing parameters are searched; the rest are automatically determined by Equations (\ref{eq:shared-A-hl})--(\ref{eq:shared-B-lin}). Therefore, the parameters to search include $(\phi_1,\phi_{-1})$ (resp. $(\tilde{T}_1,\zeta_1^H,\zeta_1^L)$) for linear (resp. high-low) pricing, whose search grid is set based on reference magnitude $\tilde{\phi}_1$ (resp. $\tilde{\zeta}_1$) given in Table \ref{tab:ref-magnitude}. For linear pricing, $\phi_1$ is generated as a ratio of $\tilde\phi_1$ from a uniform grid, and $\phi_{-1}$ as a ratio of $\tilde\phi_1q^n_1/q^n_{-1}$, where the factor $q^n_1/q^n_{-1}$ accounts for the scale difference between the queue-length terms multiplied by $\phi_1$ and $\phi_{-1}$ in Equation (\ref{benchmark_pricing_linear_zeta}). For high-low pricing, the queue threshold $\tilde{T}_1$ is generated by a ratio grid over queue-length bound $q^n_1(\rho)$, and drift-control values $\zeta_1^H$ and $\zeta_1^L$ are each generated by a shared ratio grid over $\tilde\zeta_1$, with the restriction $\zeta_1^H\geq \zeta_1^L \geq 0$. 

\textit{2) Scheduling parameters.} Recall that the tuning parameter $\tilde{\tau} \in \{\text{abs}, \text{rel}\}$ specifies the form of congestion-level term in proposed scheduling policy. For $\tilde{\tau} = \text{abs}$, we set the parameter $\tilde{c}$ as a ratio of the first-route queue-lengths bound $q^n_1(\rho)$. For $\tilde{\tau} = \text{rel}$, we set search list of $\tilde{\lambda}$ as a ratio of $\sqrt{n}$, as it balances the congestion-level term $\Delta_k(t;\text{rel},\tilde{c})$ with order $O(1)$ and the standard pressure term with order $O(\sqrt{n})$. In the Rybko-Stolyar problem, the two routes have identical bounds, making the abs and rel rules equivalent under a parameter transformation; we therefore search only $\tilde{\tau} = \text{abs}$ to avoid redundant computation.

\textit{3) Parameter $\rho$.} The search grid for $\rho$ is problem- and policy-specific. We adopt a two-stage search: a coarse grid first locates each policy's approximate optimal $\rho$, which suffices because cost and delay are monotone in $\rho$; we then refine with $7$ grid values spaced $0.01$ apart centered around this approximate optimum. A uniform grid covering all policies' optima with step $0.01$ would be computationally infeasible, since the optima differ substantially across policies. The two-stage search concentrates simulation effort near each policy's optimum.

\textit{4) Parameter $s$.} We search safety-stock $s$ over a small set of nonnegative integers, expanding the upper bound whenever the optimum is attained there.

\textit{5) Parameter $\epsilon$.} We assume $\epsilon^n = \epsilon \cdot n^{1/4}$ and search over tuning parameter $\epsilon$. The grid range is chosen such that either the optimum lies in the interior of the grid, or the optimum lies at the upper bound but the optimized cost grouped by $\epsilon$ has plateaued (i.e., the optimized cost at the upper bound differs by less than $1\%$ from that at the adjacent grid point).

Besides the tuning parameters mentioned in Table \ref{tab:search_grid}, there are additional parameters for $4$ benchmark policies using proposed scheduling policy (see Appendix \ref{sec: benchmark_combination}): we set the initial scheduling parameters $\pi_s^{\text{init}}$ as $\tilde{\lambda} = 4$ and $\tilde{c} = 0.25 q^n_1(\rho)$ for $\tilde{\tau} = \text{abs}$, and $\tilde{\lambda} = 4\sqrt{n}$ and $\tilde{c} = 0.75$ for $\tilde{\tau} = \text{rel}$; the maximum number of pricing-scheduling iterations is set as $\bar{K} = 5$.

After the search, we check whether the searched optimum lies in the interior of the grid or has plateaued at the boundary (i.e., the optimized cost at the boundary differs by less than $1\%$ from that at the adjacent grid point). If neither holds for a certain policy along some component, we extend that policy's grid for the component using the same spacing and rerun the parameter search for this policy, repeating until one of the two conditions is met. Other policies sharing the original grid are unaffected.

\subsubsection{Simulation Setup.\label{appendix: prelimit_implementation_details_simulation}} 

For each test problem, policy and tuning parameter, we simulate under the following setup.

\textit{Number of paths and simulation length.} We use $5{,}000$ simulation paths, and each path is simulated for $100{,}000$ events, where an event represents new job arrival or service completion. The discounted cost beyond the simulation horizon is approximated using a steady-state cost rate, which is estimated by the average cost rate over the last $20\%$ of the simulation events. The number of paths and simulation length are limited by the computation burden; see Appendix \ref{appendix: prelimit_implementation_details_burden}.

\textit{Initial state.} The initial state is set to $10$ jobs at the first buffer of each route and $0$ jobs elsewhere. The reported discounted cost depends on this initial state.

\textit{Delay probability.} In computing delay probability of each route, we discard jobs departing during the first $20\%$ of the simulation events, and only include those departing during the remaining $80\%$ simulation events.

\textit{Random seed.} For each test problem, all policies and parameter combinations share the same random seed, ensuring that relative performance differences reflect parameter choices rather than simulation noise.

\subsection{Computational Burden\label{appendix: prelimit_implementation_details_burden}} 

The neural network training is implemented in TensorFlow and adapts the code in \cite{Ata_Harrison_Si_2024_drift, Ata_Harrison_Si_2024_singular}. In neural network training, given each value of $\rho$ and corresponding reference policy (see Table \ref{tab: NN_parameters}), the Rybko-Stolyar problem takes approximately 3 hours using 10 CPUs, the 3-server problem takes approximately 6 hours using 5 CPUs, and the 10-server problem takes approximately 1 day using 20 CPUs.

The prelimit simulation is implemented in C++. We report an estimated computational burden of each test problem and policy in Table \ref{tab:comp_burden}, including the number of parameter combinations to be simulated (given search grid of each parameter in Table~\ref{tab:search_grid}) and a total single-core CPU time. The aliases of benchmark policies follow Table \ref{tab: benchmark_parameter_search_alias}. Recall the different search schemes introduced in Appendix \ref{sec: benchmark_combination}. For benchmarks using proposed scheduling (C1--C2, D1--D2) and thus the alternating pricing--scheduling search in the inner loop, the reported number equals the product of outer-loop grid size $|\Pi_o|$, inner-loop grid size $|\Pi_p|+|\Pi_s|$, and $2$ (the approximate average number of iterations observed in our experiments).

\begin{table}[tbp]
\TABLE
{Estimated computational burden of prelimit simulation with 5,000 paths per parameter combination. ECC denotes extended criss-cross network. \label{tab:comp_burden}}
{\tabcolsep= 0.88em
\renewcommand{\arraystretch}{0.9}
\begin{tabular}{lcccccc}
\toprule
 & \multicolumn{2}{c}{Rybko-Stolyar} & \multicolumn{2}{c}{3-server ECC} & \multicolumn{2}{c}{10-server ECC} \\
\cmidrule(lr){2-3} \cmidrule(lr){4-5} \cmidrule(lr){6-7}
Policy & \#Comb. & CPU$\cdot$days & \#Comb. & CPU$\cdot$days & \#Comb. & CPU$\cdot$days \\
\midrule
Proposed           & \textbf{11,550} & \textbf{22} &  \textbf{8,624} & \textbf{23} &  \textbf{8,624} & \textbf{420} \\
\midrule
A1, A2             & 25,410 & 11 & 13,552 &  6 & 13,552 &  12 \\
A3, A4             & 55,440 & 22 & 29,568 & 12 & 29,568 &  23 \\
B1, B2             &  5,082 &  2 &  3,388 &  1 &  3,388 &   3 \\
B3, B4             & 11,088 &  4 &  7,392 &  3 &  7,392 &   6 \\
\midrule
C1                 &  \textbf{73,920} &  \textbf{73} &  \textbf{71,456} &  \textbf{78} &  \textbf{71,456} & \textbf{164} \\
C2                 & \textbf{133,980} & \textbf{146} & \textbf{135,520} & \textbf{156} & \textbf{135,520} & \textbf{327} \\
D1                 &  14,784 &  15 &  17,864 &  19 &  17,864 &  41 \\
D2                 &  26,796 &  29 &  33,880 &  39 &  33,880 &  82 \\
\bottomrule
\end{tabular}}
{\textit{Note.} ``\#Comb.'' denotes the number of parameter combinations simulated, and ``CPU$\cdot$days'' denotes the approximate total single-core simulation time. The aliases of benchmark policies follow Table \ref{tab: benchmark_parameter_search_alias}. For A1--B4, each entry refers to a single policy (e.g., A1 and A2 each require the reported amount). For C1--D2, which use the alternating pricing--scheduling search in the inner loop of Algorithm \ref{alg:inner_loop}, entries assume two iterations of the inner loop for every outer-loop parameter $\pi_o$, which is the average number of iterations observed in our experiments. }
\end{table}

\begin{table}[htbp]
\TABLE
{Composition and aliases of the 12 benchmark policies. Least-relative-slack (queues) = Least-relative-slack policy based on queue lengths.\label{tab: benchmark_parameter_search_alias}}
{\tabcolsep= 0.4em
\renewcommand{\arraystretch}{0.9}
\begin{tabular}{lcccccc}
\toprule
Alias & & Outsourcing and Discretionary Idling & & Pricing & & Scheduling
\\  
\midrule
A1
& & Single-pushing & & Linear  & & Least-relative-slack \\
A2 & & Single-pushing & & Linear  & & Least-relative-slack (queues) \\
A3
& & Single-pushing & & High-low  & & Least-relative-slack \\
A4 & & Single-pushing & & High-low  & & Least-relative-slack (queues) \\ \midrule
B1 & & Threshold-based & & Linear  & & Least-relative-slack \\
B2  & & Threshold-based & & Linear  & & Least-relative-slack (queues) \\
B3 & & Threshold-based & & High-low  & & Least-relative-slack \\
B4  & & Threshold-based & & High-low  & & Least-relative-slack (queues) \\ \midrule
C1 & & Single-pushing & & Linear  & & Proposed scheduling \\
C2 & & Single-pushing & & High-low  & & Proposed scheduling \\ 
\midrule
D1  & & Threshold-based & & Linear  & & Proposed scheduling \\
D2  & & Threshold-based & & High-low  & & Proposed scheduling \\ 
\bottomrule
\end{tabular}}
{}
\end{table}

The number of parameter combinations varies across different benchmark policies, due to different number of tuning parameters involved (see Table \ref{tab: benchmark_three_parts}) and different search schemes used (see Appendix \ref{sec: benchmark_combination}). Among $12$ benchmark policies, the most computationally demanding benchmarks are C1/C2 using single-pushing policy, linear/high-low pricing, and proposed scheduling. In fact, they include $4$ parameters in outer-loop $\Pi_o$, $2$--$3$ parameters in inner-loop pricing $\Pi_p$, and $2$ parameters in inner-loop scheduling $\Pi_s$, thus the computational burden of their search is roughly equivalent to a $6$--$7$ dimensional joint search. In contrast, our proposed policy uses a $6$ dimensional joint search (see Section \ref{sec: policy_parameters}). Thus benchmark policies C1/C2 have significantly more parameter combinations than our proposed policy.

Despite having fewer parameter combinations, the proposed policy incur total simulation costs of the same order of magnitude as benchmarks C1/C2, as it involves evaluating neural network based gradient values per simulation step. In Rybko-Stolyar and 3-server problems, the per-simulation time is low because we cache the evaluated gradient values in a map and reuse the gradient by fast table lookup then. In the 10-server problem, such caching is infeasible due to large number of states in higher-dimensional state space, thus the neural network is evaluated during simulation and leads to large simulation time.

Assuming ideal parallelization across 200 CPUs, the prelimit simulation for all policies takes approximately 9.2 days in total, of which 2.3 days are for the proposed policy and 6.9 days are for the 12 benchmark policies. Most of the computational burden comes from the proposed policy and benchmarks C1 and C2, which together take 7.0 days. By test problem, the Rybko-Stolyar problem requires 1.2 days for these three policies (0.11, 0.36, and 0.73 days for the proposed policy, C1, and C2, respectively), the 3-server problem requires 1.3 days (0.12, 0.39, and 0.78 days), and the 10-server problem requires 4.6 days (2.10, 0.82, and 1.63 days).

\section{Further Results on Test Problems\label{appendix: test_problem_further_results}}

\subsection{Two-dimensional Slices of Gradient in $3$-server and $10$-server Problems\label{appendix: test_problem_NN_training}} 

\begin{figure}[h]
\FIGURE
{\begin{tabular}{c}
\includegraphics[width=0.5\figbase]{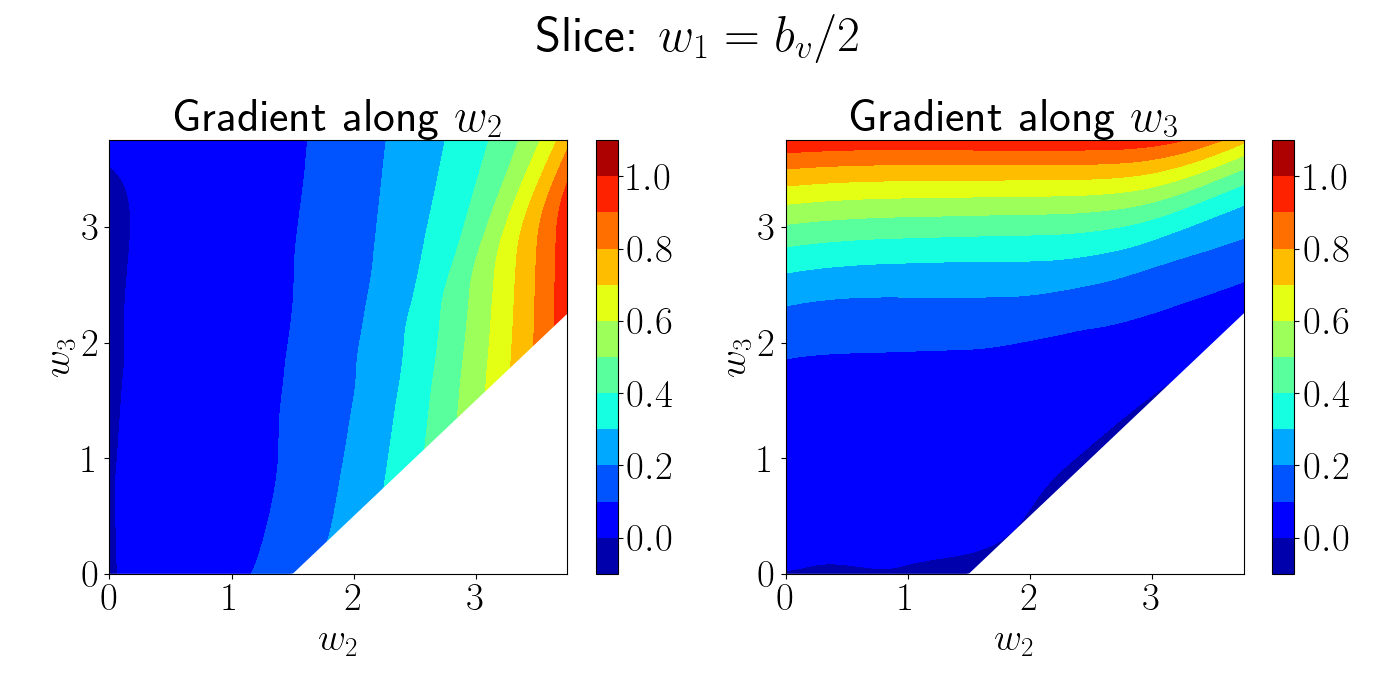}
\includegraphics[width=0.5\figbase]{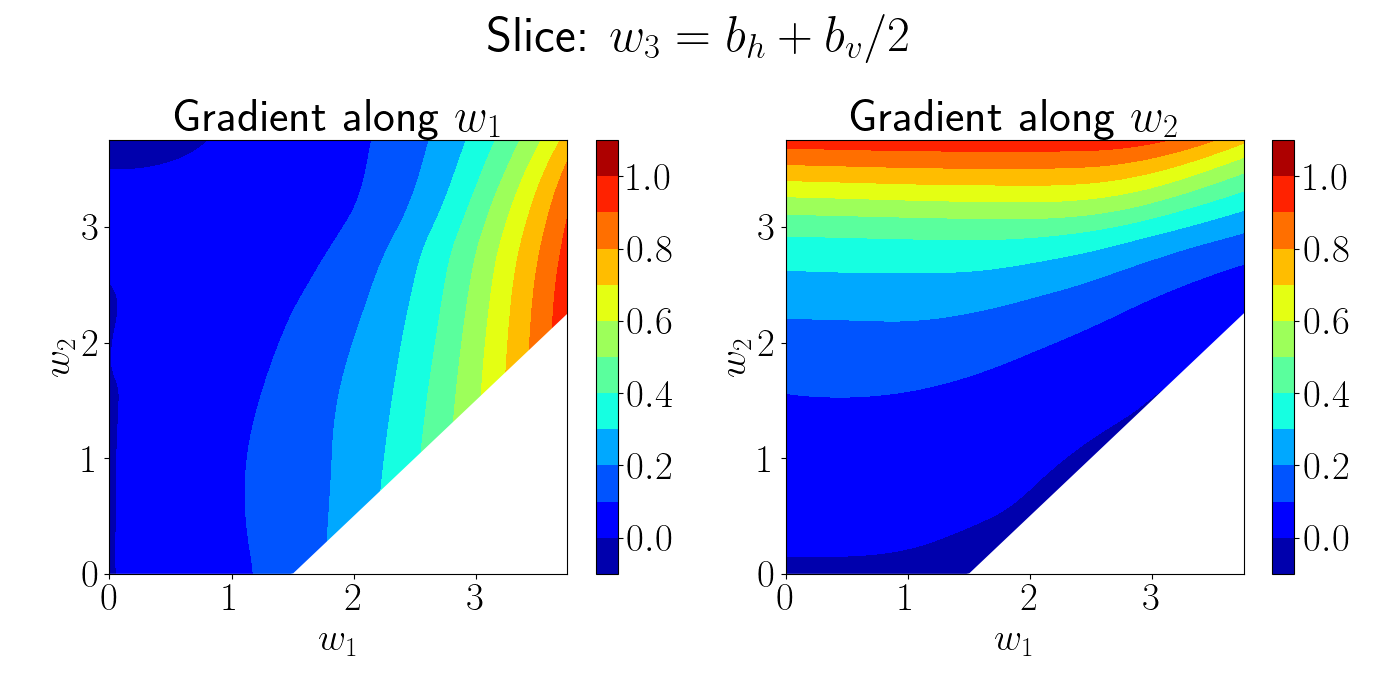}
\end{tabular}}
{Two-dimensional slices of the gradient function in $3$-server extended criss-cross problem. Denote $b_v:=10d^n_v$ and $b_h:=10d^n_h$. In the left two panels, we fix $w_1 = b_v/2$ and plot the second and third entries of gradient function on $(w_2, w_3)$-space. In the right two panels, we fix $w_3 = b_h + b_v/2$ and plot the first and second entries of gradient function on $(w_1, w_2)$-space.
\label{fig:ECC3_NN_gradient}}
{}
\end{figure}

\begin{figure}[h]
\FIGURE
{\begin{tabular}{c}
\includegraphics[width=0.5\figbase]{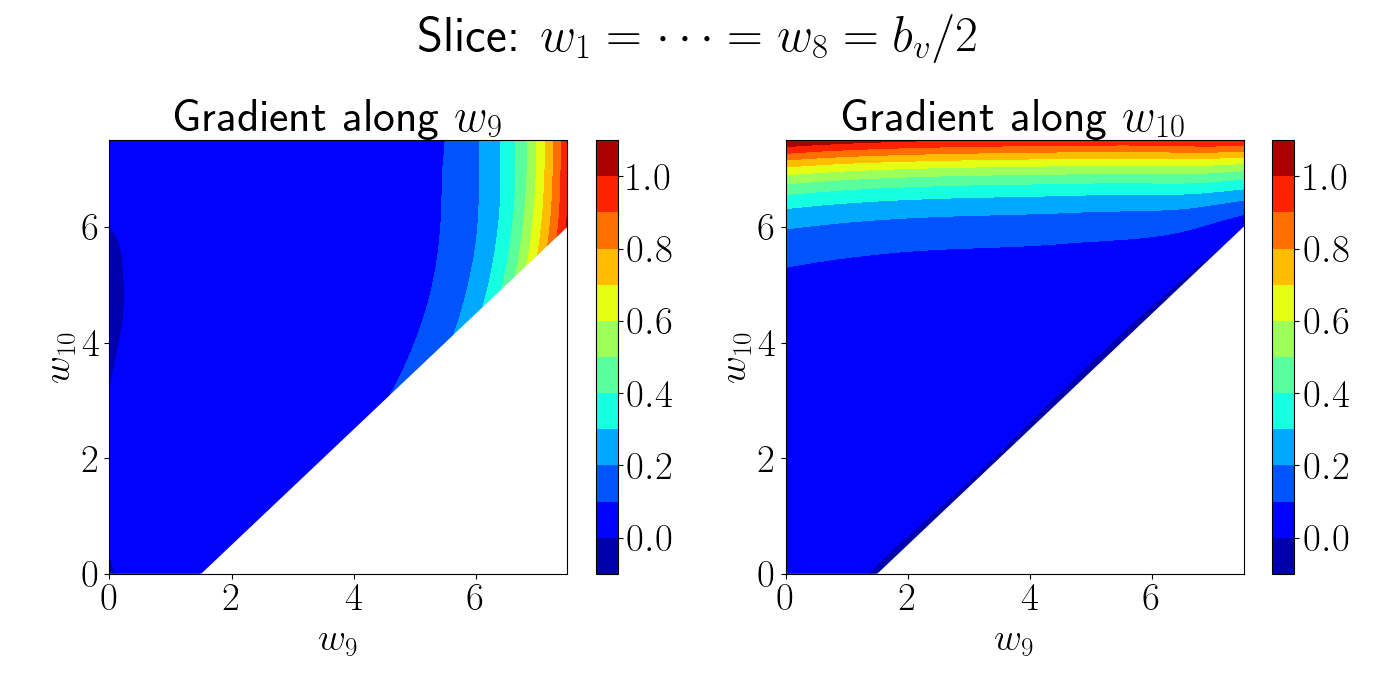}
\includegraphics[width=0.5\figbase]{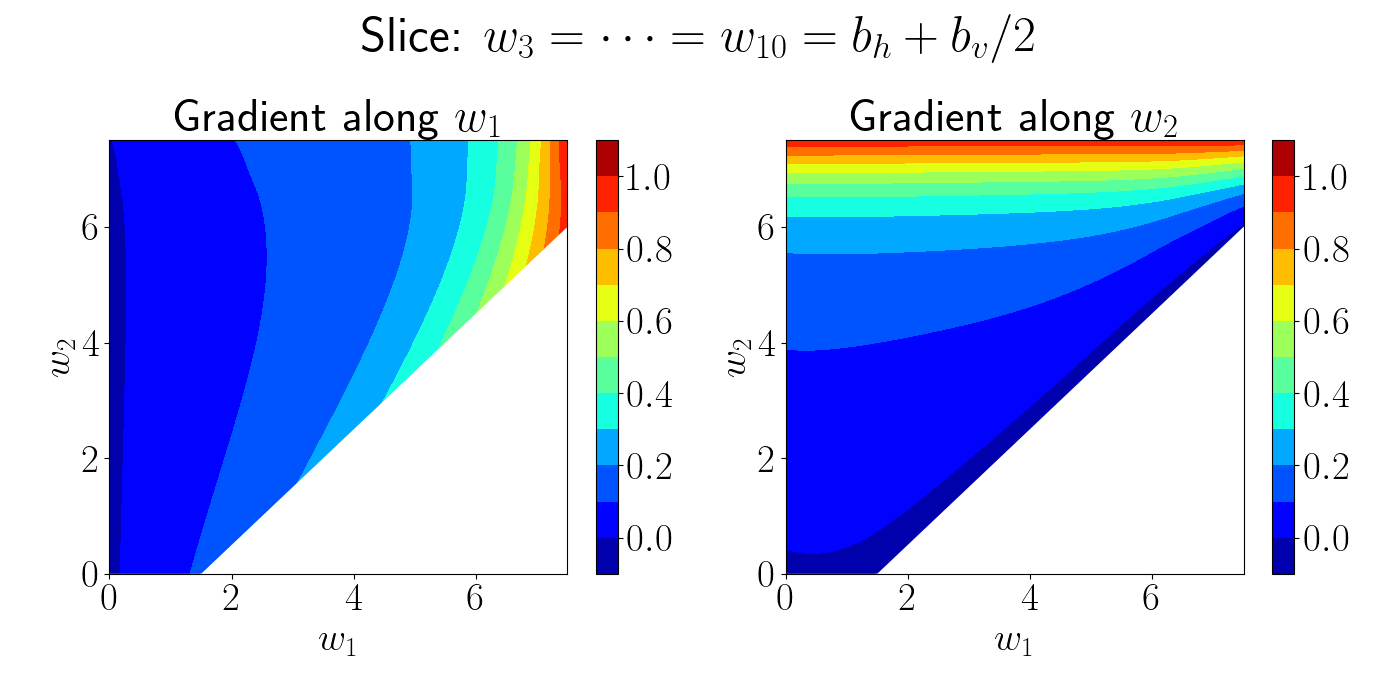}
\end{tabular}}
{Two-dimensional slides of the gradient function in $10$-server extended criss-cross problem. Denote $b_v:=10d^n_v$ and $b_h:=10d^n_h$. In the left two panels, we fix $w_1 = \cdots = w_8 = b_v/2$ and plot the $9$th and $10$th entries of gradient function on $(w_9, w_{10})$-space. In the right two panels, we fix $w_3 = \dots = w_{10} = b_h + b_v/2$ and plot the first and second entries of gradient function on $(w_1, w_2)$-space.
\label{fig:ECC10_NN_gradient}}
{}
\end{figure}

In Figures \ref{fig:ECC3_NN_gradient} and \ref{fig:ECC10_NN_gradient}, we visualize two-dimensional slices of the $3$-dimensional gradient and $10$-dimensional gradient in the workload problem for the $3$-server and $10$-server problems, respectively. In each $p$-server ($p$ equals $3$ or $10$) problem, we present two types of two-dimensional slices. The first type of slice fixes a workload value for first $p-2$ dimensions, and presents the value of last two entries of gradient $G_{w_2}(\cdot)$ along last two dimensions $(w_{p-1},w_p)$. The second type of slice fixes a workload value for last $p-2$ dimensions, and presents the value of first two entries of gradient $G_{w_2}(\cdot)$ along first two dimensions $(w_1,w_2)$.

\subsection{Prelimit Simulation\label{appendix: test_problem_prelimit_simulation}} 

We provide searched tuning parameters for proposed policy and $12$ benchmark policies in Appendix \ref{appendix: test_problem_optimal_parameters}, and the corresponding delay performance in Appendix \ref{appendix: test_problem_delay_performance}.

\subsubsection{Searched Tuning Parameters in the Original System.\label{appendix: test_problem_optimal_parameters}} 

Table \ref{tab: searched_optimal_parameters} reports the searched tuning parameters for proposed policy and $12$ benchmark policies in the original system, for which we explain in the following paragraphs. 

{\footnotesize
\begin{longtable}{@{}>{\raggedright\arraybackslash}p{1.9cm}>{\raggedright\arraybackslash}p{6.4cm}>{\raggedright\arraybackslash}p{2.0cm}>{\raggedright\arraybackslash}p{2.0cm}>{\raggedright\arraybackslash}p{2.0cm}@{}}
\caption{Searched tuning parameters for proposed policy and 12 benchmark policies.}
\label{tab: searched_optimal_parameters} \\

\toprule
Policy & Parameter $\pi_p$ & Parameter $\pi_s$ & Parameter $\pi_o$ & Cost (SE) \\ \midrule
\endfirsthead
\toprule
Policy & Parameter $\pi_p$ & Parameter $\pi_s$ & Parameter $\pi_o$ & Cost (SE) \\ \midrule
\endhead
\midrule
\multicolumn{5}{r}{\textit{Continued on next page}} \\
\endfoot
\bottomrule
\endlastfoot

\multicolumn{5}{c}{\textit{Panel A: Rybko-Stolyar network}} \\
\hline
Proposed-abs & -- & $20, 0.31$ & $1, 1.02, 1.0$ & $\mathbf{2.93\ (0.03)}$ \\
\midrule
A1 & $-0.0006, -0.0004; -0.0015, -0.001$ & -- & $1, 0.96, 2.0$ & $4.35\ (0.04)$ \\
A2 & $-0.0014, 0.0; -0.0014, 0.0$ & -- & $1, 0.96, 2.0$ & $4.03\ (0.04)$ \\
A3 & $0.8, 0.008, 0.04; 0.8, 0.02, 0.1$ & -- & $1, 0.95, 2.0$ & $4.42\ (0.04)$ \\
A4 & $0.7, 0.016, 0.064; 0.7, 0.016, 0.064$ & -- & $1, 0.96, 2.0$ & $4.15\ (0.04)$ \\
\midrule
B1 & $-0.0016, -0.0004; -0.0016, -0.0004$ & -- & $1, 1.07$ & $5.92\ (0.05)$ \\
B2 & $-0.002, 0.0; -0.002, 0.0$ & -- & $1, 1.07$ & $5.66\ (0.04)$ \\
B3 & $0.6, 0.016, 0.088; 0.6, 0.016, 0.088$ & -- & $1, 1.07$ & $6.01\ (0.05)$ \\
B4 & $0.6, 0.016, 0.096; 0.6, 0.016, 0.096$ & -- & $1, 1.07$ & $5.80\ (0.05)$ \\
\midrule
C1-abs & $-0.001, 0.0; -0.001, 0.0$ & $8, 0.31$ & $1, 1.00, 1.5$ & \uline{$3.24\ (0.04)$} \\
C2-abs & $0.7, 0.008, 0.056; 0.7, 0.008, 0.056$ & $8, 0.31$ & $1, 1.00, 2.0$ & $3.25\ (0.04)$ \\
\midrule
D1-abs & $-0.001, 0.0; -0.001, 0.0$ & $12, 0.34$ & $1, 1.11$ & $3.64\ (0.04)$ \\
D2-abs & $0.7, 0.008, 0.08; 0.7, 0.008, 0.08$ & $16, 0.31$ & $1, 1.11$ & $3.59\ (0.04)$ \\
\hline
\multicolumn{5}{c}{\textit{Panel B: $3$-server extended criss-cross network}} \\
\hline
Proposed-abs & -- & $2, 0.13$ & $0, 0.91, 1.5$ & $13.47\ (0.14)$ \\
Proposed-rel & -- & $80, 0.87$ & $0, 0.91, 1.0$ & $\mathbf{13.37\ (0.14)}$ \\
\midrule
A1 & $-0.0024, -0.00011; -0.0024, -0.0002$ & -- & $0, 0.93, 1.5$ & $16.60\ (0.16)$ \\
A2 & $-0.0048, 0.0; -0.0048, 0.0$ & -- & $0, 0.83, 1.0$ & $23.59\ (0.17)$ \\
A3 & $0.7, 0.012, 0.084; 0.7, 0.018, 0.126$ & -- & $0, 0.93, 0.5$ & $16.73\ (0.16)$ \\
A4 & $0.6, 0.024, 0.108; 0.6, 0.036, 0.162$ & -- & $0, 0.83, 0.5$ & $24.20\ (0.17)$ \\
\midrule
B1 & $-0.0028, -0.00011; -0.0028, -0.0002$ & -- & $0, 1.02$ & $17.77\ (0.15)$ \\
B2 & $-0.0036, 0.00034; -0.0036, 0.0006$ & -- & $0, 0.89$ & $20.44\ (0.16)$ \\
B3 & $0.6, 0.024, 0.084; 0.6, 0.036, 0.126$ & -- & $0, 1.02$ & $18.25\ (0.15)$ \\
B4 & $0.7, 0.012, 0.072; 0.7, 0.018, 0.108$ & -- & $0, 0.91$ & $20.71\ (0.17)$ \\
\midrule
C1-abs & $-0.002, -0.00011; -0.002, -0.0002$ & $3, 0.19$ & $0, 0.91, 1.0$ & $14.80\ (0.15)$ \\
C1-rel & $-0.0032, 0.00023; -0.0032, 0.0004$ & $80, 0.84$ & $0, 0.91, 0.5$ & \uline{$14.60\ (0.15)$} \\
C2-abs & $0.7, 0.012, 0.06; 0.7, 0.018, 0.09$ & $4, 0.19$ & $0, 0.91, 0.5$ & $14.84\ (0.16)$ \\
C2-rel & $0.7, 0.012, 0.072; 0.7, 0.018, 0.108$ & $80, 0.84$ & $0, 0.89, 0.5$ & $15.55\ (0.16)$ \\
\midrule
D1-abs & $-0.0028, 0.00023; -0.0028, 0.0004$ & $3, 0.13$ & $0, 0.94$ & $15.52\ (0.15)$ \\
D1-rel & $-0.0024, 0.0; -0.0024, 0.0$ & $80, 0.84$ & $0, 0.94$ & $15.76\ (0.15)$ \\
D2-abs & $0.7, 0.012, 0.072; 0.7, 0.018, 0.108$ & $3, 0.13$ & $0, 0.94$ & $15.63\ (0.15)$ \\
D2-rel & $0.7, 0.012, 0.072; 0.7, 0.018, 0.108$ & $80, 0.84$ & $0, 0.94$ & $15.88\ (0.15)$ \\
\hline
\multicolumn{5}{c}{\textit{Panel C: $10$-server extended criss-cross network}} \\
\hline
Proposed-abs & -- & $6, 0.22$ & $0, 0.93, 1.5$ & $\mathbf{17.37\ (0.45)}$ \\
Proposed-rel & -- & $160, 0.81$ & $0, 0.90, 0.5$ & $19.54\ (0.44)$ \\
\midrule
A1 & $-0.0008, -0.00006; -0.001, -0.0004$ & -- & $0, 0.95, 0.5$ & $28.43\ (0.62)$ \\
A2 & $-0.0024, -0.00006; -0.003, -0.0004$ & -- & $0, 0.89, 0.5$ & $43.97\ (0.49)$ \\
A3 & $0.8, 0.012, 0.048; 0.8, 0.06, 0.24$ & -- & $0, 0.95, 1.5$ & $28.49\ (0.58)$ \\
A4 & $0.8, 0.024, 0.072; 0.8, 0.12, 0.36$ & -- & $0, 0.86, 0.5$ & $46.28\ (0.62)$ \\
\midrule
B1 & $-0.0024, 0.00006; -0.003, 0.0004$ & -- & $0, 1.07$ & $36.20\ (0.54)$ \\
B2 & $-0.0024, -0.00003; -0.003, -0.0002$ & -- & $0, 0.94$ & $45.85\ (0.62)$ \\
B3 & $0.7, 0.012, 0.072; 0.7, 0.06, 0.36$ & -- & $0, 1.07$ & $37.35\ (0.59)$ \\
B4 & $0.8, 0.024, 0.096; 0.8, 0.12, 0.48$ & -- & $0, 0.94$ & $48.46\ (0.70)$ \\
\midrule
C1-abs & $-0.0012, 0.0; -0.0015, 0.0$ & $3, 0.22$ & $0, 0.93, 0.5$ & $19.06\ (0.44)$ \\
C1-rel & $-0.0016, 0.00003; -0.002, 0.0002$ & $120, 0.78$ & $0, 0.93, 0.5$ & $18.93\ (0.46)$ \\
C2-abs & $0.8, 0.012, 0.036; 0.8, 0.06, 0.18$ & $4, 0.22$ & $0, 0.93, 1.5$ & \uline{$18.92\ (0.44)$} \\
C2-rel & $0.8, 0.012, 0.036; 0.8, 0.06, 0.18$ & $80, 0.78$ & $0, 0.92, 1.5$ & $19.76\ (0.49)$ \\
\midrule
D1-abs & $-0.002, 0.00009; -0.0025, 0.0006$ & $2, 0.22$ & $0, 1.00$ & $26.82\ (0.60)$ \\
D1-rel & $-0.0024, 0.00009; -0.003, 0.0006$ & $80, 0.75$ & $0, 1.02$ & $23.11\ (0.56)$ \\
D2-abs & $0.7, 0.012, 0.048; 0.7, 0.06, 0.24$ & $3, 0.25$ & $0, 0.99$ & $27.79\ (0.65)$ \\
D2-rel & $0.8, 0.012, 0.072; 0.8, 0.06, 0.36$ & $80, 0.75$ & $0, 1.01$ & $25.80\ (0.58)$ \\
\end{longtable}}

First, we use aliases of the 12 benchmark policies, as given in Table \ref{tab: benchmark_parameter_search_alias}. Second, for policies using proposed scheduling policy, we report results separately for $\tilde{\tau} = \text{abs}$ and $\tilde{\tau} = \text{rel}$ in two rows, distinguished by the suffix appended to the policy name. To accommodate the two-row reporting for these policies, we additionally include a column ``Cost (SE)'' showing the optimized cost (with standard error) for each row; the optimized cost reported in Table \ref{tab: prelimit_simulation} corresponds to the smaller cost over $\tilde{\tau}\in\{\text{abs}, \text{rel}\}$. Third, for all policies, we report searched tuning parameters by three columns, following three-type parameter classification in Appendix \ref{sec: benchmark_combination}: 
\begin{align*}
    \pi_p:=(\text{pricing parameters in Panel B of Table \ref{tab: benchmark_three_parts}}),\quad
    \pi_s:=(\tilde{\lambda},\tilde{c}),\quad
    \pi_o: = (\rho,s,\epsilon,\tilde{\tau}).
\end{align*}
The detailed meaning of these three columns in Table \ref{tab: searched_optimal_parameters} is given below. 

\textit{Parameter $\pi_p$.} The proposed policy has no tuning pricing parameters, and the corresponding cells are marked ``--''. For policies using linear pricing, the column reports $\{(\phi_k, \phi_{-k}),k = 1,\dots,K\} \in \mathbb{R}^{K \times 2}$. For policies using high-low pricing, the column reports $\{(\tilde{T}_k/q_k^n(\rho), \zeta_k^L, \zeta_k^H),k = 1,\dots,K\} \in \mathbb{R}^{K \times 3}$. In both cases, the entries are listed route by route, with semicolons separating routes. In Panels B and C, only the entries of first and last routes are shown; the intermediate routes share identical parameter values with the first route.
    
\textit{Parameter $\pi_s$.} For policies using proposed scheduling policy, recall that $\tilde{\tau}\in\{\text{abs}, \text{rel}\}$ is indicated by the policy-name suffix as described above, the column reports $(\tilde{\lambda}, \tilde{c}/q^n_1(\rho))$ for $\tilde{\tau} = \text{abs}$ and $(\tilde{\lambda}, \tilde{c})$ for $\tilde{\tau} = \text{rel}$. The policies using least-relative-slack policy or least-relative-slack policy based on queue lengths have no scheduling parameters and are marked ``--''.
    
\textit{Parameter $\pi_o$.} For proposed policy and benchmark policies using single-pushing policy, we report $(s, \rho, \epsilon)$. For benchmark policies using threshold-based policy, we report $(s, \rho)$.

\subsubsection{Delay Performance.\label{appendix: test_problem_delay_performance}}

Table \ref{tab: delay_probabilities} reports the delay probabilities under the proposed policy and 12 benchmark policies under tuning parameters included in Table \ref{tab: searched_optimal_parameters}; the aliases of benchmark policies used hereafter refer to Table \ref{tab: benchmark_parameter_search_alias}. For policies using proposed scheduling policy, we report results separately for $\tilde{\tau} = \text{abs}$ and $\tilde{\tau} = \text{rel}$ in two rows, distinguished by the suffix appended to the policy name. The Rybko-Stolyar, 3-server, and 10-server problems have $2$, $4$, and $11$ routes respectively, and each cell lists the per-route delay probabilities ($\%$) in route order. All reported values satisfy the delay probability constraint $\varepsilon = 1\%$.

{\footnotesize
\begin{longtable}{@{}>{\raggedright\arraybackslash}p{1.9cm}>{\raggedright\arraybackslash}p{2.1cm}>{\raggedright\arraybackslash}p{2.9cm}>{\raggedright\arraybackslash}p{7.9cm}@{}}
\caption{Delay probabilities (\%) under proposed policy and benchmark policies. ECC denotes extended criss-cross network.}
\label{tab: delay_probabilities} \\
\toprule
Policy & Rybko-Stolyar & 3-server ECC & 10-server ECC \\ \midrule
\endfirsthead
\toprule
Policy & Rybko-Stolyar & 3-server ECC & 10-server ECC \\ \midrule
\endhead
\midrule
\multicolumn{4}{r}{\textit{Continued on next page}} \\
\endfoot
\bottomrule
\endlastfoot
Proposed-abs & $1.00, 1.00$ & $0.96, 0.80, 0.75, 0.99$ & $0.83, 0.89, 0.89, 0.99, 0.94, 0.97, 0.93, 0.90, 0.74, 0.49, 0.93$ \tabularnewline
Proposed-rel & -- & $0.99, 0.94, 0.89, 0.85$ & $0.82, 0.85, 0.91, 0.95, 0.90, 0.92, 0.91, 0.89, 0.83, 0.94, 0.35$ \tabularnewline
\midrule
A1 & $0.93, 0.98$ & $0.12, 0.25, 0.39, 0.99$ & $0.00, 0.01, 0.01, 0.03, 0.06, 0.10, 0.15, 0.21, 0.28, 0.36, 0.96$ \tabularnewline
A2 & $0.99, 1.00$ & $0.35, 0.71, 0.99, 0.25$ & $0.01, 0.04, 0.10, 0.18, 0.29, 0.42, 0.56, 0.76, 0.89, 1.00, 0.23$ \tabularnewline
A3 & $0.92, 0.99$ & $0.11, 0.26, 0.37, 0.99$ & $0.00, 0.01, 0.02, 0.04, 0.05, 0.10, 0.14, 0.21, 0.31, 0.36, 0.98$ \tabularnewline
A4 & $1.00, 1.00$ & $0.36, 0.71, 1.00, 0.24$ & $0.01, 0.04, 0.10, 0.18, 0.30, 0.43, 0.57, 0.71, 0.87, 0.98, 0.11$ \tabularnewline
\midrule
B1 & $0.99, 0.99$ & $0.12, 0.28, 0.49, 0.99$ & $0.00, 0.01, 0.01, 0.03, 0.05, 0.08, 0.13, 0.19, 0.28, 0.38, 0.99$ \tabularnewline
B2 & $1.00, 1.00$ & $0.14, 0.41, 0.84, 0.79$ & $0.00, 0.02, 0.07, 0.13, 0.22, 0.34, 0.50, 0.65, 0.84, 0.99, 0.39$ \tabularnewline
B3 & $0.99, 1.00$ & $0.12, 0.28, 0.49, 0.98$ & $0.00, 0.01, 0.01, 0.02, 0.05, 0.09, 0.13, 0.20, 0.26, 0.37, 0.94$ \tabularnewline
B4 & $0.98, 0.99$ & $0.15, 0.44, 0.89, 0.82$ & $0.00, 0.02, 0.06, 0.12, 0.21, 0.36, 0.48, 0.66, 0.83, 1.00, 0.41$ \tabularnewline
\midrule
C1-abs & $1.00, 0.96$ & $0.88, 0.76, 0.91, 0.95$ & $0.98, 0.87, 0.93, 0.94, 0.94, 0.95, 0.84, 0.80, 0.64, 0.45, 0.81$ \tabularnewline
C1-rel & -- & $0.99, 0.87, 0.98, 0.78$ & $0.86, 0.85, 0.89, 0.87, 0.87, 0.86, 0.86, 0.83, 0.79, 0.94, 0.73$ \tabularnewline
C2-abs & $0.99, 0.97$ & $0.92, 0.81, 0.95, 0.99$ & $0.89, 0.87, 0.85, 0.90, 0.88, 0.88, 0.86, 0.81, 0.64, 0.45, 0.85$ \tabularnewline
C2-rel & -- & $0.90, 0.79, 0.94, 0.70$ & $0.93, 0.90, 0.91, 0.89, 0.91, 0.86, 0.85, 0.75, 0.78, 0.97, 0.70$ \tabularnewline
\midrule
D1-abs & $0.96, 0.95$ & $0.98, 0.79, 0.57, 0.89$ & $0.93, 0.82, 0.90, 0.87, 0.92, 0.91, 0.88, 0.87, 0.68, 0.27, 0.90$ \tabularnewline
D1-rel & -- & $0.80, 0.69, 0.70, 0.90$ & $0.94, 0.86, 0.96, 0.98, 0.98, 0.96, 0.96, 0.95, 0.97, 0.83, 0.98$ \tabularnewline
D2-abs & $0.99, 1.00$ & $0.96, 0.80, 0.57, 0.89$ & $0.77, 0.73, 0.82, 0.82, 0.80, 0.81, 0.85, 0.76, 0.58, 0.21, 0.96$ \tabularnewline
D2-rel & -- & $0.83, 0.70, 0.72, 0.90$ & $0.96, 0.90, 0.91, 0.96, 0.96, 0.95, 0.93, 0.95, 0.92, 0.90, 0.90$ \tabularnewline
\end{longtable}}

Recall that Figures \ref{fig: delay_time_RS}--\ref{fig: delay_time_CC3} present delay time histograms across routes for the proposed policy and representative benchmark policies in Rybko-Stolyar and $3$-server problems. We further provide the results for $10$-server problem in Figure \ref{fig: delay_time_CC10}, where the observation is similar to those in $3$-server case (see Section \ref{sec: simulation_results}). 

\begin{figure}[htbp]
\FIGURE
{\includegraphics[scale=0.7]{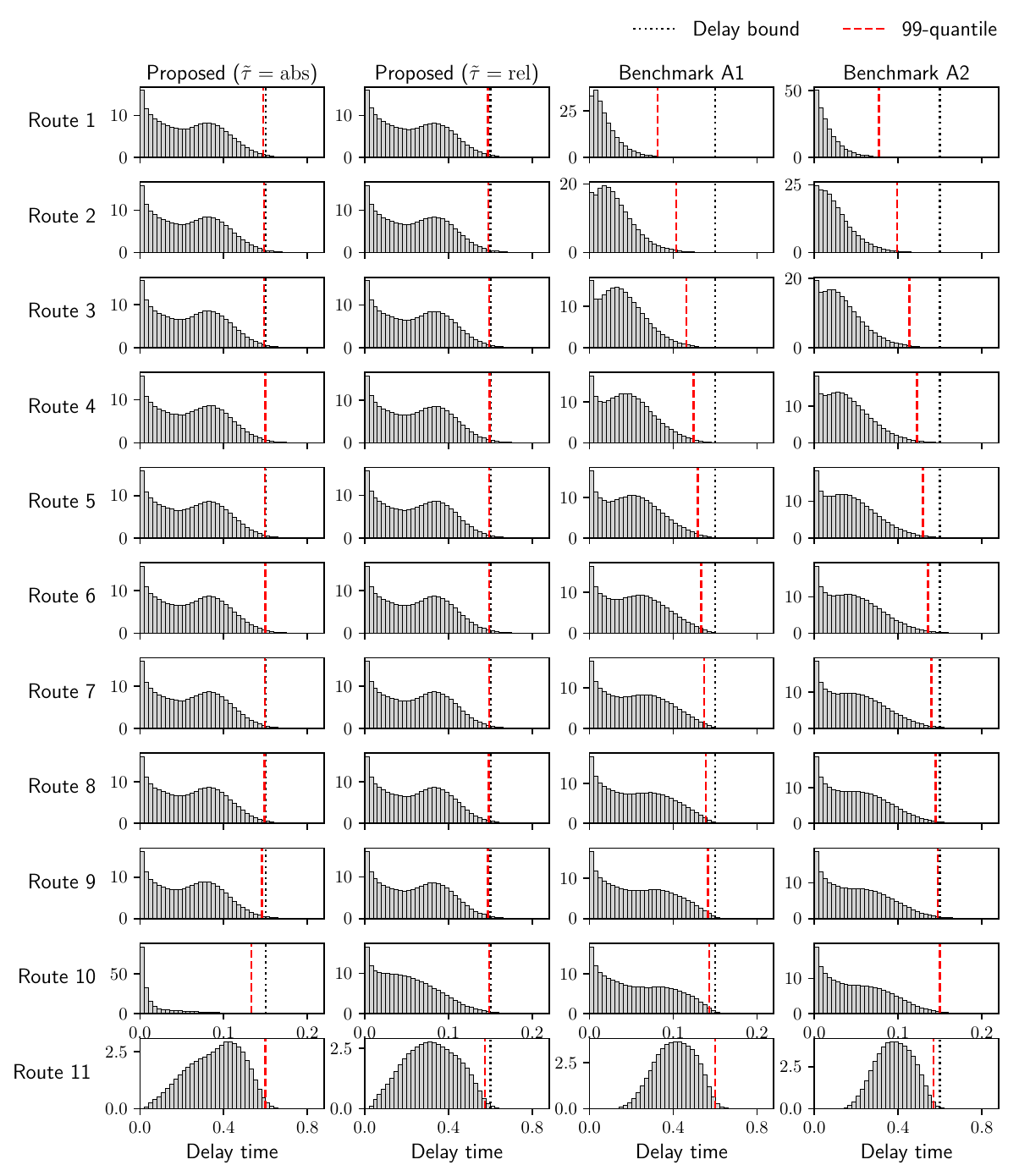}}
{Delay time histogram under the proposed policy and representative benchmark policies in the 10-server extended criss-cross network. Benchmark A1 (resp. A2): the benchmark policy using single-pushing, linear pricing, and least-relative-slack policy (resp. least-relative-slack policy based on queue lengths).\label{fig: delay_time_CC10}}
{}
\end{figure}

\end{document}